\documentclass{amsart}
\usepackage{amsmath ,amssymb ,amsthm, mathrsfs, stmaryrd, dsfont, bm, shadethm, comment, here}
\usepackage{ulem}
\usepackage[pdftex]{graphicx} 
\usepackage[pdftex ,pdfborder={0 0 0}]{hyperref} 
\usepackage{tikz, tikz-cd}
\usepackage{xcolor}
\usepackage{here}
\usepackage{url}
\usepackage{braket}
\DeclareMathAlphabet{\mathbfcal}{OMS}{cmsy}{b}{n}
\DeclareMathAlphabet{\mathpzc}{OT1}{pzc}{m}{it}

\definecolor{note}{RGB}{237, 144, 186}
\definecolor{ele}{RGB}{155, 206, 146}

\makeatletter

\@addtoreset{equation}{section}
\makeatother

\theoremstyle{definition}
\newtheorem{defin}{Definition}[section]

\newtheorem{thm}[defin]{Theorem}
\newtheorem{conj}[defin]{Conjecture}
\newtheorem{lem}[defin]{Lemma}
\newtheorem{prop}[defin]{Proposition}
\newtheorem{cor}[defin]{Corollary}

\theoremstyle{remark}
\newtheorem{rem}[defin]{Remark}

\newtheorem{quest}[defin]{Question}

\newcommand{\M}{\mathcal{M}_{\mathrm{NA}}}

\newcommand{\E}{\mathcal{E}_{\mathrm{NA}}}

\newcommand{\NAmu}{\bm{\check{\mu}}_{\mathrm{NA}}}

\title[Toric Non-archimedean $\mu$-entropy]{Toric non-archimedean $\mu$-entropy and thermodynamical structure}
\author{Eiji Inoue}

\address{RIKEN, iTHEMS, 2-1 Hirosawa, Wako, Saitama, 351-0198, Japan}

\email{eiji.inoue@riken.jp}

\begin{document}

\maketitle

\begin{abstract}
We study non-archimedean $\mu$-entropy for toric variety as a further exploration of $\mu$K-stability. 
We show the existence of optimizer of toric non-archimedean $\mu^\lambda$-entropy for $\lambda \in \mathbb{R}$ and the uniqueness for $\lambda \le 0$. 
For the proof of existence, we establish a Rellich type compactness result for convex functions on simple polytope. 
We also reveal a thermodynamical structure on toric non-archimedean $\mu$-entropy. 
This observation allows us to interpret the enigmatic parameter $T = - \frac{\lambda}{2\pi}$ as temperature and non-archimedean $\mu$-entropy as entropy of an infinite dimensional composite system. 
\end{abstract}

\tableofcontents

\section{Introduction}

\subsection{Main Results on existence}

Canonical metric and K-stability are primary interests in K\"ahler geometry. 
There is a series of studies \cite{Ino1, Ino2, Ino3, Ino4, Ino5} exploring these topics from a unique perspective which motivates a minimization problem we study in this article. 
Compared to the general case \cite{Ino5}, the problem in the toric case we discuss here can be described in a quite simple way by convex functions on convex polytope, for which we only need a few things to prepare. 
It would be simpler to begin with our convex setup and main existence result of this article, and explain its esoteric background and motivation for the problem afterwards. 

\subsubsection{Setup}

We prepare some terminologies convenient for our arguments. 
Let $V$ be an $n$-dimensional affine space over $\mathbb{R}$. 
A \textit{half space} of $V$ is a subset of the form $\ell^{-1} ([0, \infty))$ for some non-zero affine function $\ell: V \to \mathbb{R}$. 
The \textit{face} of a half space $\ell^{-1} ([0, \infty))$ is the hyperplane $\ell^{-1} (0)$. 

We call a subset $P \subset V$ a \textit{polytope} if it is a compact subset with non-empty interior and can be expressed as the intersection of finitely many half spaces. 
A subset $F \subset P$ is called a \textit{face} of $P$ if it is the intersection of $P$ and the face of a half space which contains $P$. 
A face is again a polytope of a lower dimensional affine subspace $W \subset V$. 
The \textit{relative interior} of a face is the interior in such $W$. 
A \textit{facet} of $P$ is a face of $P$ which gives a polytope of a codimension one affine subspace.  
A \textit{vertex} of $P$ is a face of $P$ with one element. 
A \textit{simple polytope} is a polytope for which every vertex can be written as the intersection of precisely $n$-facets. 
In this article, we restrict our main interest to simple polytopes. 
In toric geometry, such (rational) polytopes represent toric orbifolds. 

It is well-known that a complete parallel translation invariant measure on $V$ (with the usual $\sigma$-algebra) which puts positive finite mass on relatively compact open sets is a constant multiple of the Lebesgue measure under an affine identification $V \cong \mathbb{R}^n$. 
A \textit{flat measure} on a polytope (resp. a face of a polytope) is the restriction of such translation invariant measure on $V$ (resp. such measure on the affine subspace $W$ which contains the face as a polytope). 
We note the topological boundary $\partial P = P \setminus P^\circ$ is the union of facets. 
A flat measure on $\partial P$ is a measure whose restriction to each facet is flat. 

For a polytope $P$ and $p \ge 1$, we consider the following spaces: 
\begin{align}
\E^{\exp, p} (P)
&:= \{ \text{ lsc convex } q: P \to (-\infty, \infty] ~|~ \int_P e^{pq} d\mu < \infty \},
\\
\M^{\exp, p} (P)
&:= \{ \text{ lsc lc } u: P \to (0, \infty] ~|~ \int_P u d\mu = \int_P d\mu, \int_P u^p d\mu < \infty \},
\end{align}
which are independent of the choice of the flat measure $d\mu$ on $P$. 
Here \textit{lsc lc} is short for \textbf{l}ower \textbf{s}emi-\textbf{c}ontinuous \textbf{l}og \textbf{c}onvex. 
Namely, we have $\liminf u (x_i) \ge u (x_\infty)$ for $x_i \to x_\infty \in P$ and $u ((1-t)x_0 + tx_1) \le u^{1-t} (x_0) u^t (x_1)$. 
For two lsc convex functions $q, q'$, the condition $q = q'$ almost everywhere with respect to $d\mu$ is equivalent to the condition $q=q'$ everywhere. 
In particular, the boundary value $q|_{\partial P}$ can be recovered from interior values $q|_{P^\circ}$. 

We endow $\M^{\exp, p} (P)$ with the $L^p$-topology and call its element \textit{(nonequilibrium) state}. 
For $q \in \E^{\exp, p} (P)$, we can assign a state
\begin{align}
u (q) := \frac{\int_P d\mu}{\int_P e^q d\mu} e^q \in \M^{\exp, p} (P). 
\end{align}
For $q \in \E^{\exp, 1} (P)$, we have $q \in \E^{\exp, p} (P)$ if and only if $u (q) \in \M^{\exp, p} (P)$. 
Later we explain a K\"ahler geometric background of the space $\E^{\exp, p} (P)$. 

Now let us consider a triple $\mathcal{P} = (P, d\mu, d\sigma)$ of a simple polytope $P$, a flat measure $d\mu$ on $P$ and a flat measure $d\sigma$ on $\partial P$, which we call a \textit{simple non-archimedean Hamiltonian system} or just a \textit{system}. 
When we refer to $\mathcal{P}$ with general non-simple polytope $P$, we call it a \textit{general system}. 

For $u \in \M^{\exp, 1} (P)$, we introduce
\begin{align}
S_{\mathcal{P}} (u)
&:= - \frac{1}{\int_P d\mu} \int_P u \log u d\mu \in [-\infty, 0],
\\
U_{\mathcal{P}} (u) 
&:= \frac{1}{\int_P d\mu} \int_{\partial P} u d\sigma \in (0, \infty], 
\end{align}
which we call the \textit{(nonequilibrium) entropy} and \textit{(non-archimedean nonequilibrium) internal $\mu$-energy}, respectively. 
For $T \in \mathbb{R}$, which we call \textit{temperature}, we further introduce 
\begin{align} 
F_{\mathcal{P}} (T, u) := 
\begin{cases}
U_{\mathcal{P}} (u) - T S_{\mathcal{P}} (u)
& S_{\mathcal{P}} (u) > -\infty
\\
\infty
& S_{\mathcal{P}} (u) = -\infty
\end{cases}. 
\end{align}
We call this functional \textit{(non-archimedean nonequilibrium) free $\mu$-energy}. 
We will see $S_{\mathcal{P}} (u) = -\infty$ implies $U_{\mathcal{P}} (u) = \infty$ and this definition makes the functional $F_{\mathcal{P}} (T, u (q))$ continuous along increasing sequence $q_i \nearrow q$. 
As we explain later, there is a K\"ahler geometric background for these functionals. 

In section \ref{Thermodynamical structure}, we newly unveil there is a stochastic thermodynamical interpretation of these functionals, which gives us a strong motivation to our terminologies of thermodynamical flavor. 
We note some references (cf. \cite{Ber}) in the same field use similar terminologies like entropy and free energy, but these are not directly related to ours. 
We are mostly interested in K-unstable case, while the reference focuses on K-stable case. 

\subsubsection{Main results}

Now we state our main result. 
Its output to K\"ahler geometry is explained later. 

\begin{thm}
\label{main existence}
Let $\mathcal{P} = (P, d\mu, d\sigma)$ be an $n$-dimensional system. 
Then for every $T \in \mathbb{R}$, there exists $u \in \M^{\exp, \frac{n}{n-1}} (P)$ which minimizes $F_{\mathcal{P}} (T, \bullet): \M^{\exp, 1} (P) \to (-\infty, \infty]$. 

Moreover, 
\begin{enumerate}
    \item (Uniqueness) if $T > 0$, minimizer of $F_{\mathcal{P}} (T, \bullet)$ is unique. 

    \item (Conditional uniqueness) if $T = 0$, there exists a unique minimizer of $F_{\mathcal{P}} (0, \bullet) = U_{\mathcal{P}} (\bullet)$ which maximizes $S_{\mathcal{P}}$ among all minimizers. 

    \item (Regularity) if $T=0$ and $n=2$, any minimizer of $F_{\mathcal{P}} (0, \bullet) = U_{\mathcal{P}} (\bullet)$ is bounded and continuous. 
\end{enumerate}
We call the unique minimizer in the above (1) and (2) the \textit{$\mu$-canonical distribution} of temperature $T$. 
In K\"ahler geometric context, $q = \log u \in \E^{\exp,1} (P)$ is also called the \textit{optimizer} or \textit{optimal destabilizer} of $\NAmu^\lambda$. 
\end{thm}

This main theorem summarizes Theorem \ref{Main theorem on existence}, Theorem \ref{Main theorem on uniqueness} and Theorem \ref{Main theorem on regularity}. 
To show the existence, we establish the following compactness result, which is proved in section \ref{Rellich and Poincare type estimates}. 
The proof works not only for lsc lc functions but for general non-negative convex functions. 

\begin{thm}
Let $\{ u_i: P \to [0, \infty] \}_{i \in \mathbb{N}}$ be a sequence of non-negative convex functions with a uniform bound 
\[ \int_{\partial P} u_i d\sigma \le C. \]
Then all $u_i$ are in $L^{\frac{n}{n-1}} (P)$ and after taking a subsequence, there exists a unique non-negative lsc convex function $u: P \to [0, \infty]$ in $L^{\frac{n}{n-1}} (P)$ such that $u_i$ converges to $u$ in $L^p$-topology for every $p \in [0, \frac{n}{n-1})$. 
\end{thm}

\begin{rem}
The above theorem is reminiscent of Rellich's compactness theorem. 
It is well-known by Sobolev embedding and trace theorem that we have
\[ L^1 (\partial P) \xleftarrow{|_{\partial P}} W^{1,1} (P) \hookrightarrow L^{\frac{n}{n-1}} (P), \]
while what we prove is 
\[ \mathrm{Conv} (P) \cap L^1 (\partial P) \subset L^{\frac{n}{n-1}} (P). \]
It is natural to expect $\mathrm{Conv} (P) \cap L^1 (\partial P) = \mathrm{Conv} (P) \cap W^{1,1} (P)$, but we do not pursue it in this article. 
\end{rem}

As an interest in K\"ahler geometry, we also prove the following. 
It summarizes Theorem \ref{mu-cscK implies optimal destabilization} and Theorem \ref{muK-semistability is equivalent to mu-entropy maximization}. 
We explain terminologies in the next section. 

\begin{thm}
Let $(X, L) \circlearrowleft T$ be a toric variety and $\mathcal{P}$ be the associated system we explain later. 
Let $\ket{\xi}: P \to \mathbb{R}: \mu \mapsto \langle \mu, \xi \rangle$ denote the linear map associated to $\xi \in \mathfrak{t}$. 
For $T \ge 0$, we have the following. 
\begin{itemize}
    \item The toric variety $(X, L)$ is toric $\mu^{-2\pi T}_\xi$K-semistable if and only if $u = u (\ket{\xi})$ minimizes $F_{\mathcal{P}} (T, \bullet)$. 

    \item If a toric manifold $(X, L)$ admits a $\mu^{-2\pi T}_\xi$-cscK metric with $\xi \in \mathrm{Lie} (T_{\mathrm{cpt}})$, then $u = u (\ket{\xi})$ is the $\mu$-canonical distribution of temperature $T$. 
\end{itemize}
\end{thm}

We further prove some results in section \ref{Thermodynamical structure} motivated by thermodynamics, using our existence result. 
To appreciate the results, we must recall the origin of $\mu$-cscK metric and the enigmatic way our parameter $T$ is introduced in the theory. 
We explain these in section \ref{Thermodynamical structure}. 
Though we present our main result as the existence and uniqueness of minimizer for the free $\mu$-energy $F_{\mathcal{P}} (T, \bullet)$, the observation in section \ref{Thermodynamical structure} would be much deeper. 
It should be explored further in future study. 

\subsection{Why does it matter in K\"ahler geometry?}

\subsubsection{Toric variety and polytope}

Let $T \cong \mathbb{G}_m^{\times n}$ be an algebraic torus and $T_{\mathrm{cpt}} \subset T$ be the maximal compact torus. 
For the character lattice $M = \mathrm{Hom} (T, \mathbb{G}_m) = \mathrm{Hom} (T_{\mathrm{cpt}}, U (1))$, we put $\mathfrak{t}^\vee := M \otimes_{\mathbb{Z}} \mathbb{R}$. 
We denote the dual vector space by $\mathfrak{t}$ and identify it with the Lie algebra of $T_{\mathrm{cpt}}$. 

Ler $(X, L) \circlearrowleft T$ be a polarized toric variety of dimension $n$. 
As usual in toric geometry, we assign a general system $\mathcal{P} = (P, d\mu, d\sigma)$ as follows. 
We put
\begin{align}
    P = \mathrm{Conv} (\{ \mu \in M ~|~ H^0 (X,L)_\mu \neq 0 \}) \subset \mathfrak{t}^\vee.
\end{align}
Here $H^0 (X,L)_\mu$ denotes the eigenspace $\{ s \in H^0 (X, L) ~|~ s. t = \mu (t) s ~\forall t \in T \}$. 
The lattice $M$ defines a unique flat measure on $\mathfrak{t}^\vee$ by setting $\mu (B (m_1, \ldots, m_n)) = 1$ for the box 
\[ B (m_1, \ldots, m_n) = \{ \sum_{i=1}^m t_i m_i ~|~ 0 \le t_i \le 1 \} \]
generated by a $\mathbb{Z}$-basis $m_1, \ldots, m_n \in M$. 
Since $|\det A| =1 $ for $A \in GL (M)$, the flat measure is independent of the choice of $m_1, \ldots, m_n$. 
We denote its restriction to $P$ by $d\mu$. 
On the other hand, since the affine subspace $W \subset \mathfrak{t}^\vee$ spanned by a facet $Q$ of $P$ is a rational subspace, the intersection $W \cap M$ gives a lattice spanning $W$. 
Then we similarly define a unique flat measure on each facet $Q$ compatible with the lattice $W \cap M$ and a flat measure $d\sigma$ on $\partial P$. 

The flat measure $d\sigma$ on $\partial P$ represents the anti-canonical divisor $-K_X$. 
In general, flat signed measures on $\partial P$ are in one to one correspondence with $T$-invariant $\mathbb{R}$-Weil divisors: each facet $Q \subset \partial P$ represents a $T$-invariant prime divisor $E_Q$ and a flat signed measure $d\sigma'$ on $\partial P$ represents the divisor $\sum_Q \frac{d\sigma' (Q)}{d\sigma (Q)} E_Q$. 
For the flat signed measure $d\sigma'$ representing a $T$-invariant $\mathbb{R}$-Weil divisor $\Delta = \sum_Q a_Q E_Q$ on $X$, we put 
\begin{align}
d\sigma_\Delta := d\sigma - d\sigma' = \sum_Q (1-a_Q) d\sigma, 
\end{align} 
which represents the $\mathbb{R}$-divisor $-(K_X + \Delta)$. 
From thermodynamical perspective, the measure $d\sigma$ plays the role of Hamiltonian of the system and its variation can be interpreted as \textit{thermodynamical work}. 

A triple $(X, \Delta, L)$ of a toric variety $X$, a $T$-invariant $\mathbb{R}$-divisor $\Delta$ with coefficient $a_Q \in [0,1)$ and a $T$-equivariant polarization $L$ is called a \textit{polarized log toric variety}. 
By the above construction, we can assign a general system $\mathcal{P} = (P, d\mu, d\sigma_\Delta)$ for a polarized log toric variety. 

The associated polytope $P$ is known to be simple if and only if $X$ has at most orbifold singularity. 
See \cite[Theorem 3.1.19]{CLS}. 

\subsubsection{$\mu$-cscK metric}

Let $X$ be a compact K\"ahler manifold and $L \in H^2 (X, \mathbb{R})$ be a K\"ahler class. 
We consider a pair $(\omega, f)$ of a K\"ahler metric (form) $\omega$ in $L$ and a smooth real valued function $f$ on $X$. 
For $\lambda \in \mathbb{R}$, a pair $(\omega, f)$ is called \textit{$\mu^\lambda$-cscK metric} if it satisfies the following two conditions: 
\begin{enumerate}
    \item The complex vector field $\partial^\sharp_\omega f = g^{i \bar{j}} f_{,\bar{j}} \partial_i$ is holomorphic. 

    \item The $\mu^\lambda$-scalar curvature $s^\lambda (\omega, f) := s (\omega) + \Delta_\omega f - |\partial_\omega^\sharp f|^2 - \lambda f$ is constant. 
\end{enumerate}
When the imaginary part $\xi = - 2 \mathrm{Im} \partial^\sharp_\omega f$ is specified, we call the metric $\omega$ \textit{$\mu^\lambda_\xi$-cscK metric}. 
Since $i_\xi \omega = -2\mathrm{Im} i_{\partial^\sharp_\omega f} \omega = -2 \mathrm{Im} g (J \partial^\sharp_\omega f, \cdot) = -df$, $f$ is the Hamiltonian potential of $\xi$. 
The imaginary part $\xi$ preserves the metric $g$, so it is a Killing vector field. 
This implies we can always take a closed torus $T_{\mathrm{cpt}}$ acting on $(X, L)$ so that $\xi \in \mathfrak{t} = \mathrm{Lie} (T_{\mathrm{cpt}})$. 
We note there is a subtle but an important difference between $\mu^\lambda$-cscK metric and $\mu^\lambda_\xi$-cscK metric: when we refer to $\mu^\lambda$-cscK metric, the vector $\xi$ is not fixed. 

\begin{rem}
The above definition of $\mu^\lambda_\xi$-cscK metric is equivalent to what we call ${^1 \check{\mu}}^\lambda_\xi$-cscK metric in \cite{Ino3}. 
\end{rem}

The notion unifies two frameworks of canonical metrics: constant scalar curvature K\"ahler (cscK) metric and K\"ahler--Ricci soliton. 
Indeed, when $X$ is a Fano manifold (i.e. the anti-canonical class $-K_X = c_1 (X)$ is positive), a K\"ahler metric $\omega$ in the anti-canonical K\"ahler class $T^{-1} K_X$ ($T < 0$) is $\mu^{-2\pi T}$-cscK metric if and only if it is K\"ahler--Ricci soliton: $\mathrm{Ric} (\omega) - \mathscr{L}_{\partial^\sharp f} \omega = - 2\pi T \omega$. 
K\"ahler--Ricci soliton is defined only for Fano manifolds due to the a priori constraint 
\[ 2\pi \lambda [\omega] = [\mathrm{Ric} (\omega) - \mathscr{L}_{\partial^\sharp f} \omega] = 2\pi c_1 (X), \]
while $\mu^\lambda$-cscK metric makes sense for general polarized manifold. 

A fundamental question on $\mu^\lambda$-cscK metric is the existence and uniqueness modulo translation by automorphisms. 
When we fix $\xi$, the theory of $\mu^\lambda_\xi$-cscK metric is enclosed in the theory of weighted cscK metric. 
It is then proved by \cite{Lah2} that $\mu^\lambda_\xi$-cscK metric is unique modulo translation by aumorphisms preserving $\xi$. 
As for the uniqueness of $\mu^\lambda$-cscK metric, we must show the uniqueness of $\xi$ (modulo translation) which admits a $\mu^\lambda_\xi$-cscK metric. 
We confirm this uniqueness for toric manifolds with $\lambda \le 0$. 

\begin{thm}
\label{Uniqueness of mu-cscK metric}
Assume $\lambda \le 0$. 
On a toric manifold,  $\mu^\lambda$-cscK metrics are unique modulo the action of automorphism group. 
\end{thm}

On the other hand, fixing $\xi$, the existence of $\mu^\lambda_\xi$-cscK metric implies an algebro-geometric condition of $(X, L)$ called $\mu^\lambda_\xi$K-stability. 

\subsubsection{$\mu$K-stability}

The notion can be defined for general polarized scheme, but here for simplicity we only explain a toric version of the notion. 
It can be simply described in terms of convex functions. 

\begin{defin}[Toric $\mu$K-stability]
Let $(X, \Delta, L) \circlearrowleft T$ be a polarized log toric variety and $\mathcal{P} = (P, d\mu, d\sigma_\Delta)$ be the associated general system. 

Let $q$ be a rational piecewise affine convex function on $P$. 
Namely, $q$ is of the form $q = \max_{i=1, \ldots, m} \ell_i$ for finitely many rational affine function $\ell_i$ on $\mathfrak{t}^\vee = M \otimes_{\mathbb{Z}} \mathbb{R}$. 
Then for $\lambda \in \mathbb{R}$ and $\xi \in \mathfrak{t} = \mathrm{Lie} (T_{\mathrm{cpt}})$, we define the \textit{$\mu^\lambda_\xi$-Futaki invariant} $\mathrm{Fut}^\lambda_\xi (q)$ by 
\begin{align}
\mathrm{Fut}^\lambda_\xi (q) 
&:= \frac{2\pi \int_{\partial P} q e^{\ket{\xi}} d\sigma_{\Delta} - \lambda \int_P q \ket{\xi} e^{\ket{\xi}} d\mu}{\int_P e^{\ket{\xi}} d\mu}  - \bar{s}^\lambda_\xi \frac{\int_P q e^{\ket{\xi}} d\mu}{\int_P e^{\ket{\xi}} d\mu},
\end{align}
using the linear map $\ket{\xi}:P \to \mathbb{R}: \mu \mapsto \langle \mu, \xi \rangle$. 
Here we put 
\[ \bar{s}^\lambda_\xi := \frac{2\pi \int_{\partial P} e^{\ket{\xi}} d\sigma_\Delta - \lambda \int_P \ket{\xi} e^{\ket{\xi}} d\mu}{\int_P e^{\ket{\xi}} d\mu}. \]

We call $(X, \Delta, L)$ \textit{toric $\mu^\lambda_\xi$K-semistable} if $\mathrm{Fut}^\lambda_\xi (q) \ge 0$ for every rational piecewise affine function $q$ and \textit{toric $\mu^\lambda_\xi$K-polystable} if moreover $\mathrm{Fut}^\lambda_\xi (q) > 0$ for non-affine $q$. 
\end{defin}

\begin{rem}
The above definition of $\mu^\lambda_\xi$-Futaki invariant is equivalent to what we call ${^1 \check{\mu}}^\lambda_\xi$-Futaki invariant in \cite{Ino3}. 
Using the Cartan model of equivariant cohomology as in \cite{Ino2, Ino3}, we can check 
\[ \bar{s}^\lambda_\xi = \frac{\int_X s^\lambda_\xi (\omega) e^{\mu_\xi} \mathrm{vol}_\omega}{\int_X e^{\mu_\xi} \mathrm{vol}_\omega} \]
for $\mathrm{vol}_\omega = \omega^n/n!$. 
For $q = \ket{\eta}$, we have 
\[ \mathrm{Fut}^\lambda_\xi (\ket{\eta}) = \frac{\int_X (s^\lambda_\xi (\omega) - \bar{s}^\lambda_\xi) \mu_\eta e^{\mu_\xi} \mathrm{vol}_\omega}{\int_X e^{\mu_\xi} \mathrm{vol}_\omega}. \]
\end{rem}

When $\xi = 0$, the invariant is independent of $\lambda$: 
\begin{align} 
\mathrm{DF} (q) := \mathrm{Fut}^\lambda_0 (q) = \frac{2\pi}{\int_P d\mu} (\int_{\partial P} q d\sigma_\Delta - \frac{\int_{\partial P} d\sigma_\Delta}{\int_P d\mu} \int_P q d\mu)
\end{align}
It is called the \textit{Donaldson--Futaki invariant}. 

We note the $\mu^\lambda_\xi$-Futaki invariant is well-defined for general integrable convex function $q$ on general system $\mathcal{P}$. 
For a set $\mathcal{Q}$ of convex functions, we call $\mathcal{P}$ \textit{toric $\mu^\lambda_\xi$K-semistable (resp. $\mu^\lambda_\xi$K-polystable) with respect to $\mathcal{Q}$} if $\mathrm{Fut}^\lambda_\xi (q) \ge 0$ for every $q \in \mathcal{Q}$ (resp. if further $\mathrm{Fut}^\lambda_\xi (q) > 0$ for non-affine $q$). 
We call $\mathcal{P}$ \textit{K-unstable} if it is not $\mu_0$K-semistable. 
We note even if $X$ is K-unstable, it may be $\mu^\lambda_\xi$K-semistable with respect to some $\lambda$ and non-trivial $\xi \neq 0$. 

\begin{rem}
\label{piecewise affine}
We define $\mu^\lambda_\xi$K-stability by rational piecewise affine convex functions because these convex functions correspond to toric test configurations. 
A \textit{(compactified ample) toric test configuration} is a $T \times \mathbb{G}_m$-equivariant flat family of polarized schemes $(\bar{\mathcal{X}}, \bar{\mathcal{L}})$ over $\mathbb{P}^1$ which is endowed with a $T \times \mathbb{G}_m$-equivariant isomorphism from $X \times \mathbb{A}^1$ to $\bar{\mathcal{X}} \setminus \mathcal{X}_0$ over $\mathbb{A}^1 = \mathbb{P}^1 \setminus \{ (0:1) \}: w = (1:w)$. 
Here $T \times \mathbb{G}_m$ acts on $\mathbb{P}^1$ by $(z:w). (t, u) = (u z: w)$ and $\bar{\mathcal{L}}$ is an ample $\mathbb{Q}$-line bundle over $\bar{\mathcal{X}}$. 
We may assume $\bar{\mathcal{X}}$ is normal for our interest. 
Since $(\bar{\mathcal{X}}, \bar{\mathcal{L}}) \circlearrowleft T \times \mathbb{G}_m$ is an $n+1$-dimensional polarized toric variety, we can consider the associated polytope $Q \subset \mathfrak{t}^\vee \times \mathbb{R}$. 
We can show the polytope $Q$ can be written as $Q = \{ (x, \tau) ~|~ x \in P,~ 0 \le \tau \le -q(x) \}$ using a convex function $q$ on $P$. 
Since $Q$ is a rational polytope, $q$ is rational piecewise affine convex function. 

For a polarized variety $(X, L)$, we denote by $\mathcal{H}_{\mathrm{NA}} (X, L)$ the set of all (not necessarily $T$-equivariant) test configurations of $(X, L)$. 
Similarly for a rational polytope $P$, we denote by $\mathcal{H}_{\mathrm{NA}} (P)$ the set of all rational piecewise affine convex functions on $P$. 
Convex functions in $\mathcal{H}_{\mathrm{NA}} (P)$ represent $T$-equivariant test configurations, which form a proper subset $\mathcal{H}_{\mathrm{NA}}^T (X, L) \subset \mathcal{H}_{\mathrm{NA}} (X, L)$. 
We note every lower semi-continuous convex function $q$ has an increasing net $q_i \in \mathcal{H}_{\mathrm{NA}} (P)$ which pointwisely converges to $q$, thanks to Fenchel--Moreau theorem. 
\end{rem}

By the result of \cite{AJL} (see also \cite{Lah1, Ino3}), it is known that a polarized manifold $(X, L)$ is $\mu^\lambda_\xi$K-polystable if it admits a $\mu^\lambda_\xi$-cscK metric. 
When $\xi = 0$, the case of cscK metric, it is also known that a uniform version of toric K-polystability implies the existence of cscK metric for toric manifold by \cite{His1, CC3, Li1}. 
As far as the author knows, this implication for general $\xi \neq 0$ is still unkown. 
It is natural to expect the same method for $\xi = 0$ would work. 
For general polarized manifold, it is still a conjecture even for cscK metric, which is known as YTD conjecture, but recently there are great progress (cf. \cite{CC2, Li1, Li2, BJ3}). 
We do not pursue this direction in this article as we are rather interested in K-instability. 

\subsubsection{Perelman entropy}

Now we return to $\mu$-cscK metric. 
Originally, $\mu$-cscK metric is introduced in \cite{Ino2} based on a moment map picture on K\"ahler--Ricci soliton observed in \cite{Ino1}. 
The moment map picture heuristically explains why the existence of $\mu^\lambda_\xi$-cscK metric is related to $\mu^\lambda_\xi$K-stability and moduli problem, so this perspective is important in stability aspect. 
We review this later in order to explain how the parameter $\lambda$ appears. 

On the other hand, it turns out later in \cite{Ino4} that $\mu^\lambda$-cscK metric is also characterized by Perelman's entropy. 
This perspective illuminates ``instability aspect'' of the theory of $\mu$-cscK metric and $\mu$K-stability, which is the main interest of this article. 

Now we introduce Perelman's functionals. 
For a K\"ahler metric $\omega$ and a real valued function $f$ normalized as $\int_X e^f \mathrm{vol}_\omega = \int_X \mathrm{vol}_\omega$ for $\mathrm{vol}_\omega := \omega^n/n!$, we put 
\begin{align}
    \mathcal{W}^\lambda (\omega, f) 
    &:= - \frac{1}{\int_X \mathrm{vol}_\omega} \int_X (s (\omega) + |\partial_\omega^\sharp f|^2 - \lambda f) e^f \mathrm{vol}_\omega - \lambda \log \int_X e^{-n} \mathrm{vol}_\omega,
    \\
    \bm{\mu}^\lambda_{\mathrm{Per}} (\omega) 
    &:= \sup_f \mathcal{W}^\lambda (\omega, f), 
\end{align}
which we call Perelman's \textit{$\mathcal{W}$-entropy} and \textit{$\mu$-entropy}, respectively. 
We note the above convention differs from Perelman's original one by sign and multiplication of constants, which is off course not essential. 
Perelman considered general Riemannian metric for $\mathcal{W}^\lambda$, but we restrict our interest to K\"ahler metric in the context of $\mu^\lambda$-cscK metric, which is in turn essential for our variational result. 

It is proved in \cite{Ino4} that we can characterize $\mu^\lambda$-cscK metrics $(\omega, f)$ as the critical points of $\mathcal{W}^\lambda$. 
We note here we must restrict metrics to the space $\mathcal{H} (X, L)$ of K\"ahler metrics in $L$. 
Moreover, for $\lambda \le 0$, we can also characterize $\omega$ of $\mu^\lambda$-cscK metrics $(\omega, f)$ as minimizers of $\bm{\mu}^\lambda_{\mathrm{Per}}$. 
This characterization is nice as our differential geometric interest is in the metric $\omega$ and not so much in the function $f$ or the vector field $\xi = -2 \mathrm{Im} \partial_\omega f$. 

Recall we fix vector $\xi$ to define $\mu^\lambda_\xi$K-stability, while now we have a vector-free characterization of $\mu^\lambda$-cscK metric. 
To match up, it is better to have an algebro-geometric criterion for $\xi$ which tells if $(X, L)$ could be $\mu^\lambda_\xi$K-semistable. 
However, in general, there may be no such vector. 
In such case, we are interested in finding ``most destabilizing $q$'' in a suitable sense. 
It is revealed in \cite{Ino5} that the nature of these two questions are actually the same. 
We can formalize and solve both questions in terms of toric non-archimedean $\mu$-entropy, which is essentially the functional $F_{\mathcal{P}} (T, \bullet)$ in our main theorem. 

\subsubsection{Toric non-archimedean $\mu$-entropy}

Let $\mathcal{P}$ be the general system associated to a polarized toric variety $(X, L)$. 
For a lsc convex function $q \in \E^{\exp, 1} (P)$, we put 
\begin{align}
\bm{\sigma} (q)
&:= \frac{\int_P (n+q) e^q d\mu}{\int_P e^q d\mu} - \log \int_P e^q d\mu, 
\\
\NAmu (q) 
&:= -2\pi \frac{\int_{\partial P} e^q d\sigma}{\int_P e^q d\mu}.
\end{align}
For $\lambda \in \mathbb{R}$, we consider 
\begin{align}
\NAmu^\lambda (q) := 
\begin{cases}
\NAmu (q) + \lambda \bm{\sigma} (q)
& \bm{\sigma} (q) < \infty
\\
-\infty
& \bm{\sigma} (q) = \infty
\end{cases}
\end{align}
and call it the \textit{toric non-archimedean $\mu$-entropy}. 

These are essentially the functionals we already introduced: 
\begin{align*}
\NAmu (q) 
&= -2\pi U_{\mathcal{P}} (u (q)),
\\
\bm{\sigma} (q)
&= - S_{\mathcal{P}} (u(q)) - \log \int_P e^{-n} d\mu,
\\
\NAmu^\lambda (q) 
&= - 2\pi F_{\mathcal{P}} (-\frac{\lambda}{2\pi}, u(q)) - \lambda \log \int_P e^{-n} d\mu. 
\end{align*}

The differential of toric non-archimedean $\mu$-entropy $\NAmu^\lambda$ gives the minus of $\mu^\lambda_\xi$-Futaki invariant: 
\begin{align} 
\label{NAmu and muFutaki}
\frac{d}{dt}\Big{|}_{t=0} \NAmu^\lambda (\ket{\xi} + t q) = - \mathrm{Fut}^\lambda_\xi (q). 
\end{align}
This implies that if $\NAmu^\lambda$ is maximized at 
a vector $\ket{\xi}$, then $(X, L)$ is toric $\mu^\lambda_\xi$K-semistable. 

It is proved in \cite{Ino4} that there is an inequality between $\NAmu^\lambda$ and $\bm{\mu}_{\mathrm{Per}}^\lambda$ for $\lambda \in \mathbb{R}$: 
\begin{align}
\label{NAmu Permu inequality}
\sup_{q \in \mathcal{H}_{\mathrm{NA}} (X, L)} \NAmu^\lambda (q) \le \inf_{\omega \in \mathcal{H} (X, L)} \bm{\mu}_{\mathrm{Per}}^\lambda (\omega). 
\end{align}
Moreover if $\lambda \le 0$ and there exists a $\mu^\lambda_\xi$-cscK metric $\omega$, then the above equality is achieved by $\NAmu^\lambda (\ket{\xi}) = \bm{\mu}_{\mathrm{Per}}^\lambda (\omega)$. 
This implies if there exists a $\mu^\lambda_\xi$-cscK metric for $\lambda \le 0$, then $\ket{\xi}$ maximizes $\NAmu^\lambda$, which in particular gives another proof for the $\mu^\lambda_\xi$K-semistability of $\mu^\lambda_\xi$-cscK manifold with $\lambda \le 0$. 
It is conjectured the equality holds for $\lambda \le 0$ even when there is no $\mu^\lambda$-cscK metric. 

The non-archimedean $\mu$-entropy $\NAmu^\lambda$ is essentially introduced in \cite{Ino3} for $q \in \mathcal{H}_{\mathrm{NA}} (P)$ and more generally for a test configuration $(\mathcal{X}, \mathcal{L}) \in \mathcal{H}_{\mathrm{NA}} (X, L)$ of a non-toric polarized variety $(X, L)$. 
The paper \cite{Ino5} introduces a completion $\E^{\exp} (X, L)$ of the space $\mathcal{H}_{\mathrm{NA}} (X, L)$ of test configuraitons in terms of non-archimedean pluripotential theory developed in \cite{BJ1}. 
The space $\E^{\exp} (X, L)$ consists of functions on the Berkovich analytification of $X$, which is the reason for the word ``non-archimedean''. 
In the toric setup, the $T$-invariant part of the space $\E^{\exp} (X, L)$ corresponds to the intersection $\bigcap_{p \ge 1} \E^{\exp, p} (P)$. 
It takes considerable pages to extend the non-archimedean $\mu$-entropy to $\E^{\exp} (X, L)$ as an upper semi-continuous functional, which in our toric case turns extremely easy to show. 
In the toric case, $\NAmu^\lambda (q)$ is continuous along increasing sequence by the monotone convergence theorem, so we may replace $\sup_{q \in \mathcal{H}_{\mathrm{NA}} (X, L)} \NAmu^\lambda (q)$ in (\ref{NAmu Permu inequality}) with $\sup_{q \in \E^{\exp, 1} (P)} \NAmu^\lambda (q)$. 

Theorem \ref{main existence} confirms that there always exists a maximizer $q \in \E^{\exp, \frac{n}{n-1}} (P)$ of $\NAmu^\lambda$, even when the toric variety $(X, L)$ is $\mu^\lambda$K-unstable. 
If $q$ is a bounded convex function, we can assign a filtration $\mathcal{F}_q = \mathcal{F}_{\varphi_q}$ on the graded ring $R (X, L) = \bigoplus_m H^0 (X, L^{\otimes m})$ as in \cite{Ino5}. 
If further the filtration is finitely generated, which is the case for instance when $q \in \mathcal{H}_{\mathrm{NA}} (P)$, it produces an algebro-geometric degeneration $(\mathcal{X}, \mathcal{L})/B_\sigma$ of $(X, L)$ called polyhedral configuration in \cite{Ino5}. 
Then we can show that the central fibre $(\mathcal{X}_0, \mathcal{L}_0)$ of such degeneration is $\mu^\lambda_\xi$K-semistable with respect to the vector $\xi$ generated by the degeneration. 
See \cite{Ino5} for further detail. 

\subsubsection{Other works in the same spirit}

One ultimate goal of our study is to construct a degeneration of $\mu$K-unstable polarized variety to a $\mu$K-semistable ``space'' (see Remark \ref{remark on finite generation}), which is often referred to as optimal degeneration. 
The non-archimedean $\mu$-entropy is not the unique quantity measuring K-instability. 
We briefly review other works on K-instability as reference. 

$\bullet$ \textit{Optimal degeneration in the context of K\"ahler--Ricci soliton} 

The first pioneering work on optimal degeneration in the context of K\"ahler--Ricci soliton on Fano manifold would be Chen--Wang's work \cite{CW}. 
They constructed a metric geometric degeneration along K\"ahler--Ricci flow. 
Chen--Sun--Wang \cite{CSW} gives an algebro-geometric description of the degeneration, in which the idea of using filtration and two step degeneration is utilized. 
Dervan--Sz\'ekelyhidi \cite{DS} introduced a quantity called $H$-invariant and showed Chen--Sun--Wang's filtration minimizes the $H$-invariant. 
Han--Li \cite{HL} pursued an algebro-geometric aspect of $H$-invariant and proved the uniqueness of Chen--Sun--Wang's degeneration. 
Finally, Blum--Liu--Xu--Zhuang \cite{BLXZ} constructed Chen--Sun--Wang's degeneration in a purely algebro-geometric way, which allows them to extend the result to Fano variety with singularity. 
It is known by Wang--Zhu's work \cite{WZ} that every toric Fano manifold admits K\"ahler--Ricci soliton. 
This result can also be regarded as the explicit description of optimizer in toric case: the optimizer for $H$-invariant is given by the vector associated to K\"ahler--Ricci soliton. 
In \cite{Ino5}, it turns out that the non-archimedean $\mu$-entropy also characterizes Chen--Sun--Wang's degeneration. 

$\bullet$ \textit{Optimal degeneration in the context of extremal metric} 

Optimal degeneration problem of polarized variety would be firstly considered by Donaldson \cite{Don2} in the context of extremal metric. 
The quantity for K-instability in this context is called normalized Donaldson--Futaki invariant. 
In toric case, Sz\'ekelyhidi \cite{Sze} proved the existence of optimal destabilizing $L^2$-integrable convex function. 
Xia \cite{Xia} pursued the problem from pluripotential theoretic perspective and proved the existence of an optimal destabilizing geodesic ray in a suitable completion $\mathcal{E}^2 (X, L)$ of the space of K\"ahler metrics. 
It turns out by Li's analysis \cite{Li1} that the geodesic ray can be interpreted as a non-archimedean metric. 
Compared to K\"ahler--Ricci soliton, much less is known on the regularity of optimal destabilizer. 
We speculate this context can be treated as the limit $\lambda \to -\infty$ of our framework. 

$\bullet$ Optimal degeneration in the context of Mabuchi soliton: 

Optimal destabilization problem for Mabuchi soliton would be firstly considered by Hisamoto \cite{His2}. 
The quantity for K-instability in this context is called normalized Ding invariant. 
In toric case, Yao \cite{Yao} constructed an optimal destabilizing convex function for normalized Ding invariant. 
He also proved the crucial regularity result: the optimizer is piecewise affine. 
This is much better than the current knowledge on the optimizer for normalized Donaldson--Futaki invariant and non-archimedean $\mu$-entropy. 

$\bullet$ Still many other works...

There are still many other works on optimal degeneration. 
Among all, the structure of the above three frameworks is closest to that of our framework. 
For instance, we have Donaldson type inequality (cf. \cite{Don2, DS, His2, Xia, Ino4}) which can be interpreted as minimax theorem on some functional (cf. \cite{Ino4, Ino5}). 
Thus here we do not give further comments on those other works in the context of $\delta$-invariant and the normalized volume. 

\subsubsection*{Acknowledgments} This work is supported by RIKEN iTHEMS Program. 
We used \textit{Wolfram Mathematica} to create some graphs. 

\section{Convex analysis on simple polytope}

\subsection{Local estimates}

Here we prepare some estimates we use in the next section. 

\subsubsection{Mean value estimate}

\begin{prop}
\label{mean value estimate}
Let $u: \prod_{i=1}^n [a_i, b_i] \to [0, \infty]$ be a convex function. 
Then we have 
\[ u (x) \le \int_{\prod_{i=1}^n [a_i, b_i]} u d\bm{x} \cdot \frac{1}{\prod_{i=1}^n \min \{ x_i -a_i, b_i - x_i \}} \]
for every $x \in \prod_{i=1}^n [a_i, b_i]$. 
Here $d\bm{x}$ denotes the $n$-dimensional Lebesgue measure. 
\end{prop}

\begin{proof}
The claim is trivial if $x_i = a_i$ or $x_i = b_i$ for some $i$. 
By supporting hyperplane theorem, for $x \in \prod_{i=1}^n (a_i, b_i)$, we can take an affine function $\ell_x (y) = \sum_{i=1}^n \ell^i_x (y_i - x_i) + h$ so that $\ell_x \le u$ on $\prod_{i=1}^n [a_i, b_i]$ and $\ell_x (x) = h = u (x)$. 
If we put 
\[ \Box_i := \begin{cases} [a_i, x_i] & \ell_x^i < 0 \\ [x_i, b_i] & \ell^i_x \ge 0 \end{cases}, \]
we have $\ell_x \ge u (x)$ on $\prod_{i=1}^n \Box_i \subset \prod_{i=1}^n [a_i, b_i]$. 
Since $u$ is non-negative, we have $u \ge \max \{ \ell_x, 0 \}$, so that 
\[ \int_{\prod_{i=1}^n [a_i, b_i]} u d\bm{x} \ge \int_{\prod_{i=1}^n [a_i, b_i]} \max \{ \ell_x, 0 \} d\bm{x} \ge \int_{\prod_{i=1}^n \Box_i} d\bm{x} \cdot u (x). \]
Thus we get 
\[ u (x) \le \int_{\prod_{i=1}^n [a_i, b_i]} u d\bm{x} \cdot \frac{1}{\prod_{i=1}^n \mathrm{vol} \Box_i}. \]
Since $\mathrm{vol} \Box_i \ge \min \{ x_i -a_i, b_i - x_i \}$, we get the desired inequality. 
\end{proof}

We put 
\begin{align} 
\Delta^n_p (r) := \{ x \in [0, \infty)^p \times (-r, r)^{n-p} ~|~ \sum_{i=1}^p x_i < r \} . 
\end{align}

\begin{cor}
\label{convenient mean value estimate}
Let $u: \Delta^n_p (r) \to [0, \infty]$ be a convex function. 
For $\frac{1}{2} r \le r' < r$ and for every $x \in \Delta^n_p (r')$, we have
\[ u (x) \le \frac{(r')^p}{(r-r')^n} \int_{\Delta^n_p (r)} u d\bm{x} \cdot \frac{1}{\prod_{i=1}^p x_i}. \]
\end{cor}

\begin{proof}
For $x \in \Delta^n_p (r')$, we have $B (x) := \prod_{i=1}^p [0, \frac{r}{r'} x_i] \times \prod_{j=p+1}^n (-r, r) \subset \Delta^n_p (r)$, so that 
\begin{align*} 
u (x) 
&\le \int_{B (x)} u d\bm{x} \cdot \frac{1}{\prod_{i=1}^p \min \{ x_i, (\frac{r}{r'} -1)x_i \} \cdot (r-r')^{n-p}} 
\\
&\le \frac{(r')^p}{(r-r')^n} \int_{\Delta^n_p (r)} u d\bm{x} \cdot \frac{1}{\prod_{i=1}^p x_i} 
\end{align*}
by the above proposition. 
Here we note $\min \{ x_i, (\frac{r}{r'} -1) x_i \} = \frac{1}{r'} (r-r') x_i$ when $\frac{1}{2} r \le r'$. 
\end{proof}

\subsubsection{Boundary mean value estimate}

For $i=1, \ldots, p$, we put 
\begin{align} 
\partial_i \Delta^n_p (r) := \{ x \in \Delta^n_p (r) ~|~ x_i = 0 \} \cong \Delta^{n-1}_{p-1} (r) 
\end{align}
and consider the $(n-1)$-dimensional Lebesgue measure $d\bm{x}'$. 

\begin{prop}
Suppose $p \ge 2$. 
Let $u: \Delta^n_p (r) \to [0, \infty]$ be a convex function. 
Then for $i \neq j \in \{ 1, \ldots, p \}$, we have 
\[ \Big{(} \frac{r}{2} \Big{)}^{n-p} u (x) \le \int_{\partial_i \Delta^n_p (r)} u d\bm{x}' \cdot \frac{x_j}{(x_i + x_j)^2 \prod_{k \neq i,j} x_k} + \int_{\partial_j \Delta^n_p (r)} u d\bm{x}' \cdot \frac{x_i}{(x_i + x_j)^2 \prod_{k \neq i,j} x_k} \]
for every $x \in \Delta^n_p (\frac{r}{2})$. 
Here $k$ in $\prod$ runs over $1, \ldots, p$, excluding $i,j$. 
\end{prop}

\begin{proof}
The claim is trivial on $\{ x_i = x_j = 0 \}$. 
We assume $(x_i, x_j) \neq (0,0)$. 
For 
\begin{align*} 
\bm{x}^i 
&:= (x_1, \ldots, x_{i-1}, x_i+x_j, x_{i+1}, \ldots, x_{j-1}, 0, x_{j+1}, \ldots, x_n) \in \partial_i \Delta^n_p (r), 
\\
\bm{x}^j 
&:= (x_1, \ldots, x_{i-1}, 0, x_{i+1}, \ldots, x_{j-1}, x_i + x_j, x_{j+1}, \ldots, x_n) \in \partial_j \Delta^n_p (r), 
\end{align*}
we have 
\[ (x_1, \ldots, x_n) = \frac{x_i}{x_i+x_j} \bm{x}^j + \frac{x_j}{x_i + x_j} \bm{x}^i. \]
Then by convexity, 
\[ u (x) \le \frac{x_i}{x_i + x_j} u (\bm{x}^j) + \frac{x_j}{x_i+x_j} u (\bm{x}^i). \]
When $x \in \Delta^n_p (\frac{r}{2})$, we have $\bm{x}^i \in \partial_i \Delta^n_p (\frac{r}{2}) \cong \Delta^{n-1}_{p-1} (\frac{r}{2})$, so that the claim follows by Corollary \ref{convenient mean value estimate} applied to $u (\bm{x}^i), u(\bm{x}^j)$ on $\Delta^{n-1}_{p-1} (\frac{r}{2})$. 
\end{proof}

\begin{cor}
\label{local Rellich type estimate}
Suppose $p \ge 2$. 
There exists a function $U: \Delta^n_p (r) \to [0, \infty]$ such that $U \in L^q (\Delta^n_p (r))$ for every $q \in [1, \frac{p}{p-1})$ and 
\[ u (x) \le \sum_{i=1}^p \int_{\partial_i \Delta^n_p (r)} u d\bm{x}' \cdot U (x) \]
for every convex function $u: \Delta^n_p (r) \to [0, \infty]$ and $x \in \Delta^n_p (\frac{r}{2})$. 
\end{cor}

\begin{proof}
By the above proposition, we have 
\[ u(x) \le \Big{(} \frac{2}{r} \Big{)}^{n-p} \sum_{i=1}^p \int_{\partial_i \Delta^n_p (r)} u d\bm{x}' \cdot \min \Big{\{} \frac{1}{(x_i+x_j) \prod_{k \neq i, j} x_k} ~\Big{|}~ i \neq j \in \{ 1, \ldots, p \} \Big{\}}. \]
Here $k$ in $\prod$ runs over $1, \ldots, p$ except for $i,j$. 
We compute
\begin{align*} 
\max \{ (x_i+x_j) \prod_{k \neq i, j} x_k ~|~ i \neq j \} 
&\ge \binom{p}{2}^{-1} \sum_{i,j: i < j} (x_i+x_j) \prod_{k \neq i, j} x_k
\\
&= \frac{2}{p} \sum_{j=1}^p \prod_{k \neq j} x_k
\\
&\ge 2 \sqrt[p]{\prod_{j=1}^p \prod_{k \neq j} x_k} = 2 \prod_{j=1}^p x_j^{\frac{p-1}{p}}. 
\end{align*}
It follows that  
\[ u (x) \le \sum_{i=1}^p \int_{\partial_i \Delta^n_p (r)} u d\bm{x}' \cdot \Big{(} \frac{2}{r} \Big{)}^{n-p} \frac{1}{2 \prod_{j=1}^p x_j^{\frac{p-1}{p}}}. \]
We can easily check the function 
\[ U (x) = \Big{(} \frac{2}{r} \Big{)}^{n-p} \frac{1}{2 \prod_{j=1}^p x_j^{\frac{p-1}{p}}} \]
enjoys the desired convexity and integrability. 
\end{proof}

\begin{prop}
\label{local Poincare type estimate}
Suppose $p \ge 2$. 
Let $u: \Delta^n_p (r) \to [0, \infty]$ be a convex function. 
Then we have 
\[ \int_{\Delta^n_p (\frac{r}{2})} u^{\frac{p}{p-1}} d\bm{x} \le (p-1) \Big{(} \frac{2}{r} \Big{)}^{\frac{n-p}{p-1}} \sum_{i=1}^p \Big{(} \int_{\partial_i \Delta^n_p (r)} u d\bm{x}' \Big{)}^{\frac{p}{p-1}}. \]
\end{prop}

\begin{proof}
For $i, j \in \{ 1, \ldots, p \}$, we put 
\begin{align}
\Delta^n_{p, ij} (r) := \{ (x_1, \ldots, x_n) \in \Delta^n_p (r) ~|~ x_i, x_j \le \min_{k \neq i,j} x_k \}. 
\end{align}
Then we have 
\[ \Delta^n_p (r) = \bigcup_{i < j} \Delta^n_{p, ij} (r). \]

Now we consider 
\[ P_{ij} (r) := \{ (t_{ij}; r_{ij}, \bm{x}^{ij}) \in [0,1] \times \Delta^{n-1}_{p-1} (r) ~|~ t_{ij} r_{ij}, (1-t_{ij})r_{ij} \le \min_{k \neq i,j} x_k \}, \]
where $k$ in $\min$ runs over $1, \ldots, p$ except for $i, j$ and we put
\begin{align*}
\bm{x}^{ij} 
&:= (x_1, \ldots, x_{i-1}, x_{i+1}, \ldots, x_{j-1}, x_{j+1}, \ldots, x_n).
\end{align*}
By 
\[ (t_{ij}; r_{ij}, \bm{x}^{ij}) \mapsto (x_1, \ldots, x_{i-1}, t_{ij} r_{ij}, x_{i+1}, \ldots, x_{j-1}, (1-t_{ij}) r_{ij}, x_{j+1}, \ldots, x_n), \]
we get a map from $P_{ij} (r)$ onto $\Delta^n_{p, ij} (r)$, which gives a homeomorphism outside the zero set $\{ x_i = x_j = 0 \} \subset  \Delta^n_{p,ij} (r)$. 
The Lebesgue measure $d\bm{x} = dx_1 \dotsb dx_n$ is transformed into 
\[ dt_{ij} r_{ij} dr_{ij} d\bm{x}^{ij} = dt_{ij} r_{ij} dr_{ij} dx_1 \dotsb dx_{i-1} dx_{i+1} \dotsb dx_{j-1} dx_{j+1} \dotsb dx_n. \]
A convex function $f$ on $\Delta^n_{p,ij}$ gives a convex function $f_{t_{ij}} (r_{ij}, \bm{x}^{ij}) := f (t_{ij}; r_{ij}, \bm{x}^{ij})$ on $\Delta^{n-1}_{p-1} (r)$ for each $t_{ij} \in [0,1]$. 
We also have 
\[ f (t_{ij}; r_{ij}, \bm{x}^{ij}) \le t_{ij} f_1 (r_{ij}, \bm{x}^{ij}) + (1-t_{ij}) f_0 (r_{ij}, \bm{x}^{ij}). \]

Since $u^{\frac{p}{p-1}}$ is convex, we compute
\begin{align*} 
\int_{\Delta^n_{p, ij} (\frac{r}{2})} u^{\frac{p}{p-1}} d\bm{x}
&= \int_{P_{ij} (\frac{r}{2})} dt_{ij} d\bm{x}^{ij} dr_{ij} r_{ij} u (t_{ij}; r_{ij}, \bm{x}^{ij})^{\frac{p}{p-1}} 
\\
&\le \int_{P_{ij} (\frac{r}{2})} dt_{ij} d\bm{x}^{ij} dr_{ij} r_{ij} (t_{ij} u_1 (r_{ij}, \bm{x}^{ij})^{\frac{p}{p-1}} + (1-t_{ij}) u_0 (r_{ij}, \bm{x}^{ij})^{\frac{p}{p-1}})
\\
& \le \int_{P_{ij} (\frac{r}{2})} dt_{ij} d\bm{x}^{ij} dr_{ij} u_1 (r_{ij}, \bm{x}^{ij}) \Big{(} \Big{(} \frac{2}{r} \Big{)}^{n-p} \int_{\Delta^{n-1}_{p-1} (r)} u_1 (r_{ij}, \bm{x}^{ij}) d\bm{x}' \cdot \frac{(r_{ij} t_{ij})^{p-1}}{r_{ij} \prod_{k \neq i,j} x_k} \Big{)}^{\frac{1}{p-1}}
\\
&+ \int_{P_{ij} (\frac{r}{2})} dt_{ij} d\bm{x}^{ij} dr_{ij} u_0 (r_{ij}, \bm{x}^{ij})) \Big{(} \Big{(} \frac{2}{r} \Big{)}^{n-p} \int_{\Delta^{n-1}_{p-1} (r)} u_0 (r_{ij}, \bm{x}^{ij}) d\bm{x}' \cdot \frac{(r_{ij} (1-t_{ij}))^{p-1}}{r_{ij} \prod_{k \neq i,j} x_k} \Big{)}^{\frac{1}{p-1}}.
\end{align*}
Here again $k$ in $\prod$ runs over $1, \ldots, p$ except for $i, j$. 
Since 
\[ \frac{(r_{ij} t_{ij})^{p-1}}{r_{ij} \prod_{k \neq i,j} x_k}, \frac{(r_{ij} (1-t_{ij}))^{p-1}}{r_{ij} \prod_{k \neq i,j} x_k} \le 1 \]
on $P_{ij} (r)$, we get 
\begin{align*}
\int_{\Delta^n_{i,j} (\frac{r}{2})} u^{\frac{n}{n-1}} d\bm{x}
&\le \int_{P_{ij} (\frac{r}{2})} dt_{ij} d\bm{x}^{ij} dr_{ij} u_1 (r_{ij}, \bm{x}^{ij}) \Big{(} \Big{(} \frac{2}{r} \Big{)}^{n-p} \int_{\Delta^{n-1}_{p-1} (r)} u_1 (r_{ij}, \bm{x}^{ij}) d\bm{x}' \Big{)}^{\frac{1}{p-1}} 
\\
&+ \int_{P_{ij} (\frac{r}{2})} dt_{ij} d\bm{x}^{ij} dr_{ij} u_0 (r_{ij}, \bm{x}^{ij})) \Big{(} \Big{(} \frac{2}{r} \Big{)}^{n-p} \int_{\Delta^{n-1}_{p-1} (r)} u_0 (r_{ij}, \bm{x}^{ij}) d\bm{x}' \Big{)}^{\frac{1}{p-1}}
\\
&= \Big{(} \frac{2}{r} \Big{)}^{\frac{n-p}{p-1}} \int_{\Delta^{n-1}_{p-1} (\frac{r}{2})} u_1 (r_{ij}, \bm{x}^{ij}) d\bm{x}' \Big{(} \int_{\Delta^{n-1}_{p-1} (r)} u_1 (r_{ij}, \bm{x}^{ij}) d\bm{x}' \Big{)}^{\frac{1}{p-1}}
\\
&+ \Big{(} \frac{2}{r} \Big{)}^{\frac{n-p}{p-1}} \int_{\Delta^{n-1}_{p-1} (\frac{r}{2})} u_0 (r_{ij}, \bm{x}^{ij}) d\bm{x}' \Big{(}\int_{\Delta^{n-1}_{p-1} (r)} u_0 (r_{ij}, \bm{x}^{ij}) d\bm{x}' \Big{)}^{\frac{1}{p-1}}
\\
&\le \Big{(} \frac{2}{r} \Big{)}^{\frac{n-p}{p-1}} \Big{[} \Big{(} \int_{\partial_j \Delta^n_p (r)} u d\bm{x}' \Big{)}^{\frac{p}{p-1}} + \Big{(} \int_{\partial_i \Delta^n_p (r)} u d\bm{x}' \Big{)}^{\frac{p}{p-1}} \Big{]}.
\end{align*}
Taking the sum $\sum_{i < j}$ of this, we obtain the desired estimate. 
\end{proof}

\subsection{Rellich and Poincare type estimates}
\label{Rellich and Poincare type estimates}

Here we establish a crucial compactness result. 

\subsubsection{Rellich type estimate}

\begin{prop}
Let $P = P$ be a polytope. 
Using a flat measure $d\mu$, we put 
\[ \delta_P (x) := \inf \{ \mathrm{vol} (P \cap \ell^{-1} ([0, \infty))) ~|~ \ell: V \to \mathbb{R} \text{ is an affine function with } \ell (x) = 0 \}. \]
Then $\delta_P (x) > 0$ for every $x \in P^\circ$. 
Moreover, for any non-negative convex function $u$ and $x \in P$, we have 
\[ u (x) \le \int_P u d\mu \cdot \frac{1}{\delta_P (x)}. \]
\end{prop}

\begin{proof}
Let $S (V)$ be the set of half spaces $H^+_\ell = \ell^{-1} ([0,\infty))$ with $\ell (x) = 0$. 
We identify $S (V)$ with a sphere as topological space. 
Then the map $S (V) \to [0, \infty): H^+ \mapsto \mathrm{vol} (P \cap H^+)$ is continuous, so that the infimum of $\delta_P (x)$ is attained by some $H^+$, so that we have $\delta_P (x) > 0$ on $P^\circ$ and $\delta_P (x) = 0$ on $\partial P$. 
The inequality follows by the same argument as in the proof of Proposition \ref{mean value estimate}. 
\end{proof}

\begin{thm}
\label{Rellich type estimate}
Let $\mathcal{P} = (P, d\mu, d\sigma)$ be an $n$-dimensional system. 
There exists a lsc convex function $U: P \to [0, \infty]$ such that $U \in L^q (P, d\mu)$ for every $q \in [1, \frac{n}{n-1})$ and 
\[ u (x) \le \int_{\partial P} u d\sigma \cdot U (x) \]
for every non-negative lsc convex function $u: P \to [0, \infty]$ and $x \in P$. 

Moreover, we can take such $U$ so that it is finite valued and continuous on $P \setminus \partial^2 P$ and for each $x \in \partial^2 P$ lying in the relative interior of a codimension $p \ge 2$ face, there exists a neighbourhood $V$ such that $U|_V \in L^r (V, d\mu|_V)$ for every $r \in [1, \frac{p}{p-1})$. 
\end{thm}

\begin{proof}
We put 
\[ U (x) := \sup \Big{\{} \frac{u (x)}{\int_{\partial P} u d\sigma} ~\Big{|}~ u: P \to [0, \infty] \text{ is lsc convex with } 0 < \int_{\partial P} u d\sigma < \infty \Big{\}}. \]
This is clearly a lsc convex function on $P$. 

We firstly show that $U$ is finite valued and continuous on $P \setminus \partial^2 P$. 
For $x \in \partial P \setminus \partial^2 P$ lying in the relative interior of a facet $Q \subset P$, we have 
\[ u (x) \le \int_Q u d\sigma \cdot \frac{1}{\delta_Q (x)} \le \int_{\partial P} u d\sigma \cdot \frac{1}{\delta_Q (x)}, \]
so that 
\[ U (x) \le \frac{1}{\delta_Q (x)}. \]
Thus $U (x)$ is finite for $x \in \partial P \setminus \partial^2 P$. 
For $x \in P \setminus \partial^2 P$, we can take $x', x'' \in \partial P \setminus \partial^2 P$ and $t \in [0,1]$ so that $x = (1-t) x' + t x''$. 
Then by the convexity, we have $U (x) \le (1-t) U (x') + t U (x'') < \infty$. 

Since $P$ is a polytope, the upper semi-continuity of $U$ follows by the convexity. 
Indeed, as for the interior $P^\circ$, it is well-known that any convex function is continuous on open set. 
The upper semi-continuity around $\partial P \setminus \partial^2 P$ can be seen as follows. 
Take a local neigbourhood of a point $x \in \partial P \setminus \partial^2 P$ so that it is affine isomorphic to $\Delta^n_1 (2r)$. 
Then for any convex function $u$ on $\Delta^n_1 (r)$ we compute 
\[ u (x^i_1, \ldots, x^i_n) \le (1- \frac{x^i_1}{r}) u (0, x^i_2, \ldots, x^i_n) + \frac{x^i_1}{r} u (r, x^i_2, \ldots, x^i_n). \]
When $x^i = (x^i_1, \ldots, x^i_n) \to 0$, we have 
\[ (1- \frac{x^i_1}{r}) u (0, x^i_2, \ldots, x^i_n) + \frac{x^i_1}{r} u (r, x^i_2, \ldots, x^i_n) \to u (0) \]
by the continuity of $u$ on $\Delta^n_1 (2r)^\circ$ and on $\partial_1 \Delta^n_1 (2r)$. 
Thus we get 
\begin{align}
\label{upper semi-continuity}
\limsup_{i \to \infty} u (x^i) \le u (0), 
\end{align}
which shows the upper semi-continuity. 

Now it follows that for any compact set $K \subset P \setminus \partial^2 P$, there exists a constant $C > 0$ such that $U|_K \le C$. 
In particular, $U|_K$ is $L^q$ for any $q$. 

It suffices to show the integrability around $\partial^2 P$. 
We can reduce this task to show for every point $x_0 \in \partial^2 P$ there exists a neighbourhood $V$ of $x_0$ in $P$ and a function $f: V \to [0, \infty]$ such that $f \in L^q (V, d\mu)$ for every $q \in [1, \frac{n}{n-1})$ and $U|_V \le f$. 

For any boundary point $x_0 \in \partial P$ lying in the relative interior of dimension $p$ face, we can take an affine map $\phi: \mathbb{R}^n \to M_\mathbb{R}$ so that 
\begin{itemize}
\item $\phi (0) = x_0$,

\item The measure $d\mu$ on $V$ is transformed into the Lebesgue measure $d\bm{x}$ on $\mathbb{R}^n$. 

\item There exists $r > 0$ such that $\phi$ gives a homeomorphism from $\Delta^n_p (r)$ onto an open neighbourhood of $x_0$ in $P$. 

\item For each $i = 1,\ldots, p$, there exists a facet $Q_i \subset P$ containing $x_0$ such that $\phi$ gives a homeomorphism from $\partial_i \Delta^n_p (r)$ onto an open neighbourhood of $x_0$ in $Q_i$. 
\end{itemize}
Since $\phi$ is affine, the measure $d\sigma$ on $Q_i$ is transformed into $a_i d\bm{x}'$ on $\partial_i \Delta^n_p (r)$ for some constant $a_i > 0$. 

For a convex function $u$ on $P$, we have 
\[ \sum_{i=1}^p \int_{\partial_i \Delta^n_p (r)} u \circ \phi d\bm{x}' = \sum_{i=1}^p a_i^{-1} \int_{Q_i \cap \phi (\partial_i \Delta^n_p (r))} u d\sigma \le \frac{1}{\min \{ a_i \}} \int_{\partial P} u d\sigma. \]
Then by Corollary \ref{local Rellich type estimate}, we have a function $\tilde{f}: \Delta^n_p (r) \to [0,\infty]$ such that $\tilde{f} \in L^q (\Delta^n_p (r))$ for every $q \in [1, \frac{n}{n-1})$ and 
\[ u \circ \phi (x) \le \int_{\partial P} u d\sigma \cdot \tilde{f} (x) \] 
on $\Delta^n_p (r/2)$ for every convex $u: P \to [0, \infty]$. 
Now we put $V := \phi (\Delta^n_p (r/2))$ and $f := \tilde{f} \circ (\phi^{-1})|_V$, then we get the desired property. 
\end{proof}

\begin{cor}
\label{Rellich type compactness}
Let $\{ u_i: P \to [0, \infty] \}_{i \in \mathbb{N}}$ be a sequence of non-negative convex functions with a uniform bound 
\[ \int_{\partial P} u_i d\sigma \le C. \]
After taking a subsequence, there exists a unique lower semi-continuous convex function $u: P \to [0, \infty]$ such that $u_i$ converges to $u$ in $L^q$-topology for every $q \in [0, \frac{n}{n-1})$. 
Moreover the convergence is uniform on each compact set $K \subset P^\circ$. 
\end{cor}

\begin{proof}
If we put
\[ \hat{u}_i := \sup \{ \ell (x) ~|~ \ell: P \to \mathbb{R} \text{ is an affine function s.t. } \ell \le u \}, \]
we have $\hat{u}_i|_{P^\circ} = u_i|_{P^\circ}$ by the supporting hyperplane theorem and $\int_{\partial P} \hat{u}_i d\sigma \le \int_{\partial P} u_i d\sigma \le C$. 
Thus we may assume $u_i$ are lower semi-continuous. 
By the above theorem, we have 
\[ u_i (x) \le C \cdot U (x). \]
It follows by the continuity of $U$ on $P \setminus \partial^2 P$ that for any compact set $K \subset P \setminus \partial^2 P$, $u_i|_K$ is uniformly bounded. 
Then by a general argument, for any compact set $K' \subset P^\circ$, we get a uniform Lipschitz bound $|\nabla u_i|_{K'}| \le C_{K'}$. 
Then by Arzel\`a--Ascoli theorem, we can find a subsequence $u_i$ and a function $u^\circ$ on $P^\circ$ so that $u_i$ converges uniformly to $u^\circ$ on every compact set $K \subset P^\circ$. 
Since $u^\circ \le C \cdot U (x)$, we have $u_i \to u^\circ$ in $L^q (P^\circ)$ for every $q \in [1, \frac{n}{n-1})$ by the dominated convergence theorem. 

Now we consider the lsc envelope 
\[ u (x) := \sup \{ \ell (x) ~|~ \ell:V \to \mathbb{R} \text{ is an affine function s.t.} \ell|_{P^\circ} \le u^\circ \} \]
on $P$. 
By the supporting hyperplane theorem, $u$ coincides with $u^\circ$ on $P^\circ$. 
Then since $\partial P$ is a zero set, we have $u_i \to u$ in $L^q (P)$ for every $q \in [1, \frac{n}{n-1})$. 

Since any convex function is continuous on open set, the $L^q$-convergence characterizes $u$ on $P^\circ$. 
Since $u$ restricted to a segment is automatically upper semi-continuous, the lowe semi-continuity of $u$ implies that $u|_{\partial P}$ is uniquely determined by $u|_{P^\circ}$, which shows the uniqueness. 
\end{proof}

\subsubsection{Poincare type estimate}

\begin{thm}
\label{Poincare type estimate}
Let $\mathcal{P} = (P, d\mu, d\sigma)$ be an $n$-dimensional system. 
For every $q \in [1, \frac{n}{n-1}]$, there exists a constant $C_{P, q} > 0$ such that 
\[ \Big{(} \int_P u^q d\mu \Big{)}^{\frac{1}{q}} \le C_{P, q} \int_{\partial P} u d\sigma \]
for every non-negative convex function $u: P \to [0, \infty]$. 
\end{thm}

\begin{proof}
We may assume $u$ is lower semi-continuous. 
Since the measure $d\mu$ is finite, it suffices to show the claim for $q = \frac{n}{n-1}$. 
For each $x \in \partial^2 P$, we take an affine chart $\phi_x: \mathbb{R}^n \to M_{\mathbb{R}}$ as in the above proof and cover $P$ by open sets $V_x := \phi_x (\Delta^n_{p_x} (r_x/2))$ and $W = \{ x \in P \setminus \partial^2 P ~|~ U(x) < C \}$ for some $C > 0$. 
Take a finite subcover $\{ V_{x_1}, \ldots, V_{x_N}, W \}$ so that it still covers $P$. 
On $W$, we have 
\[ \Big{(} \int_W u^{\frac{n}{n-1}} d\mu \Big{)}^{\frac{n-1}{n}} \le C \mathrm{vol} (P)^{\frac{n-1}{n}}. \]
By Proposition \ref{local Poincare type estimate}, we have 
\begin{align*} 
\Big{(} \int_{V_{x_m}} u^{\frac{n}{n-1}} d\mu \Big{)}^{\frac{n-1}{n}} 
&= \Big{(} \int_{\Delta^n_{p_{x_m}} (r_{x_m}/2)} (u \circ \phi_{x_m})^{\frac{n}{n-1}} d\bm{x} \Big{)}^{\frac{n-1}{n}} 
\\
& \le \mathrm{vol} (\Delta^n_{p_{x_m}} (r_{x_m}/2))^{\frac{n-1}{n} - \frac{p_{x_m}-1}{p_{x_m}}} \Big{(} \int_{\Delta^n_{p_{x_m}} (r_{x_m}/2)} (u \circ \phi_{x_m})^{\frac{p_{x_m}}{p_{x_m}-1}} d\bm{x} \Big{)}^{\frac{p_{x_m}-1}{p_{x_m}}}
\\
&\le C_{x_m} \sum_{i=1}^{p_{x_m}} \int_{\partial_i \Delta^n_{p_{x_m}} (r_{x_m})} u \circ \phi_{x_m} d\bm{x}'
\\
&\le \frac{C_{x_m}}{\min \{ a_i (x_m) \}} \int_{\partial P} u d\sigma. 
\end{align*}
Here we put 
\[ C_{x_m} := (p_{x_m} (p_{x_m}-1))^{\frac{n-p_{x_m}}{p_{x_m}}} \Big{(} \frac{2}{r_{x_m}} \Big{)}^{\frac{n-p_{x_m}}{p_{x_m}}} \mathrm{vol} (\Delta^n_{p_{x_m}} (r_{x_m}/2))^{\frac{n-1}{n} - \frac{p_{x_m}-1}{p_{x_m}}}. \]
Take the sum for $m=1, \ldots, N$, we get the result. 
\end{proof}

\begin{rem}
By inductive argument, we obtain the following: 
Let $P = (P, d\mu)$ be an $n$-dimensional simple polytope. 
If we endow a flat measure $d\sigma_k$ on $\partial^k P = \bigcup \{ \text{ codimension $k$ faces } \}$, then for $q \in [1, \frac{n}{n-k}]$, there exists a constant $C_{P, d\sigma_k, q} > 0$ such that 
\[ \Big{(} \int_P u^q d\mu \Big{)}^{\frac{1}{q}} \le C_{P, d\sigma_k, q} \int_{\partial^k P} u d\sigma_k  \]
for every non-negative convex function $u: P \to [0, \infty]$. 
It would be interesting to find application of this fact to the higher regularity of minimizer of $F_{\mathcal{P}} (T, \bullet)$. 
\end{rem}

We use the following log Sobolev type estimate in the next section. 

\begin{cor}
\label{Entropy bound}
For $q \in (1, \frac{n}{n-1}]$, we have 
\[ \int_P u \log u d\mu \le \frac{q}{q-1} \log (C_{P, q} \int_{\partial P} u d\sigma) \]
for every non-negative convex function $u: P \to [0, \infty]$ with $\int_P u d\mu = 1$. 
\end{cor}

\begin{proof}
Since $u d\mu$ is a probability measure, we compute
\[ \int_P u \log u d\mu = \frac{1}{q-1} \int_P \log u^{q-1} ~u d\mu \le \frac{q}{q-1} \log \Big{(} \int_P u^q d\mu \Big{)}^{\frac{1}{q}} \]
by Jensen's inequality on $\log x$. 
The claim follows by applying the above theorem. 
\end{proof}

\section{Optimizers for non-archimedean $\mu$-entropy}

\subsection{Existence and Uniqueness of optimizers}

\subsubsection{Existence of minimizer of $F_{\mathcal{P}}$}

Here we show the existence part of main theorem. 
Let us firstly observe the lower semi-continuity of $F_{\mathcal{P}} (T, u)$. 

\begin{prop}
Suppose $u_i \in \M^{\exp, 1} (P)$ converges to $u_\infty \in \M^{\exp, 1} (P)$ in $L^1$-topology, then 
\begin{align}
\label{lower semi-continuity of internal energy}
U_{\mathcal{P}} (u_\infty) \le \liminf_{i \to \infty} U_{\mathcal{P}} (u_i). 
\end{align}
If moreover the $L^p$-norm of $u_i, u_\infty$ is bounded for some $p >1$, then we have
\begin{align}
S_{\mathcal{P}} (u_i) = \lim_{i \to \infty} S_{\mathcal{P}} (u_\infty).
\end{align}
\end{prop}

\begin{proof}
We firstly note the following fact: for a sequence $u_i$ of convex functions on $P$, if the restriction $u_i|_{P^\circ}$ pointwisely converges to a convex function $u^\circ$ on $P^\circ$, then for the lower semi-continuous extension $u$ on $P$ (see the proof of Corollary \ref{Rellich type compactness}), we have 
\begin{align} 
u (x) \le \liminf_{i \to \infty} u_i (x) 
\end{align}
for every point $x \in P$. 
This can be seen as follows. 
Take a point $p \in P^\circ$. 
For $x \in P$, we have 
\[ u_i ((1-t) p + t x) \le (1-t) u_i (p) + t u_i (x). \]
For $t \in [0,1)$, we have $(1-t) p + tx \in P^\circ$, so that we get 
\[ u ((1-t)p+tx) \le (1-t) u (p) + t \liminf_{i \to \infty} u_i (x) \]
by the pointwise convergence on $P^\circ$. 
As $u$ is lower semi-continuous, it is continuous on the segment $\{ (1-t) p +tx \}_{t \in [0,1]}$. 
Thus by taking the limit $t \to 1$, we get the claim. 

Now since $\int_P d\mu \cdot U_{\mathcal{P}} (u) = \int_{\partial P} u d\sigma$, the inequality (\ref{lower semi-continuity of internal energy}) is a consequence of Fatou's lemma. 
Only the pointwise convergence on $P^\circ$ is important for this. 

On the other hand, by mean value theorem on $t^2 \log t^2$, we compute
\begin{align*} 
|x \log x - y \log y| 
&= 2 |\sqrt{x} - \sqrt{y}| |2 \sqrt{z} \log \sqrt{z} + \sqrt{z}| 
\\
&\le 2 |\sqrt{x} - \sqrt{y}| (2 \max \{ e^{-1}, \epsilon^{-1} \sqrt{z}^{1+\epsilon}| \} + \sqrt{z})
\\
&\le 2 |\sqrt{x} - \sqrt{y}| (2 \max \{ e^{-1}, \epsilon^{-1} \sqrt{y}^{1+ \epsilon} \} + \sqrt{y})
\end{align*}
for $0 \le x \le y$ and $\epsilon > 0$, by taking suitable $z \in [x, y]$. 
It follows by Cauchy--Schwarz theorem that 
\begin{align*}
   \Big{|} \int_P u_i \log u_i d\mu 
   &- \int_P u_\infty \log u_\infty d\mu \Big{|}
   \\
   &\le 2 \int_P |\sqrt{u_i} - \sqrt{u_\infty}| 
   \\
   &\qquad \cdot (2 \max \{ e^{-1}, \epsilon^{-1} \sqrt{u_i}^{1+\epsilon}, \epsilon^{-1} \sqrt{u_\infty}^{1+\epsilon} \} + \max \{ \sqrt{u_i}, \sqrt{u_\infty} \}) d\mu
   \\
   &\le \| \sqrt{u_i} - \sqrt{u_\infty} \|_{L^2} \cdot C  
\end{align*}
with a constant $C$ which depends only on $\epsilon$ and a uniform bound on $L^{1+ \epsilon}$-norm of $u_i, u_\infty$. 
This proves the claim for $\int_P d\mu \cdot S_{\mathcal{P}} (u) = -\int_P u\log u d\mu$. 
\end{proof}

We prepare the following uniform estimate. 

\begin{lem}
For any $T, C \in \mathbb{R}$, there exists a constant $\tilde{C} > 0$ depending only on $T, C, n, C_{P, \frac{n}{n-1}}$ such that the following holds: if a non-negative convex function $u: P \to [0, \infty]$ with $\int_P u d\mu = 1$ satisfies 
\[ \int_{\partial P} u d\sigma + T \int_P u \log u d\mu \le C, \]
then we have 
\[ \int_{\partial P} u d\sigma \le \tilde{C}. \]
\end{lem}

\begin{proof}
When $T \ge 0$, we can choose $\tilde{C} = C$ since we have $\int_P u \log u d\mu \ge 0$ by Jensen's inequality on $x \log x$. 

Suppose $T < 0$. 
By Corollary \ref{Entropy bound}, we get  
\[ \int_{\partial P} u d\sigma \le - n T \log (C_{P, \frac{n}{n-1}} \int_{\partial P} u d\sigma) + C. \]
We get the result for the constant
\[ \tilde{C} := \sup \{ x \in [0, \infty) ~|~ x \le -nT \log (C_{P, \frac{n}{n-1}} x) + C \} < \infty. \]
\end{proof}

Now we prove our main existence theorem. 

\begin{thm}
\label{Main theorem on existence}
Let $\mathcal{P} = (P, d\mu, d\sigma)$ be an $n$-dimensional system. 
For every $T \in \mathbb{R}$, there exists a convex function $u \in \M^{\exp, \frac{n}{n-1}} (P)$ which minimizes $F_{\mathcal{P}} (T, \bullet)$ on $\M^{\exp, 1} (P)$. 
\end{thm}

\begin{proof}
We recall 
\[ \int_P d\mu \cdot F_{\mathcal{P}} (T, u) = \int_{\partial P} u d\sigma + T \int_P u \log u d\mu. \]
By the above lemma, for any constant $C$, $\int_{\partial P} e^q d\sigma$ is uniformly bounded on the subset 
\[ \{ u \in \M^{\exp, 1} (P) ~|~ F_{\mathcal{P}} (T, u) \le C \}. \]

Take a sequence $u_i \in \M^{\exp, 1} (P)$ so that $F_{\mathcal{P}} (T, u_i) \searrow \inf_u F_{\mathcal{P}} (T, u)$. 
Thanks to the above uniform bound and the compactness in  Corollary \ref{Rellich type compactness}, after taking a subsequence we may assume $u_i$ converges in $L^p$-topology ($p \in [1, \frac{n}{n-1})$) to a lsc log convex function $u_\infty$.  (Note log concavity is preserved by pointwise convergence. ) 
By the $L^1$-convergence, we have $\int_P u_\infty d\mu = \lim \int_P u_i d\mu = \int_P d\mu$, so that $u_\infty > 0$ and hence $u_\infty \in \M^{\exp, 1} (P)$. 
By the lower semi-continuity of $F_{\mathcal{P}} (T, \bullet)$ with respect to $L^{1+\epsilon}$-topology, we get $F_{\mathcal{P}} (T, u_\infty) = \inf_u F_{\mathcal{P}} (T, u)$, which shows the existence of minimizer. 
Finally, since $\int_{\partial P} u_\infty d\sigma \le \liminf_i \int_{\partial P} u_i d\sigma < \infty$ by the above uniform bound, we conclude $u_\infty \in \M^{\exp, \frac{n}{n-1}} (P)$ by Theorem \ref{Poincare type estimate}.  
\end{proof}

\subsubsection{Uniqueness for $T \ge 0$}

While $\NAmu^\lambda$ has no convexity along the linear path $(1-t) q_0 + t q_1$ in $\E^{\exp, 1} (P)$, $F_{\mathcal{P}}$ has convexity along the linear path $(1-t) u_0 + t u_1$ in $\M^{\exp, 1} (P)$, which corresponds to the `log linear exp' path $\log ((1-t) e^{q_0} + t e^{q_1})$ in $\E^{\exp, 1} (P)$. 

\begin{prop}
For $u_0, u_1 \in \M^{\exp, 1} (P)$ and $t \in [0,1]$, we put 
\[ u_t := (1-t) u_0 + t u_1. \]
Then we have $u_t \in \M^{\exp, 1} (P)$ and 
\begin{itemize}
\item For $T = 0$, $F_{\mathcal{P}} (0, u_t) = U_{\mathcal{P}} (u_t)$ is affine on $t$. 

\item For $T > 0$, $F_{\mathcal{P}} (T, u_t)$ is strictly convex when $u_0 \neq u_1$. 

\item For $T < 0$, $F_{\mathcal{P}} (T, u_t)$ is strictly concave when $u_0 \neq u_1$. 
\end{itemize}
\end{prop}

\begin{proof}
To see the log convexity of $u_t$, we consider the space $\Omega = \{ 0, 1 \}$ with the probability measure 
\[ p_t (i) = \begin{cases} 1-t & i=0 \\ t & i=1 \end{cases} \]
and we compute 
\begin{align*} 
u_t ((1-\theta) x_0 + \theta x_1) 
&= \int_\Omega u_i ((1-\theta) x_0+ \theta x_1) dp_t (i) 
\\
&\le \int_\Omega u_i^{1-\theta} (x_0) u_i^\theta (x_1) dp_t (i) 
\\
&\le \Big{(} \int_\Omega u_i (x_0) dp_t (i) \Big{)}^{1-\theta} \Big{(} \int_\Omega u_i (x_1) d p_t (i) \Big{)}^\theta
= u_t^{1-\theta} (x_0) u_t^\theta (x_1)
\end{align*}
by the H\"older inequality, for every $x_0, x_1 \in P$ and $\theta \in [0,1]$. 

Since $U_{\mathcal{P}} (u) = \int_{\partial P} u d\sigma$, we obviously have 
\begin{align}
\label{Internal energy is affine}
U_{\mathcal{P}} (u_t) = (1-t) U_{\mathcal{P}} (u_0) + t U_{\mathcal{P}} (u_1). 
\end{align}
On the other hand, since $S_{\mathcal{P}} (u) = - \int_P u \log u d\mu$, the strict convexity of $x \log x$ implies 
\begin{align} 
\label{Entropy is strictly concave}
S_{\mathcal{P}} (u_t) 
&\ge (1-t) S_{\mathcal{P}} (u_0) + t S_{\mathcal{P}} (u_1)
\end{align}
with the equality iff $u_0 = u_1$. 
This proves the claim. 
\end{proof}

Now we obtain the following uniqueness. 

\begin{thm}
\label{Main theorem on uniqueness}
For every $T > 0$, there exists a unique $u_T^{\mathrm{can}} \in \M^{\exp, \frac{n}{n-1}} (P)$ which minimizes $F_{\mathcal{P}} (T, \bullet)$, while for $T =0$, there exists a unique $u_0^{\mathrm{can}} \in \M^{\exp, \frac{n}{n-1}} (P)$ which satisfies the following 
\begin{itemize}
\item $u_0^{\mathrm{can}}$ minimizes $F_{\mathcal{P}} (0, \bullet) = U_{\mathcal{P}} (\bullet)$ and 

\item $S_{\mathcal{P}} (u_0^{\mathrm{can}}) = \max \{ S_{\mathcal{P}} (u) ~|~ U_{\mathcal{P}} (u) = \min U_{\mathcal{P}} (\bullet) \}$. 
\end{itemize}
\end{thm}

\begin{proof}
For $T > 0$, if we have two minimizers $u_0 \neq u_1 \in \M^{\exp, 1} (P)$ of $F_{\mathcal{P}} (T, \bullet)$, we get $F_{\mathcal{P}} (T, \frac{1}{2} u_0 + \frac{1}{2} u_1) < F_{\mathcal{P}} (T, u_0) = \min F_{\mathcal{P}} (T, \bullet)$ by the strict convexity, which contradicts to the assumption on $u_i$. 

Similarly, for $T = 0$, if $u_0 \neq u_1$ satisfy the above two conditions, then $\frac{1}{2} u_0 + \frac{1}{2} u_1$ satisfy the first condition by affinity of $U_{\mathcal{P}}$, while we have $S_{\mathcal{P}} (\frac{1}{2} u_0 + \frac{1}{2} u_1) > S_{\mathcal{P}} (u_0)$, which is a contradiction. 
The existence of $u_0^{\mathrm{can}}$ satisfying the two conditions is another application of our compactness: by Corollary \ref{Rellich type compactness} and the lower semi-continuity of $U_{\mathcal{P}}$, the set 
\[ \{ u \in \M^{\exp, 1} (P) ~|~ U_{\mathcal{P}} (u) = \min U_{\mathcal{P}} \} \]
is compact in $L^{1+\epsilon}$-topology, hence we can find a minimizer $u_0^{\mathrm{can}}$ of $U_{\mathcal{P}}$ which maximizes $S_{\mathcal{P}}$ among all minimizers of $U_{\mathcal{P}}$ thanks to the continuity of $S_{\mathcal{P}}$ with respect to $L^{1+\epsilon}$-topology. 
\end{proof}

Let us adopt terminologies from physics for later discussion. 

\begin{defin}[$\mu$-canonical distribution, ground states, optimizer]
We call the above $u_T^{\mathrm{can}}$ the \textit{$\mu$-canonical distribution} of temperature $T \ge 0$. 
For $T = 0$, we call minimizers of $F_{\mathcal{P}} (0, \bullet) = U_{\mathcal{P}} (\bullet)$ \textit{ground states}. 
We also call $\log u_T^{\mathrm{can}} + \text{const.} \in \E^{\exp, 1} (P)$ the \textit{optimizer} for $\NAmu^{-2\pi T}$. 
\end{defin}

The following is a natural question. 

\begin{quest}
Are all ground states are $\mu$-canonical?
\end{quest}

\subsubsection{Family over $T \in [0, \infty]$}

Let $u_T^{\mathrm{can}}$ be the $\mu$-canonical distribution for $T \ge 0$ as in Theorem \ref{Main theorem on uniqueness}. 
We firstly observe the following generality. 

\begin{prop}
Let $\mathcal{P}$ be a general system. 
The function $F_{\mathcal{P}} (T, u_T^{\mathrm{can}})$ is increasing and concave, and the functions $U_{\mathcal{P}} (u_T^{\mathrm{can}}), S_{\mathcal{P}} (u_T^{\mathrm{can}})$ are increasing on $T \ge 0$. 
Moreover, we have 
\[ \lim_{T \to \infty} S_{\mathcal{P}} (u_T^{\mathrm{can}})= 0. \]
\end{prop}

\begin{proof}
The infimum of concave functions is concave, so that \[ F_{\mathcal{P}} (T, u_T^{\mathrm{can}}) = \min_u \{ U_{\mathcal{P}} (u) - T S_{\mathcal{P}} (u) \} \] 
is concave. 
Since $- S_{\mathcal{P}} (u) \ge 0$, each $U_{\mathcal{P}} (u) - T S_{\mathcal{P}} (u)$ is increasing, so that $F_{\mathcal{P}} (T, u_T^{\mathrm{can}})$ is increasing. 

To see that $S_{\mathcal{P}} (u_T^{\mathrm{can}})$ is increasing, we compute
\begin{align*}
F_{\mathcal{P}} (T, u_T^{\mathrm{can}}) \le F_{\mathcal{P}} (T, u_{T'}^{\mathrm{can}}) 
&= F_{\mathcal{P}} (T', u_{T'}^{\mathrm{can}}) + (T' - T) S_{\mathcal{P}} (u_{T'}^{\mathrm{can}})
\\
&\le F_{\mathcal{P}} (T', u_T^{\mathrm{can}}) + (T' - T) S_{\mathcal{P}} (u_{T'}^{\mathrm{can}})
\\
&= F_{\mathcal{P}} (T, u_T^{\mathrm{can}}) + (T' -T) (S_{\mathcal{P}} (u_{T'}^{\mathrm{can}}) - S_{\mathcal{P}} (u_T^{\mathrm{can}})), 
\end{align*}
which shows 
\[ (T' -T) (S_{\mathcal{P}} (u_{T'}^{\mathrm{can}}) - S_{\mathcal{P}} (u_T^{\mathrm{can}})) \ge 0. \]

For $T' \ge T \ge 0$, we compute
\begin{align*}
    U_{\mathcal{P}} (u_T^{\mathrm{can}}) 
    &= F_{\mathcal{P}} (T, u_T^{\mathrm{can}}) + T S_{\mathcal{P}} (u_T^{\mathrm{can}}) 
    \\
    &\le F_{\mathcal{P}} (T, u_{T'}^{\mathrm{can}}) + T S_{\mathcal{P}} (u_T^{\mathrm{can}}) = U_{\mathcal{P}} (u_{T'}^{\mathrm{can}}) + T (S_{\mathcal{P}} (u_T^{\mathrm{can}}) - S_{\mathcal{P}} (u_{T'}^{\mathrm{can}}))
    \\
    &\le U_{\mathcal{P}} (u_{T'}^{\mathrm{can}}),
\end{align*} 
which shows the monotonicity for $U_{\mathcal{P}}$. 

Now we note we have the following uniform bounds:
\begin{align*}
F_{\mathcal{P}} (T, u_T^{\mathrm{can}}) 
&\le F_{\mathcal{P}} (T, 1_P) = U_{\mathcal{P}} (1_{\mathcal{P}}),
\\
S_{\mathcal{P}} (u_T^{\mathrm{can}}) 
&\le 0, 
\\
U_{\mathcal{P}} (u_T^{\mathrm{can}}) 
&= F_{\mathcal{P}} (T, u_T^{\mathrm{can}}) + T S_{\mathcal{P}} (u_T^{\mathrm{can}}) 
\\
&\le F_{\mathcal{P}} (T, 1_P)  =U_{\mathcal{P}} (1_P).
\end{align*}
It then follows by monotonicity that the limits of these exist as $T$ tends to $\infty$. 

Since we have 
\begin{align}
\label{TS estimate}
U_{\mathcal{P}} (1_P) = F_{\mathcal{P}} (T, 1_P) 
&\ge F_{\mathcal{P}} (T, u_T^{\mathrm{can}}) = U_{\mathcal{P}} (u_T^{\mathrm{can}}) - T S_{\mathcal{P}} (u_T^{\mathrm{can}})
\\ \notag
&\ge U_{\mathcal{P}} (u_0^{\mathrm{can}}) - T S_{\mathcal{P}} (u_T^{\mathrm{can}}), 
\end{align}
we get 
\[ 0 \ge T S_{\mathcal{P}} (u_T^{\mathrm{can}}) \ge U_{\mathcal{P}} (u_0) - U_{\mathcal{P}} (1_P) \]
for every $T \ge 0$. 
It follows that 
\[ \lim_{T \to \infty} S_{\mathcal{P}} (u_T^{\mathrm{can}}) = 0. \]
\end{proof}

Using the compactness we established, we further obtain the following. 

\begin{thm}
\label{continuous family of canonical distributions}
Let $\mathcal{P}$ be a system. 
The map 
\[ [0,\infty] \to \M^{\exp, \frac{n}{n-1}} (P): T \mapsto \begin{cases} u_T^{\mathrm{can}} & T \in [0, \infty) \\ 1_P & T = \infty \end{cases} \]
is continuous with respect to $L^p$-topology for every $p \in [1, \frac{n}{n-1})$ with continuous $F_{\mathcal{P}} (T, u_T^{\mathrm{can}}), U_{\mathcal{P}} (u_T^{\mathrm{can}}), S_{\mathcal{P}} (u_T^{\mathrm{can}})$. 
Furthermore, we have 
\[ \lim_{T \to \infty} F_{\mathcal{P}} (T, u_T^{\mathrm{can}}) = \lim_{T \to \infty} U_{\mathcal{P}} (u_T^{\mathrm{can}}) = U_{\mathcal{P}} (1_P), \]
\[ \lim_{T \to \infty} T S_{\mathcal{P}} (u_T^{\mathrm{can}}) = 0. \]
\end{thm}

\begin{proof}
As we see in the above proof, we have $U_{\mathcal{P}} (u_T^{\mathrm{can}}) \le U_{\mathcal{P}} (1_P)$, so by Corollary \ref{Rellich type compactness} the family $\{ u_T^{\mathrm{can}} \}_{T \in [0, \infty)}$ is relatively compact in $L^p$-topology for $p \in [1, \frac{n}{n-1})$. 

Take a convergent sequence $\infty \neq T_i \to T_\infty \in [0, \infty]$ and a subsequence $T_{i (j)}$ so that $u_{T_{i (j)}}^{\mathrm{can}} \to u$ in $L^p$-topology to some $u \in \M^{\exp, 1} (P)$. 
We show the limit is $u_T^{\mathrm{can}}$ for $T \in [0, \infty)$ and $1_P$ for $T = \infty$, independent of the choice of subsequence. 

Assume $T_\infty \in [0, \infty)$. 
By the lower semi-continuity, we get 
\[ \liminf_{j \to \infty} F_{\mathcal{P}} (T_{i (j)}, u_{T_{i (j)}}^{\mathrm{can}}) = \liminf_{j \to \infty} F_{\mathcal{P}} (T_\infty, u_{T_{i (j)}}^{\mathrm{can}}) \ge F_{\mathcal{P}} (T_\infty, u). \]
For any $u' \in \M^{\exp, 1} (P)$, we compute 
\[ F_{\mathcal{P}} (T_\infty, u') = \lim_{j \to \infty} F_{\mathcal{P}} (T_{i (j)}, u') \ge \liminf_{j \to \infty} F_{\mathcal{P}} (T_{i (j)}, u_{T_{i (j)}}^{\mathrm{can}}) \ge F_{\mathcal{P}} (T_\infty, u),  \]
which shows the limit $u$ is a minimizer of $F_{\mathcal{P}} (T_\infty, \bullet)$. 
This shows $u= u_{T_\infty}$ for $T_\infty > 0$ by the uniqueness of minimizer. 

Suppose $T_\infty = 0$ and $u \neq u_0$. 
Then we have $S_{\mathcal{P}} (u_{T_{i (j)}}^{\mathrm{can}}) \to S_{\mathcal{P}} (u) < S_{\mathcal{P}} (u_0^{\mathrm{can}})$ by the uniqueness of $\mu$-canonical distribution. 
Take large $j_0$ so that $S_{\mathcal{P}} (u_{T_{i (j)}}^{\mathrm{can}}) < S_{\mathcal{P}} (u_0^{\mathrm{can}})$ for $j \ge j_0$, which contradicts to the fact that $S_{\mathcal{P}} (u_T^{\mathrm{can}})$ is increasing. 
Thus we have $u = u_0^{\mathrm{can}}$ when $T_\infty =0$. 

Suppose $T_\infty = \infty$. 
By the above proposition, we have 
\[ S_{\mathcal{P}} (u) = \lim S_{\mathcal{P}} (u_{T_{i (j)}}^{\mathrm{can}}) = 0. \]
This implies $u = 1_P$ as $0$ is the maximum value of $S_{\mathcal{P}}$ attained only at $1_P$. 
Thus we proved the continuity of the family $u_T^{\mathrm{can}}$ on $[0, \infty]$. 

Now it follows that 
\[ U_{\mathcal{P}} (1_P) \ge \limsup_{T \to \infty} U_{\mathcal{P}} (u_T^{\mathrm{can}}) \ge \liminf_{T \to \infty} U_{\mathcal{P}} (u_T^{\mathrm{can}}) \ge U_{\mathcal{P}} (1_P). \]
On the other hand, by (\ref{TS estimate}), we have 
\[ 0 \ge T S_{\mathcal{P}} (u_T^{\mathrm{can}}) \ge U_{\mathcal{P}} (u_T^{\mathrm{can}}) - U_{\mathcal{P}} (1_P) \searrow 0, \]
which shows 
\[ \lim_{T \to \infty} T S_{\mathcal{P}} (u_T^{\mathrm{can}}) = 0. \]
Putting these together, we obtain 
\[ \lim_{T \to \infty} F_{\mathcal{P}} (T, u_T^{\mathrm{can}}) = 0. \]

Since $F_{\mathcal{P}} (T, u_T^{\mathrm{can}})$ is concave, it is continuous. 
Since $S_{\mathcal{P}}$ is continuous with respect to $L^{1+ \epsilon}$-topology, $S_{\mathcal{P}} (u_T^{\mathrm{can}})$ is continuous. 
Finally since $U_{\mathcal{P}} (u_T^{\mathrm{can}}) = F_{\mathcal{P}} (T, u_T^{\mathrm{can}}) + T S_{\mathcal{P}} (u_T^{\mathrm{can}})$, $U_{\mathcal{P}} (u_T^{\mathrm{can}})$ is continuous on $[0, \infty)$, while the continuity at $T = \infty$ is already proved. 
\end{proof}

As a sophisticated version of the extremal limit observation in \cite{Ino2, Ino4}, we speculate the optimal destabilizer for normalized Donaldson--Futaki invariant appears in the rescaled limit. 

\begin{conj}
As $T \to \infty$, the convex function $q_T := T \log u_T^{\mathrm{can}}$ converges at least in $L^2$-topology to an lsc convex function $q_{\mathrm{ext}} \in L^2 (P)$ characterized as follows: 
$q_{\mathrm{ext}}$ minimizes the following normalized Donaldson--Futaki invariant among all $L^2$-integrable convex functions on $P$: 
\[ \frac{\mathrm{DF} (q)}{\| \hat{q} \|_{L^2}} = \frac{2\pi \int_{\partial P} q d\sigma + \bar{s} \int_P q d\mu}{(\int_P (q - \bar{q})^2 d\mu)^{1/2}}, \]
where $\bar{q} = \int_P q d\mu/\int_P d\mu$. 
\end{conj}

The existence of $q_{\mathrm{ext}}$ is proved in \cite{Sze}. 
For a non-negative optimizer $q \ge 0$ of $\NAmu^\lambda$, we have $\frac{1}{p!} q^p \le e^q$ for every $p \in \mathbb{N}$, so that $q$ is $L^p$ for every $p \in [1, \infty)$. 
This regularity is much better than $L^2$-regularity of $q_{\mathrm{ext}}$ proved in \cite{Sze}. 

\subsection{Consequences on $\mu$-cscK metric and $\mu$K-stability}

\subsubsection{Uniqueness of $\mu$-cscK metrics on toric manifolds for $\lambda \le 0$}

Now we see $\mu^\lambda$-cscK metric determines optimizer. 
Note if we have a $\mu^\lambda$-cscK metric $\omega$ on toric manifold $(X, L) \circlearrowleft T$, we can take $g \in \mathrm{Aut} (X, L)$ so that $g^* \omega$ is a $\mu^\lambda_\xi$-cscK metric with $\xi \in \mathrm{Lie} (T_{\mathrm{cpt}})$. 

\begin{thm}
\label{mu-cscK implies optimal destabilization}
Assume $\lambda \le 0$. 
If a toric manifold $(X, L)$ admits a $\mu^\lambda_\xi$-cscK metric with $\xi \in \mathrm{Lie} (T_{\mathrm{cpt}})$, then the linear map $\ket{\xi}: P \to \mathbb{R}: \mu \mapsto \langle \mu, \xi \rangle$ is the optimizer of $\NAmu^\lambda$. 
Here the metric is not necessarily a priori $T_{\mathrm{cpt}}$-invariant. 
\end{thm}

\begin{proof}
Suppose we have a $\mu^\lambda_\xi$-cscK metric $\omega$ with $\xi \in \mathrm{Lie} (T_{\mathrm{cpt}})$. 
It is proved in \cite{Ino4} (see also \cite{Ino5}) that $\ket{\xi}$ is a maximizer of $\NAmu^\lambda$ for $\lambda \le 0$.
Indeed, we have 
\[ \NAmu^\lambda (\ket{\xi}) \le \max_q \NAmu^\lambda (q) \le \inf_{\omega_\phi} \bm{\mu}_{\mathrm{Per}}^\lambda (\omega_\phi) \le \bm{\mu}_{\mathrm{Per}}^\lambda (\omega), \]
and for $\lambda \le 0$ we have 
\[ \bm{\mu}_{\mathrm{Per}}^\lambda (\omega) = \NAmu^\lambda (\ket{\xi}). \]
This proves the claim for $\lambda < 0$. 

To see that $\ket{\xi}$ is the optimizer for $\lambda = 0$, we perturb the $\mu^0_\xi$-cscK metric $\omega$ to $\mu^\lambda_{\xi_\lambda}$-cscK metrics $\omega_\lambda$ for $\lambda \in (-\epsilon, 0)$ as in the construction in \cite{Ino2}, which is just an application of implicit function theorem. 
We obviously have $\xi_\lambda \to \xi$ as $\lambda \to 0$ by the construction. 
We already know $\ket{\xi_\lambda}$ are the optimizers. 
By Theorem \ref{continuous family of canonical distributions}, we conclude $\ket{\xi} = \lim_{\lambda \to 0} \ket{\xi_\lambda}$ is also the optimizer. 
\end{proof}

By the uniqueness of optimizers of $\NAmu^\lambda$ for $\lambda \le 0$, we conclude the vectors $\xi, \xi'$ associated to $\mu^\lambda$-cscK metrics $\omega, \omega'$ are conjugate. 
Then thanks to the result \cite{Lah2} on the uniqueness of $\mu^\lambda_\xi$-cscK metric for fixed $\xi$, we obtain Theorem \ref{Uniqueness of mu-cscK metric}. 
Here we mention it again. 

\begin{cor}
Assume $\lambda \le 0$. 
On a toric manifold,  $\mu^\lambda$-cscK metrics are unique modulo the action of automorphism group. 
\end{cor}

\subsubsection{Toric $\mu$K-semistability is $\mu$-entropy maximization}

\begin{thm}
\label{muK-semistability is equivalent to mu-entropy maximization}
A toric variety $(X, L)$ is toric $\mu^\lambda_\xi$K-semistable for $\lambda \le 0$ if and only if the linear map $\ket{\xi}: P \to \mathbb{R}$ maximizes $\NAmu^\lambda$. 
\end{thm}

\begin{proof}
Thanks to (\ref{NAmu and muFutaki}), $(X, L)$ is $\mu^\lambda_\xi$K-semistable if $\ket{\xi}$ maximizes $\NAmu^\lambda$. 

Suppose $(X, L)$ is $\mu^\lambda_\xi$K-semistable for $\lambda \le 0$. 
Take a continuous convex function $q: P \to \mathbb{R}$. 
For the log linear exp path 
\[ q_t := \log ((1-t) e^{\ket{\xi}} + t \frac{\int_P e^{\ket{\xi}} d\mu}{\int_P e^q d\mu} e^q) \in \E^{\exp, 1} (P), \]
we compute 
\[ \frac{d}{dt} q_t = \frac{\frac{\int_P e^{\ket{\xi}} d\mu}{\int_P e^q d\mu} e^q - e^{\ket{\xi}}}{(1-t) e^{\ket{\xi}} + t \frac{\int_P e^{\ket{\xi}} d\mu}{\int_P e^q d\mu} e^q}. \]
Since $\frac{d}{dt} q_\bullet: P \times [0,1] \to \mathbb{R}$ is a continuous function, its absolute value is bounded from above by a uniform constant, so that we can compute
\begin{align*} 
\frac{d}{dt}\Big{|}_{t=0} \NAmu^\lambda (q_t) 
&= - \mathrm{Fut}^\lambda_\xi (\frac{d}{dt}\Big{|}_{t=0} q_t)
= -\frac{\int_P e^{\ket{\xi}} d\mu}{\int_P e^q d\mu} \mathrm{Fut}^\lambda_\xi (e^{q - \ket{\xi}}). 
\end{align*}
By Fenchel--Moreau theorem, we can take an increasing sequence $0 \le u_i \in \mathcal{H}_{\mathrm{NA}} (P)$ so that it converges to $e^{q - \ket{\xi}}$ pointwiesly. 
By monotone convergence theorem, we get 
\[ 0 \le \frac{\int_P e^{\ket{\xi}} d\mu}{\int_P u_i e^{\ket{\xi}} d\mu} \mathrm{Fut}^\lambda_\xi (u_i) \to \frac{\int_P e^{\ket{\xi}} d\mu}{\int_P e^q d\mu} \mathrm{Fut}^\lambda_\xi (e^{q - \ket{\xi}}). \]
Therefore, the $\mu^\lambda_\xi$K-semistability implies 
\[ \frac{d}{dt}\Big{|}_{t=0} \NAmu^\lambda (q_t) \le 0. \]
Since $\NAmu^\lambda$ is concave, we get $\NAmu^\lambda (q) \le \NAmu^\lambda (\ket{\xi})$ for continuous $q$. 

For general $q \in \E^{\exp, 1} (P)$, if $\int_{\partial P} e^q d\sigma = \infty$, we obviously have $-\infty = \NAmu^\lambda (q) \le \NAmu^\lambda (\ket{\xi})$. 
If $\int_{\partial P} e^q d\sigma < \infty$, take an increasing sequence $q_i \in \mathcal{H}_{\mathrm{NA}} (P)$ converging to $q$ pointwisely. 
Then by the monotone convergence theorem we get $\NAmu^\lambda (q_i) \to \NAmu^\lambda (q)$. 
This shows $\NAmu^\lambda (q) \le \NAmu^\lambda (\ket{\xi})$ for general $q \in \E^{\exp, 1} (P)$. 
\end{proof}

In particular for $\lambda < 0$, a toric variety $(X, L)$ can be $\mu^\lambda_\xi$K-semistable at most one $\xi \in \mathfrak{t}$. 
The following remark further implies a similar conclusion for $\lambda = 0$ under $\mu^\lambda_\xi$K-polystability assumption. 

\begin{rem}
\label{Fut of exp q is the difference of NAmu}
For $\lambda = 0$, we can also compute directly 
\[ \mathrm{Fut}_\xi (e^{q - \ket{\xi}}) = \frac{\int_P e^q d\mu}{\int_P e^{\ket{\xi}}d\mu} (\NAmu (\ket{\xi}) - \NAmu (q)) \]
for $q \in \E^{\exp, 1} (P)$ with $\int_{\partial P} e^q d\sigma < \infty$. 
\end{rem}

\begin{cor}
For $\lambda = 0$, if a toric variety $(X, L)$ is toric $\mu_\xi$K-polystable with respect to $\E^{\frac{n}{n-1}} (P) := \mathrm{Conv} (P) \cap L^{\frac{n}{n-1}} (P)$, then ground states are unique and hence $\ket{\xi}$ is the optimizer for $\NAmu$. 
\end{cor}

\begin{proof}
Suppose $\NAmu (q) = \max \NAmu = \NAmu (\ket{\xi})$ for $q \in \E^{\exp, 1} (P)$. 
Since $\int_{\partial P} e^q d\sigma < \infty$, we have $e^{q - \ket{\xi}} \in \mathrm{Conv} (P) \cap L^{\frac{n}{n-1}} (P)$. 
By the above remark, we have $\mathrm{Fut}_\xi (e^{q - \ket{\xi}}) = 0$, which implies $q = \ket{\xi}$ by our $\mu_\xi$K-polystability assumption. 
\end{proof}

\subsection{Ground state in dimension 2 is bounded}

As we have noted in Remark \ref{piecewise affine}, a rational piecewise affine convex function on toric polytope realizes an algebro-geometric degeneration of toric variety: $X \leadsto \mathcal{X}_0$. 
In comparison with \cite{BLXZ, Yao}, the best regularity we can expect for maximizer of $\NAmu^\lambda$ would be piecewise affinity. 
As noted in \cite{Ino1, BLXZ} in the context of K\"ahler--Ricci soliton, this conjecture is deeply related to moduli theory for polarized varieties. 
To appreciate its extreme importance, we refer to this conjecture as \textit{crystal conjecture}. 

\begin{conj}[$\mu$-entropy maximizer is crystalline]
\label{regularity conjecture}

Let $\mathcal{P}$ be a system. 
Then for each $\lambda \in \mathbb{R}$, every maximizer $q$ of $\NAmu^\lambda$ is piecewise affine. 
\end{conj}

\begin{rem}
\label{remark on finite generation}
For a rational piecewise affine function $q$, the base polytope of an affine component of $q$ (the image of a facet of $Q \subset \mathfrak{t}^\vee \times \mathbb{R}$ contained in the roof $\{ (x, -q(x)) ~|~ x \in P \} \subset Q$ along the projection $Q \to P$) is rational and realizes an irreducible component of the central fibre $\mathcal{X}_0$ of the associated test configuration. 

On the other hand, for an non-rational piecewise affine function $q$, the base polytope of an affine component of $q$ might be not rational. 
(In this case, the filtration on the graded ring $R (X, L) = \bigoplus_m H^0 (X, L^{\otimes m})$ associated to $q$ as in \cite{Ino5} would be not finitely generated. )
This implicates the category of algebraic scheme might be insufficient to realize $q$ in a geometric way. 
The author speculates there would be a nice category of spaces extending that of algebraic schemes in which we can realize optimizer of $\NAmu^\lambda$ in a geometric way. 
We do not pursue this further in this article, but we just refer to \cite{Joy} as a hint for ``extended toric geometry'' on non-rational polytope/fan. 
\end{rem}

We have proved in the existence theorem that the exponential $e^q$ of a maximizer $q$ is $L^{\frac{n}{n-1}}$-intergable. 
To approach crystal conjecture, we confirm the continuity of maximizer for $\lambda = 0$ in dimension 2. 

\begin{thm}
\label{Main theorem on regularity}
Let $\mathcal{P}$ be a $2$-dimensional system. 
Then every maximizer $q \in \E^{\exp, 1} (P)$ of $\NAmu$ is bounded and continuous. 
\end{thm}

We firstly observe every maximizer $q$ of $\NAmu$ is equal to
\begin{align}
\label{tight convex function}
\tilde{q} := \sup \{ \ell: \text{affine} ~|~ \ell|_{\partial P} \le q|_{\partial P} \}.
\end{align} 
Indeed, recall firstly 
\[ q = \hat{q} := \sup \{ \ell: \text{affine} ~|~ \ell \le q \} \]
for any lsc convex function $q$ by Fenchel--Moreau theorem. 
Since we obviously have $\hat{q} \le \tilde{q}$, we get 
\[ q|_{\partial P} = \hat{q}|_{\partial P} \le \tilde{q}|_{\partial P} \le q|_{\partial P} \]
and hence $\tilde{q}|_{\partial P} = q|_{\partial P}$ for any lsc convex function $q$. 
If $q \neq \tilde{q}$, we have $q \le \tilde{q}$ and $q|_B \le \tilde{q}|_B -\varepsilon$ for small $\varepsilon > 0$ on a small ball $B \subset P^\circ$. 
Then we compute
\[ -\frac{1}{2\pi} \NAmu (q) = \frac{\int_{\partial P} e^q d\sigma}{\int_P e^q d\mu} \ge \frac{\int_{\partial P} e^{\tilde{q}} d\sigma}{\int_{P \setminus B} e^{\tilde{q}} d\mu + \int_B e^{\tilde{q} - \varepsilon} d\mu} > \frac{\int_{\partial P} e^{\tilde{q}} d\sigma}{\int_P e^{\tilde{q}} d\mu} = - \frac{1}{2\pi} \NAmu (\tilde{q}), \]
so that $q$ does not maximize $\NAmu$. 

In what follows, we consider convex function $q \not\equiv \infty$ of the form $q = \tilde{q}$. 
Take a vertex $v \in P$. 
For $h \in \mathbb{R}$, we consider 
\begin{align}
q_{v, h} := \sup \{ \ell :\text{affine} ~|~ \ell|_{\partial P} \le q|_{\partial P}, \ell (v) \le h \} \le q. 
\end{align}
We will show if $q (v) = \infty$ then we can find large $h$ and small $\theta \in (0,1)$ so that $\NAmu (q) < \NAmu (q_{v, h, \theta})$ for $q_{v, h, \theta} := (1-\theta) q + \theta q_{v,h}$. 
For the computation, we observe a concrete description of $q_{v,h}$. 

\begin{lem}
Suppose $q = \tilde{q}$ as above. 
Let $E^0 = \overline{vv^0}, E^1 = \overline{vv^1} \subset \partial P$ be the edges (dimension one face) containing $v$. 
Assume for each $i = 0,1$ the right differential $\partial_{t+} q ((1-t)v +t v^i)$ diverges to $-\infty$ as $t \to 0$, which is the case for instance when $q (v) =\infty$. 

Then for $h < q (v)$ sufficiently colse to $q (v)$, there exists a unique affine function $\ell_h$ on $P$ such that $\ell_h (v) = h$, $\ell_h \le q$ on $P$ and $\ell_h (p^i) = q (p^i)$ for some $p^i \in E^i \setminus v$ for each $i=0,1$. 
Moreover, we have $\ell \le \ell_h$ for any affine function $\ell$ on $P$ satisfying $\ell (v) = h$ and $\ell \le q$. 
\end{lem}

\begin{proof}
Observe for each $i$, there exists a unique affine function $\ell^i_h$ on the edge $E^i$ such that $\ell^i_h (v) =h$, $\ell^i_h \le q|_{E^i}$ and $\ell^i_h (p^i) = q (p^i)$ for some point $p^i \in E^i \setminus v$. 
Thanks to $\ell^i_h ((1-t) v +t p^i) = (1-t) h+ t q (p^i)$, we have ${\ell^i}' \le \ell^i_h$ for any other affine function ${\ell^i}'$ on $E^i$ satisfying ${\ell^i}' (v) = h$ and ${\ell^i}' \le q|_{E^i}$.

Let $\ell_h$ be the affine function whose graph represents the plane spanned by the graphs of $\ell^0_h$ and $\ell^1_h$. 
We clearly have $\ell_h (v) = h$. 
Let $Q \in \mathbb{R}$ be the minimum of $q|_{\partial P}$. 
By our assumption on $q$, we have $q (v) > Q$. 
Moreover, for $t_h \in [0,1]$ with $p^i_h = (1-t_h) v + t_h v^i$, we have 
\[ \partial_{t+} \ell^i_h ((1-t)v + t v^i) \le \partial_{t+} q ((1-t_h) v+t_h v^i), \]
so that taking $h \ge Q$ sufficiently close to $q (v)$, we may assume $\ell^i_h (v^i) \le Q$. 
(Compare the subsequent lemma. 
It implies $t_h \to 0$ under our assumption, so that $\partial_{t+} q ((1-t_h) v+t_h v^i) \to -\infty$ as $h \to q (v)$. )
This implies $\ell_h|_{\partial P \setminus (E^0 \cup E^1)} \le Q$, so that we get $\ell_h|_{\partial P} \le q|_{\partial P}$. 
It follows that $\ell_h \le \tilde{q} = q$. 

Let $\ell$ be an affine function on $P$ satisfying $\ell (v) = h$ and $\ell \le q$, which in particular implies $\ell|_{E^i} \le q|_{E^i}$. 
By the above remark, we have $\ell|_{E^i} \le \ell^i_h$, hence $\ell \le \ell_h$ on $P$. 
\end{proof}

Since $F^i := \{ x \in E^i \setminus v ~|~ (q -\ell_h) (x) = 0 \}$ is closed, we have a unique point $p^i_h \in E^i \setminus v$ which is closest to $v$ in $F^i$. 
For this point, we have the following. 

\begin{lem}
Under the assumption in the above lemma, we have $p^i_h \to v$ monotonically as $h \to q (v)$. 
\end{lem}

\begin{proof}
This is visually clear, but we show it logically in order to check how our assumption on $q$ works. 
For $h \le h'$, we have 
\[ \ell_h (v) =h \le h' = \ell_{h'} (v),~~ \ell_h (p^i_{h'}) \le q (p^i_{h'}) = \ell_{h'} (p^i_{h'}), ~~ \ell_h (p^i_h) = q (p^i_h) \ge \ell_{h'} (p^i_h), \]
which implies the length of the segment $\overline{v p^i_{h'}}$ is shorter than $\overline{v p^i_h}$. 

By the monotonicity, we have a limit point $p^i_{q(v)} := \lim_{h \to q (v)} p^i_h \in E^i$. 
Suppose $p^i_{q (v)} \neq v$. 
Then taking the limit $h \to \infty$ of
\[ (1-t) h + t q (p^i_h) = \ell^i_h ((1-t)v + t p^i_h) \le q ((1-t)v + t p^i_h), \]
we get 
\[ (1-t) q (v) + t q (p^i_{q (v)}) \le q ((1-t)v + t p^i_{q(v)}), \]
while 
\[ q ((1-t)v + t p^i_{q(v)}) \le (1-t) q (v) + t q (p^i_{q (v)}) \]
by convexity. 
This implies $q$ is affine on the segment $\overline{v p^i_{q(v)}}$, which contradicts to our assumption.  
\end{proof}

\begin{lem}
Under the above assumption, we have 
\[ q_{v,h} (x) = 
\begin{cases} 
\ell_h & \text{ on the triangle } \triangle v p^0_h p^1_h 
\\ 
q & \text{ on } P \setminus \triangle v p^0_h p^1_h 
\end{cases}. \]
\end{lem}

\begin{proof}
For $x \in P^\circ \setminus \triangle v p^0_h p^1_h$, take an affine function $\ell_x$ so that $\ell_x \le q$ on $P$ and $\ell_x (x) = q (x)$. 
The two segments $\overline{p^0_h p^1_h}$ and $\overline{v x}$ crosses at a point $p^t_h = (1-t) p^0_h + t p^1_h \in \overline{p^0_h p^1_h}$. 
Since $\ell_x (p^t_h) \le q (p^t_h) = \ell_h (p^t_h)$ and $\ell_x (x) = q (x) \ge \ell_h (x)$, we have $\ell_x (v) \le \ell_h (v) = h$. 
Then by 
\[ q(x) = \ell_x (x) \le q_{v,h} (x) \le q (x), \]
we get $q_{v,h} (x) = q(x)$. 
By the lower semi-continuity, we have 
\[ q_{v,h}|_{P \setminus \triangle vp^0_h p^1_h} = q|_{P \setminus \triangle vp^0_h p^1_h}. \]

Take an affine function $\ell$ on $P$ so that $\ell \le q$ and $\ell (v) \le h$. 
For $x \in \triangle v p^0_h p^1_h \setminus v$, let $p^t_h \in \overline{p^0_h p^1_h}$ be the intersection point with the line spanned by the segment $\overline{vx}$. 
Since $\ell (v) \le h = \ell_h (v)$ and $\ell (p^t_h) \le q (p^t_h) = \tilde{q} (p^t_h) = \ell_h (p^t_h)$, 
we have $\ell (x) \le \ell_h (x)$. 
It follows that $q_h (x) \le \ell_h (x)$. 
On the other hand, we have $q_h \ge \ell_h$ on $P$ as $\ell_h (v) = h$ and $\ell_h|_{\partial P} \le q|_{\partial P}$. 
On the other hand, we obviously have $q_h (v) = h$. 
Thus we get $q_h (x) = \ell_h (x)$ for $x \in \vartriangle vp^0_h p^1_h$. 
\end{proof}

Now we show the continuity. 
By \cite{GKR} (or by an extension of the argument of (\ref{upper semi-continuity})), every bounded convex function on a convex polytope is known to be upper semi-continuous. 
Since we assume the lower semi-continuity for $q \in \E^{\exp, 1} (P)$, it suffices to show the boundedness. 

\begin{proof}[Proof of Theorem \ref{Main theorem on regularity}]
Let $q$ be a maximizer of $\NAmu$. 
We may normalize $q$ so that $q \ge 0$. 
Suppose $q$ is unbounded, then we can find a vertex $v \in P$ with $q (v) = \infty$ by an easy argument. 
For sufficiently large $h \ge 0$, we take $q_{v, h}, \ell_h$ and $p^0_h, p^1_h$ as above. 
We note 
\[ \mu (\triangle vp^0_h p^1_h), \sigma (\overline{vp^0_h}), \sigma (\overline{vp^1_h}) \to 0 \] 
as $h \to \infty$, while 
\[ \frac{\mu (\triangle vp^0_h p^1_h)}{\sigma (\overline{vp^0_h}) \sigma (\overline{vp^1_h})} \] 
is constant. 

We put
\[ u_h := q - q_{v, h} = \begin{cases} q- \ell_h & \text{ on the triangle } \triangle v p^0_h p^1_h \\ 0 & \text{ on } P \setminus \triangle v p^0_h p^1_h \end{cases}. \]
We consider a map from $[0,1]^2$ to $\triangle v p_h^0 p_h^1$ given by 
\[ (s; r) \mapsto v+ (1-s) r \overrightarrow{v p^0_h} + sr \overrightarrow{vp^1_h} \in \triangle v p^0_h p^1_h. \]
The measure $d\mu|_{\triangle v p^0_h p^1_h}$ is transformed to $2 \mu (\triangle v p^0_h p^1_h) r dr ds$ on $[0,1]^2$ and the measure $d\sigma|_{\overline{vp^i_h}}$ is transformed to $\sigma (\overline{vp^i_h}) dr$ on $\{ i \} \times [0,1]$. 

Now we compute
\begin{align*}
\int_P u_h e^q d\mu 
&\le \int_{\triangle v p^0_h p^1_h} u_h e^q d\mu 
\\
&= 2 \mu (\triangle v p^0_h p^1_h) \int_0^1 dr \cdot r \int_0^1 u_h (s;r) e^{q (s; r)} ds
\\
&\le 2 \mu (\triangle v p^0_h p^1_h) \int_0^1 dr \cdot r e^{\max_{s \in [0,1]} \ell_h (s; r)} \int_0^1 u_h (s; r) e^{u_h (s; r)} ds
\\
&\le 2 \mu (\triangle v p^0_h p^1_h) \int_0^1 dr \cdot r e^{\max \{ \ell_h (0; r), \ell_h (1; r) \}} \int_0^1 \{ (1-s) u_h (0; r) e^{u_h (0;r)} + s u_h (1;r) e^{u_h (1;r)} \} ds
\\
&\le \mu (\triangle v p^0_h p^1_h) \int_0^1 dr \cdot r e^{\max \{ \ell_h (0;r ), \ell_h (1;r ) \}} \{ u_h (0;r ) e^{u_h (0;r )} + u_h (1;r ) e^{u_h (1;r )} \} .
\end{align*}
Assume $\max \{ \ell_h (0; 1), \ell_h (1; 1) \} = \ell_h (0;1 )$, we have $\max \{ \ell_h (0; r), \ell_h (1; r) \} = \ell_h (0;r )$ for every $r$, then we further compute 
\begin{align*} 
\int_P u_h e^q d\mu 
&\le \mu (\triangle v p^0_h p^1_h) \int_0^1 dr \cdot r e^{\ell_h (0;r )} \{ u_h (0;r ) e^{u_h (0;r )} + u_h (1;r ) e^{u_h (1;r )} \} 
\\
&\le \mu (\triangle v p^0_h p^1_h) \int_0^1 dr \cdot \{ u_h (0;r ) e^{q (0;r )} + r e^{\ell_h (0; r)} u_h (1; r) e^{q (1;r )} \}
\\
&\le \mu (\triangle v p^0_h p^1_h) \{ \int_0^1 u_h (0; r) e^{q (0;r)} dr + \int_0^1 e^{q (0;r)} dr \cdot \int_0^1 u_h (1;r ) e^{q (1;r)} dr \}
\\
& \le \frac{\mu (\triangle v p^0_h p^1_h)}{\sigma (\overline{vp^0_h})} \int_0^1 u_h (0; r) e^{q (0;r)} \sigma (\overline{vp^0_h}) dr 
\\
&\qquad + \frac{\mu (\triangle v p^0_h p^1_h)}{\sigma (\overline{vp^0_h}) \sigma (\overline{vp^1_h})} \int_0^1 e^{q (0;r)} \sigma (\overline{vp^0_h}) dr \cdot \int_0^1 u_h (1;r ) e^{q (1;r )} \sigma (\overline{vp^1_h}) dr 
\\
&= \frac{\mu (\triangle v p^0_h p^1_h)}{\sigma (\overline{vp^0_h})} \int_{E^0} u_h e^q d\sigma|_{E^0} 
\\
&\qquad + \frac{\mu (\triangle v p^0_h p^1_h)}{\sigma (\overline{vp^0_h}) \sigma (\overline{vp^1_h})} \int_{\overline{vp^0_h}} e^q d \sigma|_{\overline{vp^0_h}} \cdot \int_{E^1} u_h e^q d\sigma|_{E^1}, 
\end{align*}
where the third inequality follows by 
\[ re^{\ell_h (0; r)} \le \int_0^1 e^{\ell_h (0; r)} dr \le \int_0^1 e^{q (0; r)} dr. \]

Fix $h$ large so that 
\[ \frac{\mu (\triangle v p^0_h p^1_h)}{\sigma (\overline{vp^0_h})}, \frac{\mu (\triangle v p^0_h p^1_h)}{\sigma (\overline{vp^0_h}) \sigma (\overline{vp^1_h})} \int_{\overline{vp^0_h}} e^q d \sigma|_{\overline{vp^0_h}} < \Big{(} \frac{\int_{\partial P} e^q d\mu}{\int_P e^q d\sigma} \Big{)}^{-1}. \]
Then for $q_\theta := (1-\theta) q + \theta q_h$, we compute 
\begin{align*}
\frac{d}{d\theta}\Big{|}_{\theta = 0} \frac{\int_{\partial P} e^{q_\theta} d\sigma}{\int_P e^{q_\theta} d\mu}
&= \frac{d}{d\theta}\Big{|}_{\theta = 0} \frac{\int_{\partial P} e^{q - \theta u_h} d\sigma}{\int_P e^{q- \theta u_h} d\mu}
\\
&= \frac{1}{\int_P e^q d\mu} \Big{(} - \int_{\partial P} u_h e^q d\sigma + \frac{\int_{\partial P} e^q d\sigma}{\int_P e^q d\mu} \int_P u_h e^q d\mu \Big{)}
\\
&< \frac{1}{\int_P e^q d\mu} \Big{(} - \int_{\partial P} u_h e^q d\sigma + \int_{E^0} u_h e^q d\sigma|_{E^0} + \int_{E^1} u_h e^q d\sigma|_{E^1} \Big{)}
\\
&= 0, 
\end{align*}
which shows 
\[ \frac{\int_{\partial P} e^{q_\theta} d\sigma}{\int_P e^{q_\theta} d\mu} < \frac{\int_{\partial P} e^q d\sigma}{\int_P e^q d\mu} \]
for sufficiently small $\theta$. 
This contradicts to the assumption that $q$ is a maximizer of $\NAmu$. 
The case $\max \{ \ell_h (0; 1), \ell_h (1; 1) \} = \ell_h (1;1 )$ is similar. 
Thus $q$ must be bounded. 

\end{proof}

\section{Thermodynamical structure}
\label{Thermodynamical structure}

To appreciate the results of this section, let us briefly recall how the parameter $\lambda$ appeared in the beginning \cite{Ino2} of $\mu$-cscK metric. 
The notion of $\mu$-cscK metric was introduced in \cite{Ino2} based on the moment map picture on K\"ahler--Ricci soliton observed in \cite{Ino1}. 
Given a symplectic manifold $(M, \omega)$ with a Hamiltonian vector field $\xi$, the space of $\xi$-invariant almost complex structures $\mathcal{J}_\xi (M, \omega)$ possesses a symplectic structure $\Omega_\xi$ and a moment map $\mathcal{S}_\xi: \mathcal{J}_\xi (M, \omega) \to \mathfrak{ham}_\xi (M, \omega)^\vee$ with respect to the action of $\xi$-compatible Hamiltonian diffeomorphism group $\mathrm{Ham}_\xi (M,\omega)$. 
These are dependent on $\xi$. 
The Hamiltonian potential $\mu_\xi: M \to \mathbb{R}$ of $\xi$ defines an element $\bra{\mu_\xi}$ of $\mathfrak{ham}_\xi (M, \omega)^\vee$ fixed by the coadjoint action, so that $\mathcal{S}_\xi + \lambda \bra{\mu_\xi}$ gives another moment map for the same symplectic structure $\Omega_\xi$. 
If we choose $[\omega] = -T^{-1} c_1 (M, \omega)$ and $\lambda = -2\pi T$ ($T < 0$), the symplectic reduction $(S_\xi + \lambda \bra{\mu_\xi})^{-1} (0)/\mathrm{Ham}_\xi (M, \omega)$ can be interpreted as the moduli space of K\"ahler--Ricci soliton (cf. \cite{Ino1}). 

At this point, there are at least two points of view to extend the theory of K\"ahler--Ricci soliton to general polarized manifold: 
\begin{enumerate}
\item Determine the ``best'' $\lambda = \lambda (X, L)$ for a given polarized manifold $(X, L)$ so that all things go well like K\"ahler--Ricci soliton. 
At least for a Fano manifold $(X, L) = (X, -T^{-1} K_X)$, it would be $\lambda = -2\pi T$. 

\item Regard $\lambda \in \mathbb{R}$ as a free parameter in the theory and construct our theory for any $\lambda \in \mathbb{R}$ as much as possible. 
\end{enumerate}
Nakagawa \cite{Nak} takes the first stance, but it turns out in \cite{Ino2} that the latter perspective is more fruitful than the first one: the product $\mu$-cscK metrics behaves well for the same $\lambda$, extremal metric appears in the limit $\lambda \to -\infty$ and so on. 
Moreover, contrary to the theory on K\"ahler--Ricci soliton, the general theory works well for $\lambda \le 0$ rather than $\lambda > 0$. 
In any case, the parameter $\lambda$ was just introduced at first as a free parameter which enriches the theory. 
Its geometric role was unclear. 

Here we show our theory carries a thermodynamical sturcture for $T = - \frac{\lambda}{2\pi} \ge 0$. 
It gives us a better geometric intuition on our enigmatic parameter $T$ and also implicates the geometry of $\mu$-cscK metric and $\mu$K-stability is of infinite dimensional nature: it is the geometry of a composition of the space of our interest and an infinite dimensional space working as heat bath. 
This is reminiscent of Perelman's statistical mechanical heuristic argument \cite{Per} on his $W$-entropy. 

\subsection{Optimizer for product}

Mutual interaction of thermodynamical systems is a key concept in thermodynamics. 
Mathematically, it is described as a process (or its terminal state) on the product (or tensor product, in quantum setup) of two systems $\mathcal{P}_1, \mathcal{P}_2$. 
Let us begin with the simplest case: mutual interaction of isothermal systems. 

For two systems $\mathcal{P}_1 = (P_1, d\mu_1, d\sigma_1)$, $\mathcal{P}_2 = (P_2, d\mu_2, d\sigma_2)$, we consider the following \textit{composite system}
\[ \mathcal{P}_1 \times \mathcal{P}_2 := (P_1 \times P_2, d\mu_1 \times d\mu_2, d\sigma_1 \boxplus d\sigma_2), \]
where $d\mu_1 \times d\mu_2$ denotes the product measure and $d\sigma_1 \boxplus d\sigma_2$ denotes the measure on $\partial (P_1 \times P_2) = \partial P_1 \times P_2 \cup P_2 \times \partial P_2$ given by 
\[ d\sigma_1 \times d\mu_2 + d\mu_1 \times d\sigma_2. \]
We call $\mathcal{P}_1, \mathcal{P}_2$ subsystems of this composite system. 

\begin{thm}
\label{canonical distribution of product}
For $T \ge 0$, the $\mu$-canonical distribution $u^{\mathrm{can}}_T \in \M^{\exp, 1} (P_1 \times P_2)$ of temperature $T$ on the composyte system $\mathcal{P}_1 \times \mathcal{P}_2$ is the product 
\[ u^{\mathrm{can}}_T = u_1 \times u_2\]
of the $\mu$-canonical distributions $u_1 \in \M^{\exp,1} (P_1), u_2 \in \M^{\exp, 1} (P_2)$ of the same temperature $T$ on the subsystems $\mathcal{P}_1, \mathcal{P}_2$, respectively. 
\end{thm}

\begin{proof}
For $u \in \M^{\exp, 1} (P_1 \times P_2)$, we define a function $u_i$ on $P_i$ by 
\begin{align} 
\label{state of subsystem 1}
u_1 (x) 
&:= \frac{1}{\int_{P_2} d\mu_2} \int_{P_2} u (x, y) d\mu_2 (y), 
\\ \label{state of subsystem 2}
u_2 (y) 
&:= \frac{1}{\int_{P_1} d\mu_1} \int_{P_1} u (x, y) d\mu_1 (x).  
\end{align}
These are non-zero log convex functions by H\"older's inequality, lower semi-continuous by Fatou's lemma and $\int_{P_1} u_1 d\mu_1 = \int_{P_1} d\mu_1, \int_{P_2} u_2 d\mu_2 = \int_{P_2} d\mu_2$ by Fubini's theorem, so that $u_1 \in \M^{\exp, 1} (P_1)$ and $u_2 \in \M^{\exp, 1} (P_2)$. 
If $u$ is of the form $u = u' \times u''$ for $u' \in \M^{\exp, 1} (P_1)$ and $u'' \in \M^{\exp, 1} (P_2)$, we have $u_1 = u'$ and $u_2 = u''$. 

Now we compute 
\begin{align*}
\int_{\partial (P_1 \times P_2)} u d\sigma 
&= \int_{\partial P_1} \int_{P_2} u d\mu_2 d\sigma_1 + \int_{\partial P_2} \int_{P_1} u d\mu_1 d\sigma_2
\\
&= \int_{P_2} d\mu_2 \cdot \int_{\partial P_1} u_1 d\sigma_1 + \int_{P_1} d\mu_1 \cdot \int_{\partial P_2} u_2 d\sigma_2,
\end{align*}
which shows 
\begin{align}
\label{Internal energy is the sum of marginals}
U_{\mathcal{P}_1 \times \mathcal{P}_2} (u) = U_{\mathcal{P}_1} (u_1) +U_{\mathcal{P}_2} (u_2). 
\end{align}
This implies if $u$ minimizes $U_{\mathcal{P}_1 \times \mathcal{P}_2}$, then $u_i$ minimizes $U_{\mathcal{P}_i}$ for each $i = 1, 2$. 

On the other hand, we have 
\begin{align}
\label{Entropy is maximized by product}
S_{\mathcal{P}_1 \times \mathcal{P}_2} (u) \le S_{\mathcal{P}_1} (u_1) + S_{\mathcal{P}_2} (u_2) 
\end{align}
with the equality only when $u= u_1 \times u_2$. 
This is a paraphrase of the non-negativity of mutual information known in information theory. 
We can prove this as follows. 
Put $d\tilde{p} = u d\mu_1 \times d\mu_2/\int_{P_1 \times P_2} u d\mu_1 \times d\mu_2$, $d\tilde{p}_1 = u_1 d\mu_1/\int_{P_1} u_1 d\mu_1$ and $d\tilde{p}_2 = u_2 d\mu_2/\int_{P_2} u_2 d\mu_2$. 
Then we have 
\begin{align*} 
\int_{P_1 \times P_2} \frac{d\tilde{p}}{d\tilde{p}_1 d\tilde{p}_2} \log \frac{d\tilde{p}}{d\tilde{p}_1 d\tilde{p}_2} d\tilde{p}_1 d\tilde{p}_2 \ge 0
\end{align*}
by Jensen's inequality, where the equality holds if and only if $d\tilde{p} = d\tilde{p}_1 \times d\tilde{p}_2$. 
For $dp = d\mu/\int_{P_1 \times P_2} d\mu$, $dp_1 = d\mu_1/\int_{P_1} d\mu_1$ and $dp_2 =  d\mu_2/\int_{P_2} d\mu_2$, a simple calculation shows 
\begin{align*} 
\int_{P_1 \times P_2} \frac{d\tilde{p}}{d\tilde{p}_1 d\tilde{p}_2} \log \frac{d\tilde{p}}{d\tilde{p}_1 d\tilde{p}_2} d\tilde{p}_1 d\tilde{p}_2 
&= \int_{P_1 \times P_2} \frac{d\tilde{p}}{dp_1 dp_2} \log \frac{d\tilde{p}}{dp_1 dp_2} dp_1 dp_2
\\
& - \int_{P_1} \frac{d\tilde{p}_1}{dp_1} \log \frac{d\tilde{p}_1}{dp_1} dp_1 - \int_{P_2} \frac{d\tilde{p}_2}{dp_2} \log \frac{d\tilde{p}_2}{dp_2} dp_2.  
\end{align*}
The consequence is the inequality (\ref{Entropy is maximized by product}). 

Now for $T \ge 0$, let $u' \in \M^{\exp,1} (P_1), u'' \in \M^{\exp, 1} (P_2)$ be the $\mu$-canonical distributions of temperature $T$. 
Then for any $u \in \M^{\exp, 1} (P_1 \times P_2)$, we have 
\begin{align*} 
F_{\mathcal{P}_1 \times \mathcal{P}_2} (T, u) 
&\ge F_{\mathcal{P}_1} (T, u_1) + F_{\mathcal{P}_2} (T, u_2) 
\\
&\ge F_{\mathcal{P}_1} (T, u') + F_{\mathcal{P}_2} (T, u'') = F_{\mathcal{P}_1 \times \mathcal{P}_2} (T, u' \times u''),
\end{align*}
so that $u' \times u''$ minimizes $F_{\mathcal{P}_1 \times \mathcal{P}_2} (T, \bullet)$, which shows the claim for $T > 0$ by the uniqueness of minimizer. 
On the other hand, if $u$ minimizes $U_{\mathcal{P}_1 \times \mathcal{P}_2}$, then $u_i$ minimizes $U_{\mathcal{P}_i}$, so that we compute 
\[ S_{\mathcal{P}_1 \times \mathcal{P}_2} (u) \le S_{\mathcal{P}_1} (u_1) + S_{\mathcal{P}_2} (u_2) \le S_{\mathcal{P}_1} (u') + S_{\mathcal{P}_2} (u'') = S_{\mathcal{P}_1 \times \mathcal{P}_2} (u' \times u''). \]
Now the uniqueness of $\mu$-canonical distribution completes the proof. 
\end{proof}

\subsection{Equilibrium and isothermality} 

Equilibrium and isothermality are fundamental notions in thermodynamics. 
To compare our argument, let us briefly recall how the notion of equilibrium is formalized in stochastic thermodynamics. 
In stochastic thermodynamics, a measure space $(\Omega, d\mu)$ (in a simple setup, a finite set with the counting measure) together with a function $E: \Omega \to \mathbb{R}$ called Hamiltonian is interpreted as a thermodynamical system, which serves as the space of possible (deterministic) states. 
A probability measure $p$ on $\Omega$ is interpreted as nonequilibrium state. 
Nonequilibrium entropy $S (p)$ is defined by the relative entropy $- \int_\Omega \frac{dp}{dp_0} \log \frac{dp}{dp_0}$ with respect to $dp_0 = d\mu/|\Omega|$ and its nonequilibrium internal energy $U (p)$ is defined by $\int_\Omega E dp$. 
Equilibrium is described as entropy maximizer on a level set of internal energy. 

By the method of Lagrange multiplier, we can describe equilibrium as critical point of $\beta U -S$ for a proper choice of $\beta \in \mathbb{R}$ depending on energy level. 
The critical state $p_{\mathrm{can}} (\beta)$ can be written explicitly as 
\[ p_{\mathrm{can}} (\beta) = \frac{e^{-\beta E}d\mu}{\int_\Omega e^{-\beta E} d\mu}, \]
which is well known as canonical distribution. 
We can apply a similar argument to non-archimedean $\mu$-entropy. 
We note maximizer of non-archimedean $\mu$-entropy has no simple explicit description, which is different from stochastic thermodynamics. 

\subsubsection{Equilibrium}

Let $\mathcal{P} = (P, d\mu, d\sigma)$ be a system. 
For $U \in \mathbb{R}$, we consider the energy level set: 
\begin{align}
\M^{\exp, 1} (\mathcal{P}, U) := \{ u \in \M^{\exp,1} (P) ~|~ U_{\mathcal{P}} (u) = U \}. 
\end{align}
We call a state $u \in \M^{\exp, 1} (\mathcal{P}, U)$ \textit{equilibrium of internal $\mu$-energy $U$} if it maximizes the entropy $S_{\mathcal{P}}$ on the level set $\M^{\exp, 1} (\mathcal{P}, U)$. 
For instance, the trivial state $1_P$ is equilibrium of internal $\mu$-energy $U = U_{\mathcal{P}} (1_P)$. 

Our main interest is in the equilibrium of internal $\mu$-energy $U$ in the interval
\begin{align}
\mathfrak{U}_{\mathcal{P}} := [\min U_{\mathcal{P}}, U_{\mathcal{P}} (1_P)],
\end{align}
for which we can show the existence. 
We should note even though we have proved the compactness of the sublevel set $\bigcup_{U \ge U'} \M^{\exp, 1} (\mathcal{P}, U')$, we cannot conclude the compactness of the level set $\M^{\exp, 1} (\mathcal{P}, U)$ as $U_{\mathcal{P}}$ is only lower semi-continuous, so the existence of equilibrium is not a direct consequence of our compactness result. 
We make use of the continuity of the family of $\mu$-canonical distributions $u^{\mathrm{can}}_T$, which is a consequence of the uniqueness. 

\begin{thm}
\label{Existence of equilibrium}
Let $\mathcal{P}$ be a system. 
For $U \in \mathfrak{U}_{\mathcal{P}}$, there exists a unique equilibrium $u^{\mathrm{eq}} (U)$ of internal $\mu$-energy $U$. 
\end{thm}

\begin{proof}
The uniqueness of equilibrium follows by the affinity of $U_{\mathcal{P}}$ and the strict convexity of $S_{\mathcal{P}}$. 
This is a general fact for $U \in \mathbb{R}$. 

As for $U = \min U_{\mathcal{P}}$, equilibrium of this internal $\mu$-energy is nothing but the $\mu$-canonical distribution of temperature $T=0$, so we already proved the existence. 

Now we assume $U \in (\min U_{\mathcal{P}}, U_{\mathcal{P}} (1_P)]$ and show the existence of equilibrium of internal $\mu$-energy $U$. 
Recall we proved $U_{\mathcal{P}} (u_T^{\mathrm{can}})$ for the $\mu$-canonical distribution $u_T^{\mathrm{can}}$ is continuous on $[0, \infty]$ and its image is $[\min U_{\mathcal{P}}, U_{\mathcal{P}} (1_P)]$. 
It follows that there is $T \in (0, \infty]$ satisfying $U_{\mathcal{P}} (u_T^{\mathrm{can}}) = U$, hence $u_T^{\mathrm{can}} \in \M^{\exp, 1} (\mathcal{P}, U)$. 
For any other $u \in \M^{\exp, 1} (\mathcal{P}, U)$, we have 
\[ S_{\mathcal{P}} (u) = \frac{1}{T} (U - F_{\mathcal{P}} (T, u)) \le \frac{1}{T} (U - F_{\mathcal{P}} (T, u_T^{\mathrm{can}})) = S_{\mathcal{P}} (u_T^{\mathrm{can}}), \]
which shows $u_T^{\mathrm{can}}$ is the equilibrium of internal $\mu$-energy $U$. 
\end{proof}

We recall a system $\mathcal{P}$ is called K-unstable if $\mathrm{DF} (q) < 0$ for some convex function $q$. 
Thanks to Theorem \ref{muK-semistability is equivalent to mu-entropy maximization}, this is equivalent to say 
\begin{align}
\mathfrak{U}_{\mathcal{P}}^* := [\min U_{\mathcal{P}}, U_{\mathcal{P}} (1_P)) 
\end{align}
is nonempty. 
For a system $\mathcal{P}_1$ and a K-unstable system $\mathcal{P}_2$, we have $\mathfrak{U}_{\mathcal{P}_1} + \mathfrak{U}_{\mathcal{P}_2}^* = \mathfrak{U}_{\mathcal{P}_1 \times \mathcal{P}_2}^*$ for the Minkowski sum, and $\mathfrak{U}_{\mathcal{P}_1}^* + \mathfrak{U}_{\mathcal{P}_2}^* = \mathfrak{U}_{\mathcal{P}_1 \times \mathcal{P}_2}^*$ when further $\mathcal{P}_1$ is K-unstable. 

Now let $\mathcal{P}$ be a K-unstable system. 
By the monotonic continuity of $U_{\mathcal{P}} (u^{\mathrm{can}}_T)$, for each $U \in \mathfrak{U}_{\mathcal{P}}^*$, the subset
\begin{align}
\mathbb{T}_{\mathcal{P}}^{\mathrm{can}} (U) := \{ T \in [0, \infty) ~|~ U_{\mathcal{P}} (u_T^{\mathrm{can}}) = U \} 
\end{align}
is a nonempty compact interval of $[0, \infty)$, which has only one element except for at most countably many $U \in \mathfrak{U}_{\mathcal{P}}^*$. 
By the above proof, the $\mu$-canonical distribution of temperature $T \ge 0$ is the equilibrium of internal $\mu$-energy $U \in \mathfrak{U}_{\mathcal{P}}^*$ if and only if $T \in \mathbb{T}_{\mathcal{P}}^{\mathrm{can}} (U)$. 
Later we characterize this $\mathbb{T}_{\mathcal{P}}^{\mathrm{can}} (U)$ in terms of the equilibria $u^{\mathrm{eq}} (U)$ rather than the $\mu$-canonical distributions $u_T^{\mathrm{can}}$.

By Remark \ref{Fut of exp q is the difference of NAmu}, for any $u = u (q) \in \M^{\exp, 1} (P)$ with $U_{\mathcal{P}} (u) \in \mathfrak{U}_{\mathcal{P}}^*$, we have 
\[ \frac{1}{\int_P u d\mu} \mathrm{DF} (u) = 2\pi (U_{\mathcal{P}} (u) - U_{\mathcal{P}} (1_P)) < 0. \] 
This allows us to interpret $U_{\mathcal{P}} - U_{\mathcal{P}} (1_P)$ as ``ability of destabilization'' and equilibrium of internal $\mu$-energy $U \in \mathfrak{U}_{\mathcal{P}}^*$ as ``the most balanced state among all states of the same ability of destabilization''. 

\subsubsection{Isothermality}

We firstly began our theory with the enigmatic parameter $T$ called temperature, studied the minimization of $F_{\mathcal{P}} (T, \bullet)$ and finally reached the notion of equilibrium, which is a priori irrelevant to the parameter $T$. 

Now we can shift our perspective. 
According to the teaching of thermodynamics, we can rediscover the temperature $T$ at least in three essentially equivalent ways: 
\begin{enumerate}
    \item (Lagrange multiplier) $1/T^{\mathrm{eq}} (U) = (\partial S^{\mathrm{eq}}/\partial U) (U)$. 

    \item (Mutual interaction) Mutual interaction with thermostat. 
    
    \item (Carnot theorem) The ratio of heats $Q_1/Q_2$ in Carnot cycle. 
\end{enumerate}
The aim of our observation is to understand the role of temperature $T$ and the free $\mu$-energy $F_{\mathcal{P}} (T, u)$ in terms of mutual interaction. 

Now consider equilibria $u_1, u_2$ of internal $\mu$-energy $U_1, U_2$ on systems $\mathcal{P}_1, \mathcal{P}_2$, respectively. 
We are interested in comparing the product state $u_1 \times u_2$ and the equilibrium $u$ of internal $\mu$-energy $U_1 + U_2$ of the composite system $\mathcal{P}_1 \times \mathcal{P}_2$. 
Since in general we have the strict inclusion
\begin{align}
\label{product state}
\M^{\exp,1} (\mathcal{P}_1, U_1) \times \M^{\exp,1} (\mathcal{P}_2, U_2) \subsetneq \M^{\exp,1} (\mathcal{P}_1 \times \mathcal{P}_2, U_1 + U_2), 
\end{align}
the equilibrium $u$ of the composite system may differ from the initial interaction-free state $u_1 \times u_2$. 
In thermodynamics, this is known as thermalisation: the initial states are preserved under mutual interaction only when the systems are isothermal, otherwise, one subsystem is eventually warmed up and the other is cooled down. 

\begin{prop}
Let $u$ be an equilibrium on the composite system $\mathcal{P}_1 \times \mathcal{P}_2$. 
Then $u_i \in \M^{\exp, 1} (P_i)$ defined by (\ref{state of subsystem 1}) and (\ref{state of subsystem 2}) are equilibria of subsystems $\mathcal{P}_i$ and we have $u = u_1 \times u_2$. 
\end{prop}

\begin{proof}
This is a consequence of (\ref{Internal energy is the sum of marginals}) and (\ref{Entropy is maximized by product}). 
\end{proof}

\begin{defin}
Let $u', u''$ be equilibria of two K-unstable systems $\mathcal{P}_1, \mathcal{P}_2$, respectively. 
We say $u'$ and $u''$ are \textit{isothermal} if the product $u' \times u''$ is the equilibrium of internal $\mu$-energy $U_{\mathcal{P}_1} (u') + U_{\mathcal{P}_2} (u'')$ on the composite system $\mathcal{P}_1 \times \mathcal{P}_2$. 
\end{defin}

Let us observe the isothermality is reflexive. 

\begin{lem}
\label{isothermal is reflexive}
Let $u$ be an equilibrium of a system $\mathcal{P}$. 
Then $u^{\times k}$ on $\mathcal{P}^{\times k}$ is the equilibrium of internal $\mu$-energy $k U_{\mathcal{P}} (u)$. 
In particular, $u^{\times k}$ and $u^{\times \ell}$ are isothermal. 
\end{lem}

\begin{proof}
Put $U := U_{\mathcal{P}} (u)$. 
Let $\tilde{u} \in \M^{\exp, 1} (P^{\times k})$ be the equilibrium of internal $\mu$-energy $kU$. 
We put 
\begin{align*}
\tilde{u}_i := \frac{1}{\int_P d\mu} \int_P \tilde{u} (x_1, \ldots, x_k) d\mu (x_i). 
\end{align*}
Suppose $\tilde{u}_i \neq \tilde{u}_j$ for some $i, j$. 
Then for $u' := (\tilde{u}_1 + \cdots + \tilde{u}_k)/k$, we have 
\begin{align*}
U_{\mathcal{P}^{\times k}} ({u'}^{\times k}) 
&= k U_{\mathcal{P}} (u') = U_{\mathcal{P}} (\tilde{u}_1) + \cdots + U_{\mathcal{P}} (\tilde{u}_k) = U_{\mathcal{P}^{\times k}} (\tilde{u}) = kU
\\
S_{\mathcal{P}^{\times k}} ({u'}^{\times k}) 
& = k S_{\mathcal{P}} (u') > S_{\mathcal{P}} (\tilde{u}_1) + \cdots + S_{\mathcal{P}} (\tilde{u}_k) \ge S_{\mathcal{P}^{\times k}} (\tilde{u})
\end{align*}
by the strict concavity of $S_{\mathcal{P}}$. 
This contradicts to the assumption that $\tilde{u}$ is the maximizer of $S_{\mathcal{P}^{\times k}}$ on $\M^{\exp, 1} (\mathcal{P}^{\times k}, kU)$. 
Thus we have $\tilde{u}_i = \tilde{u}_j = u$ for every $i, j$ and $\tilde{u} = u^{\times k}$, which shows the claim. 
\end{proof}

We expect the isothermality is an equivalence relation for pairs $(\mathcal{P}, U)$, which in thermodynamics is assumed by the zeroth law.  
(Observe this is not true if we include K-stable systems. )
To see the transitivity, we need strict monotonicity of the canonical entropy $S_{\mathcal{P}}^{\mathrm{can}} (T) = S_{\mathcal{P}} (u^{\mathrm{can}}_T)$ on $T \ge 0$, but what we know at present is monotonicity in a weak sense. 

We will prove the strict monotonicity under a slightly better regularity assumption on optimizer. 
Before discussing it, we observe the Lagrange multiplier interpretation of temperature. 

\subsubsection{Equilibrium temperature as Lagrange multiplier}

Let $\mathcal{P}$ be a system. 
For $U \in \mathbb{R}$, we put 
\begin{align}
S_{\mathcal{P}}^{\mathrm{eq}} (U) := \sup \{ S_{\mathcal{P}} (u) ~|~ u \in \M^{\exp, 1} (\mathcal{P}, U) \}. 
\end{align}
We note $S_{\mathcal{P}}^{\mathrm{eq}} (U) = -\infty$ on $U < \min U_{\mathcal{P}}$. 
The following superadditivity is clear from (\ref{product state}): 
\begin{align} 
S_{\mathcal{P}_1 \times \mathcal{P}_2}^{\mathrm{eq}} (U_1 + U_2) \ge S_{\mathcal{P}_1}^{\mathrm{eq}} (U_1) + S_{\mathcal{P}_2}^{\mathrm{eq}} (U_2). 
\end{align}

For general $U_1, U_2 \in \mathbb{R}$, we say pairs $(\mathcal{P}_1, U_1)$ and $(\mathcal{P}_2, U_2)$ are \textit{isothermal} if 
\[ S_{\mathcal{P}_1 \times \mathcal{P}_2}^{\mathrm{eq}} (U_1 + U_2) = S_{\mathcal{P}_1}^{\mathrm{eq}} (U_1) + S_{\mathcal{P}_2}^{\mathrm{eq}} (U_2). \]
This definition makes sense even when there is no equilibria of internal $\mu$-energy $U_i$. 
When $U_i \in \mathfrak{U}_{\mathcal{P}}$, this definition is equivalent to the former definition: the equality holds if and only if the equilibria $u_1, u_2$ of internal $\mu$-energy $U_1, U_2$ on $\mathcal{P}_1, \mathcal{P}_2$, respectively, are isothermal. 
In particular, we have 
\[ S_{\mathcal{P}^{\times k}}^{\mathrm{eq}} (kU) = k S_{\mathcal{P}}^{\mathrm{eq}} (U) \]
by Lemma \ref{isothermal is reflexive}. 

\begin{prop}
The functional $S_{\mathcal{P}}^{\mathrm{eq}} (U)$ is concave on $U \in \mathbb{R}$ whose maximum is achieved at $U = U_{\mathcal{P}} (1_P)$. 
Moreover, it is strictly concave and strictly increasing on the interval $\mathfrak{U}_{\mathcal{P}}$. 
\end{prop}

\begin{proof}
For $t = p/(p+q), 1-t = q/(p+q)$ with $p,q \in \mathbb{N}$, we compute 
\begin{align*} 
S_{\mathcal{P}}^{\mathrm{eq}} ((1-t) U_0 + t U_1) 
&= \frac{1}{p+q} S_{\mathcal{P}^{\times (p+q)}}^{\mathrm{eq}} (q U_0 + p U_1)
\\
&\ge \frac{q}{p+q} S_{\mathcal{P}}^{\mathrm{eq}} (U_0) + \frac{p}{p+q} S_{\mathcal{P}}^{\mathrm{eq}} (U_1), 
\end{align*}
which shows the concavity. 
This product trick is expressive of a thermodynamical intuition behind the concavity, though we present another proof in the following. 

We can also see the concavity in a more direct way: 
for $u_0 \in \M^{\exp, 1} (\mathcal{P}, U_0)$ and $u_1 \in \M^{\exp, 1} (\mathcal{P}, u_1)$, we have $(1-t) u_0 + t u_1 \in \M^{\exp, 1} (\mathcal{P}, (1-t) U_0 + t U_1)$, so that 
\begin{align*} 
S_{\mathcal{P}}^{\mathrm{eq}} ((1-t) U_0 + t U_1) 
&= \sup \{ S_{\mathcal{P}} (u) ~|~ u \in \M^{\exp, 1} (\mathcal{P}, (1-t) U_0 + t U_1) \}
\\
&\ge \sup \{ S_{\mathcal{P}} ((1-t) u_0 + t u_1) ~|~ u_i \in \M^{\exp, 1} (\mathcal{P}, U_i) \text{ for } i=0,1 \}
\\
&\ge \sup \{ (1-t) S_{\mathcal{P}} (u_0) + t S_{\mathcal{P}} (u_1) ~|~ u_i \in \M^{\exp, 1} (\mathcal{P}, U_i) \text{ for } i=0,1 \}
\\
&= (1-t) S_{\mathcal{P}}^{\mathrm{eq}} (U_0) + t S_{\mathcal{P}}^{\mathrm{eq}} (U_1).
\end{align*}
Since the last supremum is achieved for $U_0, U_1 \in [\min U_{\mathcal{P}}, U_{\mathcal{P}} (1_P)]$, the third inequality is strict for $U_0 \neq U_1$ and $t\neq 0,1$ by the strict concavity of $S_{\mathcal{P}}$. 
The strict monotonicity is a consequence of strict concavity. 
\end{proof}

By the concavity, the right and left differentials exist. 
We put 
\begin{align}
\beta_{\mathcal{P}}^{\mathrm{eq}, +} (U) &:= \partial_{U_+} S_{\mathcal{P}}^{\mathrm{eq}} (U), \quad
T_{\mathcal{P}}^{\mathrm{eq}, +} (U) := \beta_{\mathcal{P}}^{\mathrm{eq}, +} (U)^{-1}, 
\\
\beta_{\mathcal{P}}^{\mathrm{eq}, -} (U) &:= \partial_{U_-} S_{\mathcal{P}}^{\mathrm{eq}} (U), \quad
T_{\mathcal{P}}^{\mathrm{eq}, -} (U) := \beta_{\mathcal{P}}^{\mathrm{eq}, -} (U)^{-1}, 
\end{align}
which are right and left continuous, respectively. 
We further put 
\begin{align}
\mathbb{T}_{\mathcal{P}}^{\mathrm{eq}} (U) := [T_{\mathcal{P}}^{\mathrm{eq}, -} (U), T_{\mathcal{P}}^{\mathrm{eq}, +} (U)] .
\end{align}

For $U' < U \in \mathfrak{U}_{\mathcal{P}}^*$, we have 
\[ 0 = T_{\mathcal{P}}^{\mathrm{eq}, -} (\min U_{\mathcal{P}}) \le \mathbb{T}_{\mathcal{P}}^{\mathrm{eq}} (U') < \mathbb{T}_{\mathcal{P}}^{\mathrm{eq}} (U) < \infty \]
by the strict concavity, where for two intervals $[a, b], [c,d]$ we write $[a,b] \le [c,d]$ (resp. $[a,b] < [c,d]$) if $b \le c$ (resp. $b < c$). 
Now we show the following. 

\begin{prop}
Let $\mathcal{P}$ be K-unstable system. 
Then for $U \in \mathfrak{U}_{\mathcal{P}}^*$, we have 
\[ \mathbb{T}_{\mathcal{P}}^{\mathrm{can}} (U) = \mathbb{T}_{\mathcal{P}}^{\mathrm{eq}} (U). \]
\end{prop}

\begin{proof}
Since $\mathcal{P}$ is K-unstable, for the $\mu$-canonical distribution $u_T^{\mathrm{can}}$ of temperature $T \ge 0$, we have $U_{\mathcal{P}} (u_T^{\mathrm{can}}) < U_{\mathcal{P}} (1_P)$: if $U_{\mathcal{P}} (u_T^{\mathrm{can}}) = U_{\mathcal{P}} (1_P)$, we have $u_T^{\mathrm{can}} = 1_P$, so that $\mathcal{P}$ is K-semistable. 
This implies
\[ \bigsqcup_{U \in \mathfrak{U}_{\mathcal{P}}^*} \mathbb{T}_{\mathcal{P}}^{\mathrm{can}} (U) = [0,\infty). \]
By the monotonicity of $U_{\mathcal{P}} (u_T^{\mathrm{can}})$, we have 
\[ \mathbb{T}_{\mathcal{P}}^{\mathrm{can}} (U') < \mathbb{T}_{\mathcal{P}}^{\mathrm{can}} (U) \]
for $U' < U \in \mathfrak{U}_{\mathcal{P}}^*$. 
It follows that 
\begin{align}
\label{limit of minimum canonical temperature}
\lim_{U \nearrow U_{\mathcal{P}} (1_P)} \min \mathbb{T}_{\mathcal{P}}^{\mathrm{can}} (U) = \infty.  
\end{align}
Indeed, if not, we have $U_{\mathcal{P}} (u_T^{\mathrm{can}}) \ge U_{\mathcal{P}} (1_P)$ for $T > \lim_{U \nearrow U_{\mathcal{P}} (1_P)} \min \mathbb{T}_{\mathcal{P}}^{\mathrm{can}} (U)$, which is a contradiction. 

We observe
\begin{align*}
\mathbb{T}_{\mathcal{P}}^{\mathrm{eq}} (U) \subset \mathbb{T}_{\mathcal{P}}^{\mathrm{can}} (U).
\end{align*}
For $\beta \in [\beta_{\mathcal{P}}^{\mathrm{eq}, +} (U), \beta_{\mathcal{P}}^{\mathrm{eq}, -} (U)]$ and $U' \in \mathbb{R}$, we have 
\[ S_{\mathcal{P}}^{\mathrm{eq}} (U') \le \beta (U-U') + S_{\mathcal{P}}^{\mathrm{eq}} (U) \]
by concavity. 
Thus for $\beta \in [\beta_{\mathcal{P}}^{\mathrm{eq}, +} (U), \beta_{\mathcal{P}}^{\mathrm{eq}, -} (U)]$, we have 
\[ U - \beta^{-1} S_{\mathcal{P}}^{\mathrm{eq}} (U) \le U' - \beta^{-1} S_{\mathcal{P}}^{\mathrm{eq}} (U'). \]
Now suppose $u$ is the equilibrium of internal $\mu$-energy $U$, then for another state $u' \in \M^{\exp, 1} (P)$, we compute 
\begin{align*}
F_{\mathcal{P}} (\beta^{-1}, u) 
&= U - \beta^{-1} S_{\mathcal{P}}^{\mathrm{eq}} (U) 
\\
&\le U_{\mathcal{P}} (u') - \beta^{-1} S_{\mathcal{P}}^{\mathrm{eq}} (U_{\mathcal{P}} (u')) \le U_{\mathcal{P}} (u') - \beta^{-1} S_{\mathcal{P}} (u') \le F_{\mathcal{P}} (\beta^{-1}, u'). 
\end{align*}
Thus $u$ is the $\mu$-canonical distribution of temperature $T = \beta^{-1}$ for $T \in \mathbb{T}_{\mathcal{P}}^{\mathrm{eq}} (U)$, which shows $\mathbb{T}_{\mathcal{P}}^{\mathrm{eq}} (U) \subset \mathbb{T}_{\mathcal{P}}^{\mathrm{can}} (U)$. 

As a general property on left and right derivative of convex function, we have 
\[ \bigsqcup_{U \in \mathfrak{U}_{\mathcal{P}}^*} \mathbb{T}_{\mathcal{P}}^{\mathrm{eq}} (U) = [0, \lim_{U \nearrow U_{\mathcal{P}} (1_P)} T_{\mathcal{P}}^{\mathrm{eq}, +} (U)). \]
On the other hand, by $\mathbb{T}_{\mathcal{P}}^{\mathrm{eq}} (U) \subset \mathbb{T}_{\mathcal{P}}^{\mathrm{can}} (U)$ and (\ref{limit of minimum canonical temperature}), we have 
\[ \lim_{U \nearrow U_{\mathcal{P}} (1_P)} T_{\mathcal{P}}^{\mathrm{eq}, +} (U) \ge \lim_{U \nearrow U_{\mathcal{P}} (1_P)} \min \mathbb{T}_{\mathcal{P}}^{\mathrm{can}} (U) = \infty, \]
so that we get
\[ \bigsqcup_{U \in \mathfrak{U}_{\mathcal{P}}^*} \mathbb{T}_{\mathcal{P}}^{\mathrm{eq}} (U) = [0, \infty). \]
It follows that 
\[ \bigsqcup_{U \in \mathfrak{U}_{\mathcal{P}}^*} (\mathbb{T}_{\mathcal{P}}^{\mathrm{can}} (U) \setminus \mathbb{T}_{\mathcal{P}}^{\mathrm{eq}} (U)) = \bigsqcup_{U \in \mathfrak{U}_{\mathcal{P}}^*} \mathbb{T}_{\mathcal{P}}^{\mathrm{can}} (U) \setminus \bigsqcup_{U \in \mathfrak{U}_{\mathcal{P}}^*} \mathbb{T}_{\mathcal{P}}^{\mathrm{eq}} (U) = \emptyset, \]
which shows $\mathbb{T}_{\mathcal{P}}^{\mathrm{can}} (U) = \mathbb{T}_{\mathcal{P}}^{\mathrm{eq}} (U)$. 
\end{proof}

\subsubsection{Thermalisation}

Now we can characterize the isothermality by quantities for equilibrium. 

\begin{prop}
\label{characterization of isothermality}
Let $\mathcal{P}_1, \mathcal{P}_2$ be K-unstable systems and $u_1, u_2$ be the equilibria of internal $\mu$-energy $U_1 \in \mathfrak{U}_{\mathcal{P}_1}^*, U_2 \in \mathfrak{U}_{\mathcal{P}_2}^*$ on $\mathcal{P}_1, \mathcal{P}_2$, respectively. 
Then the equilibria $u_1$ and $u_2$ are isothermal if and only if 
\begin{align}
\label{equivalent condition for isothermality}
\mathbb{T}_{\mathcal{P}_1}^{\mathrm{eq}} (U_1) \cap \mathbb{T}_{\mathcal{P}_2}^{\mathrm{eq}} (U_2) \neq \emptyset.
\end{align} 
In this case, we have 
\[ \mathbb{T}_{\mathcal{P}_1 \times \mathcal{P}_2}^{\mathrm{eq}} (U_1 +U_2) =  \mathbb{T}_{\mathcal{P}_1}^{\mathrm{eq}} (U_1) \cap \mathbb{T}_{\mathcal{P}_2}^{\mathrm{eq}} (U_2). \]
\end{prop}

\begin{proof}
Assume (\ref{equivalent condition for isothermality}) and take $T \in \mathbb{T}_{\mathcal{P}_1}^{\mathrm{eq}} (U_1) \cap \mathbb{T}_{\mathcal{P}_2}^{\mathrm{eq}} (U_2)$. 
Since $u_1, u_2$ are the $\mu$-canonical distribution of temperature $T$, the product $u_1 \times u_2$ is also the $\mu$-canonical distribution of temperature $T$ by Theorem \ref{canonical distribution of product}. 
Thus we have $T \in \mathbb{T}_{\mathcal{P}_1 \times \mathcal{P}_2}^{\mathrm{eq}} (U_1 +U_2)$ and hence $u_1 \times u_2$ is the equilibrium of internal $\mu$-energy $U_1 + U_2$. 
Therefore, $u_1$ and $u_2$ are isothermal. 

Suppose conversely $u_1$ and $u_2$ are isothermal. 
Since $u_1 \times u_2$ is the equilibrium of internal $\mu$-energy $U_1 + U_2$, it is the $\mu$-canonical distribution of temperature $T \in \mathbb{T}_{\mathcal{P}_1 \times \mathcal{P}_2}^{\mathrm{eq}} (U_1 + U_2)$. 
Again by Theorem \ref{canonical distribution of product}, $u_1, u_2$ are also the $\mu$-canonical distribution of temperature $T$ by the above proof. 
Thus we have $T \in \mathbb{T}_{\mathcal{P}_1}^{\mathrm{can}} (U_1) \cap \mathbb{T}_{\mathcal{P}_2}^{\mathrm{can}} (U_2)$. 
Since $\mathbb{T}_{\mathcal{P}_i}^{\mathrm{can}} (U_i) = \mathbb{T}_{\mathcal{P}_i}^{\mathrm{eq}} (U_i)$, we obtain the claim. 
\end{proof}

Now we introduce the following notion. 

\begin{defin}
Let $\mathcal{P}$ be a K-unstable system. 
We call $\mathcal{P}$ \textit{mild} if $\# \mathbb{T}_{\mathcal{P}}^{\mathrm{eq}} (U) = 1$ for every $U \in \mathfrak{U}_{\mathcal{P}}^*$. 
\end{defin}

We speculate every K-unstable system is mild, but this is still a conjecture. 
In the next section, we will see the mildness of some systems including $\mu$K-semistable systems as a consequence of slightly better regularity of equilibrium. 
We note if $\mathcal{P}_1$ is a system and $\mathcal{P}_2$ is a mild K-unstable system, then the composite system $\mathcal{P}_1 \times \mathcal{P}_2$ is also a mild K-unstable system. 

In view of Proposition \ref{characterization of isothermality}, the mildness assumption is essential for the transitivity of isothermaility. 
Here we conclude the following. 

\begin{cor}
Let $\mathcal{P}_1, \mathcal{P}_2, \mathcal{P}_3$ be K-unstable systems. 
Suppose $\mathcal{P}_2$ is mild.  
If for $U_1 \in \mathfrak{U}_{\mathcal{P}_1}^*, U_2 \in \mathfrak{U}_{\mathcal{P}_2}^*, U_3 \in \mathfrak{U}_{\mathcal{P}_3}^*$, $(\mathcal{P}_1, U_1)$ and $(\mathcal{P}_2, U_2)$, $(\mathcal{P}_2, U_2)$ and $(\mathcal{P}_3, U_3)$ are isothermal, respectively, then $(\mathcal{P}_1, U_1)$ and $(\mathcal{P}_3, U_3)$ are also isothermal. 
\end{cor}

\begin{figure}[H]
\begin{tikzpicture}
\draw(0.5,0.5)node{$\mathcal{P}_1$};
\draw(2,0.5)node{$\mathcal{P}_2$};
\draw[<-]
(0.8,0.5)--(1.2,0.5);
\draw[thick, blue]
(0,0)--(1,0)--(1,1)--(0,1)--(0,0);
\draw[thick, red]
(1,-0.5)--(1,1.5)--(3,1.5)--(3,-0.5)--(1,-0.5);
\end{tikzpicture}
\caption{Thermalisation of composite system: entropy increases and energy transfer is recognized as heat. }
\end{figure}
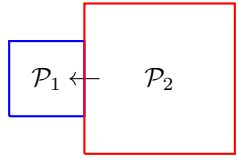

When two systems are not isothermal, the composite system is thermalized to medium temperatue. 

\begin{prop}
Let $\mathcal{P}_1, \mathcal{P}_2$ be K-unstable systems and $\mathcal{P} = \mathcal{P}_1 \times \mathcal{P}_2$ be the composite system. 
For $U_1 \in \mathfrak{U}_{\mathcal{P}_1}^*, U_2 \in \mathfrak{U}_{\mathcal{P}_2}^*$, we have either 
\[ \mathbb{T}_{\mathcal{P}_1}^{\mathrm{eq}} (U_1) < \mathbb{T}_{\mathcal{P}_2}^{\mathrm{eq}} (U_2), \quad \mathbb{T}_{\mathcal{P}_1}^{\mathrm{eq}} (U_1) \cap \mathbb{T}_{\mathcal{P}_2}^{\mathrm{eq}} (U_2) \neq \emptyset, \quad \mathbb{T}_{\mathcal{P}_1}^{\mathrm{eq}} (U_1) > \mathbb{T}_{\mathcal{P}_2}^{\mathrm{eq}} (U_2). \]
Suppose $\mathbb{T}_{\mathcal{P}_1}^{\mathrm{eq}} (U_1) < \mathbb{T}_{\mathcal{P}_2}^{\mathrm{eq}} (U_2)$, then we have 
\[ \mathbb{T}_{\mathcal{P}_1}^{\mathrm{eq}} (U_1) < \mathbb{T}_{\mathcal{P}_1 \times \mathcal{P}_2}^{\mathrm{eq}} (U_1 + U_2) < \mathbb{T}_{\mathcal{P}_2}^{\mathrm{eq}} (U_2). \]
\end{prop}

\begin{proof}
Let $u_1, u_2$ be the equilibria associated to $u$ as in (\ref{state of subsystem 1}) and (\ref{state of subsystem 2}). 
Since $u_1$ and $u_2$ are isothermal, by the above proposition, we have 
\[ \mathbb{T}_{\mathcal{P}_1 \times \mathcal{P}_2}^{\mathrm{eq}} (U_1 +U_2) =  \mathbb{T}_{\mathcal{P}_1}^{\mathrm{eq}} (U_{\mathcal{P}_1} (u_1)) \cap \mathbb{T}_{\mathcal{P}_2}^{\mathrm{eq}} (U_{\mathcal{P}_2} (u_2)). \]

Now we have exactly three possibilities: $U_{\mathcal{P}_1} (u_1) < U_1$, $U_{\mathcal{P}_1} (u_1) = U_1$ and $U_{\mathcal{P}_1} (u_1) > U_1$. 
Suppose $U_{\mathcal{P}_1} (u_1) > U_1$, then 
\[ \mathbb{T}_{\mathcal{P}_1 \times \mathcal{P}_2}^{\mathrm{eq}} (U_1 + U_2) = \mathbb{T}_{\mathcal{P}_1}^{\mathrm{eq}} (U_{\mathcal{P}_1} (u_1)) \cap \mathbb{T}_{\mathcal{P}_2}^{\mathrm{eq}} (U_{\mathcal{P}_2} (u_2)) \ge \mathbb{T}_{\mathcal{P}_1}^{\mathrm{eq}} (U_{\mathcal{P}_1} (u_1)) > \mathbb{T}_{\mathcal{P}_1}^{\mathrm{eq}} (U_1). \]
Since $U_{\mathcal{P}_2} (u_2) = (U_1 - U_{\mathcal{P}_1} (u_1)) + U_2 < U_2$, we get 
\[ \mathbb{T}_{\mathcal{P}_2}^{\mathrm{eq}} (U_2) > \mathbb{T}_{\mathcal{P}_1 \times \mathcal{P}_2}^{\mathrm{eq}} (U_1 + U_2) \]
by the same argument. 
It follows that $\mathbb{T}_{\mathcal{P}_1}^{\mathrm{eq}} (U_1) < \mathbb{T}_{\mathcal{P}_2}^{\mathrm{eq}} (U_2)$. 

Similarly, we obtain $\mathbb{T}_{\mathcal{P}_1}^{\mathrm{eq}} (U_1) > \mathbb{T}_{\mathcal{P}_2}^{\mathrm{eq}} (U_2)$ when $U_{\mathcal{P}_1} (u_1) < U_1$, $\mathbb{T}_{\mathcal{P}_1}^{\mathrm{eq}} (U_1) = \mathbb{T}_{\mathcal{P}_2}^{\mathrm{eq}} (U_2)$ when $U_{\mathcal{P}_1} (u_1) = U_1$. 
Since these are exclusive conditions, the claim is proved. 
\end{proof}

We will see the medium temperature $\mathbb{T}_{\mathcal{P}_1 \times \mathcal{P}_2}^{\mathrm{eq}} (U_1+U_2)$ can be arbitrary close to $\mathbb{T}_{\mathcal{P}_2}^{\mathrm{eq}}$ when the system $\mathcal{P}_2$ is sufficiently large. 

\subsection{Thermodynamics of non-archimedean $\mu$-entropy}

In the previous section, we discuss equilibrium and thermalisation. 
Here we rediscover the non-archimedean $\mu$-entropy from a further exploration on thermalisation. 
This observation is well-known in thermodynamics. 
We present it in a mathematically rigorous way. 

\subsubsection{Temperature as variable}

Let us firstly consider change of variables as usual in thermodynamics. 

Let $\mathcal{P}$ be a K-unstable system. 
Recall for $T \in [0, \infty)$, there exists unique $U \in \mathfrak{U}_{\mathcal{P}}^*$ such that $T \in \mathbb{T}_{\mathcal{P}}^{\mathrm{eq}}$. 
For a system $\mathcal{P}$ and $T \in [0, \infty)$, we put 
\begin{align}
U_{\mathcal{P}}^{\mathrm{can}} (T) := 
\begin{cases}
U \text{ satisfying } T \in \mathbb{T}_{\mathcal{P}}^{\mathrm{eq}} (U)
& \mathcal{P} \text{ is K-unstable}
\\
U_{\mathcal{P}} (1_P)
& \mathcal{P} \text{ is K-semistable}
\end{cases}
\end{align}
This can be regarded as the inverse function of $\mathbb{T}_{\mathcal{P}}^{\mathrm{eq}} (U)$. 
It is continuous and increasing on $T$. 
We further put 
\begin{align}
S_{\mathcal{P}}^{\mathrm{can}} (T) := S_{\mathcal{P}}^{\mathrm{eq}} (U_{\mathcal{P}}^{\mathrm{can}} (T)),
\end{align}
which is also continuous and increasing. 

Now for $T \in \mathbb{R}$, we introduce 
\begin{align} 
F_{\mathcal{P}}^{\mathrm{can}} (T) := \min_{u \in \M^{\exp, 1} (P)} F_{\mathcal{P}} (T, u) = \min_{u \in \M^{\exp, 1} (P)} (U_{\mathcal{P}} ( u) - T S_{\mathcal{P}} ( u)). 
\end{align}
This is a concave function on $T$. 
Then we have the following. 

\begin{prop}
\label{canonical free energy}
For $T \in [0, \infty)$, we have 
\[ F_{\mathcal{P}}^{\mathrm{can}} (T) = U_{\mathcal{P}}^{\mathrm{can}} (T) - T S_{\mathcal{P}}^{\mathrm{can}} (T). \]
Moreover, $F_{\mathcal{P}}^{\mathrm{can}} (T)$ is continuously differentiable on $T \in [0, \infty)$ with
\begin{align} 
\partial_T F_{\mathcal{P}}^{\mathrm{can}} (T) = - S_{\mathcal{P}}^{\mathrm{can}} (T).
\end{align}
\end{prop}

\begin{proof}
For the $\mu$-canonical distribution of temperature $T$, we have 
\[ F_{\mathcal{P}}^{\mathrm{can}} (T) = F_{\mathcal{P}} (u_T^{\mathrm{can}}) =U_{\mathcal{P}} (u_T^{\mathrm{can}}) - T S_{\mathcal{P}} (u_T^{\mathrm{can}}) \]
by definiton. 

The $\mu$-canonical distribution $u_T^{\mathrm{can}}$ of temperature $T \in \mathbb{T}_{\mathcal{P}}^{\mathrm{eq}} (U)$ has internal $\mu$-energy $U_{\mathcal{P}} (u_T^{\mathrm{can}}) = U$, so that we have 
\[ U_{\mathcal{P}} (u_T^{\mathrm{can}}) = U_{\mathcal{P}}^{\mathrm{can}} (T). \]

Since $u_T^{\mathrm{can}} = u^{\mathrm{eq}} (U)$ for the equilibrium of internal $\mu$-energy $U = U_{\mathcal{P}}^{\mathrm{can}} (T)$, we have
\[ S_{\mathcal{P}} (u_T^{\mathrm{can}}) = S_{\mathcal{P}} (u^{\mathrm{eq}} (U)) = S_{\mathcal{P}}^{\mathrm{eq}} (U) = S_{\mathcal{P}}^{\mathrm{can}} (T). \] 
Therefore we get the first claim.

To see the second claim, note
\[ F_{\mathcal{P}}^{\mathrm{can}} (T) \le U_{\mathcal{P}} (u) - T S_{\mathcal{P}} (u) \]
for every $u \in \M^{\exp, 1} (P)$. 
In particular for $T, T' \ge 0$, we have 
\[ F_{\mathcal{P}}^{\mathrm{can}} (T) \le U_{\mathcal{P}}^{\mathrm{can}} (T') - T S_{\mathcal{P}}^{\mathrm{can}} (T') \]
with the equality at $T= T'$ thanks to the first claim. 
It follows that $S_{\mathcal{P}}^{\mathrm{can}} (T')$ is a subdifferential of the convex function $-F_{\mathcal{P}}^{\mathrm{can}} (T)$ at $T=T'$. 
Since $S_{\mathcal{P}}^{\mathrm{can}} (T)$ is continuous, $-F_{\mathcal{P}}^{\mathrm{can}} (T)$ is actually differentiable on $T \ge 0$ with the derivative $S_{\mathcal{P}}^{\mathrm{can}} (T)$. 
\end{proof}

\subsubsection{Mild regularity hypothesis}

The mildness of a K-unstable system $\mathcal{P}$ is equivalent to the strict monotonicity of $U_{\mathcal{P}}^{\mathrm{can}} (T)$. 
Since $S_{\mathcal{P}}^{\mathrm{eq}} (U)$ is strictly increasing, it is also equivalent to the strict monotonicity of $S_{\mathcal{P}}^{\mathrm{can}} (T)$. 

\begin{prop}
Let $\mathcal{P}$ be a K-unstable system. 
Suppose for every $T \ge 0$ there exists $p >1$ such that $\int_{\partial P} (u_T^{\mathrm{can}})^p d\sigma < \infty$ for the $\mu$-canonical distribution $u_T^{\mathrm{can}}$ of temperature $T$. 
Then $S_{\mathcal{P}}^{\mathrm{can}} (T)$ is strictly increasing, hence $\mathcal{P}$ is mild. 
\end{prop}

\begin{proof}
Take $T \in [0, \infty)$. 
It suffices to show $S_{\mathcal{P}}^{\mathrm{can}} (T') > S_{\mathcal{P}}^{\mathrm{can}} (T)$ for every $T' > T$. 
Suppose we could find a family $\{ u_t \}_{t \in (-\varepsilon, \varepsilon)}$ in $\M^{\exp, 1} (P)$ so that 
\begin{itemize}
\item $u_0 = u_T^{\mathrm{can}}$

\item $F_{\mathcal{P}} (T, u_t), S_{\mathcal{P}} (u_t)$ are differentiable at $t=0$ and $\dot{S} := \partial_{t=0} S_{\mathcal{P}} (u_t) > 0$. 
\end{itemize}
Since $u_T^{\mathrm{can}}$ minimizes $F_{\mathcal{P}} (T, \bullet)$, we have $\partial_{t=0} F_{\mathcal{P}} (T, u_t) = 0$. 
This implies that for every $T' > T$, there exists $t_{T'} > 0$ such that  
\[ (T'-T) \frac{S_{\mathcal{P}} (u_t) - S_{\mathcal{P}} (u_T^{\mathrm{can}})}{t} > \frac{F_{\mathcal{P}} (T, u_t) - F_{\mathcal{P}} (T, u_T^{\mathrm{can}})}{t} \]
for every $0 < t < t_{T'}$. 
Then using $F_{\mathcal{P}} (T', u_t) = F_{\mathcal{P}} (T, u_t) - (T' - T) S_{\mathcal{P}} (u_t)$, we compute
\begin{align*} 
F_{\mathcal{P}}^{\mathrm{can}} (T) - (T' - T) S_{\mathcal{P}}^{\mathrm{can}} (T) 
&= F_{\mathcal{P}} (T, u_T^{\mathrm{can}}) - (T' - T) S_{\mathcal{P}} (u_T^{\mathrm{can}}) 
\\
&> F_{\mathcal{P}} (T, u_t) - (T'-T) S_{\mathcal{P}} (u_t) 
\\
&= F_{\mathcal{P}} (T', u_t) \ge F_{\mathcal{P}}^{\mathrm{can}} (T'). 
\end{align*}
It follows by convexity that 
\[ S_{\mathcal{P}}^{\mathrm{can}} (T') = - \partial_T F_{\mathcal{P}}^{\mathrm{can}} (T') \ge - \frac{F_{\mathcal{P}}^{\mathrm{can}} (T') - F_{\mathcal{P}}^{\mathrm{can}} (T)}{T' -T} > S_{\mathcal{P}}^{\mathrm{can}} (T), \]
which shows the strict monotonicity of $S_{\mathcal{P}}^{\mathrm{can}} (T)$. 

Now we construct the family $\{ u_t \}_{t \in (-\varepsilon, \varepsilon)}$ under the regularity assumption on $u_T^{\mathrm{can}}$. 
Take $q \in \E^{\exp, \frac{n}{n-1}} (P)$ with $q \ge 0$ so that $u (q) = u_T^{\mathrm{can}}$. 
We put $u_t := u ((1-t)q)$. 
For $\frac{d}{dt} e^{(1-t)q} = - q e^{(1-t)q}$, we have $\int_{\partial P} q e^{(1-t) q} d\sigma \le \int_{\partial P} e^{(1-t+\epsilon) q} d\sigma < \infty$ for small $t$ by assumption. 
This implies the differentiability of $F_{\mathcal{P}} (T, u_t)$ around $t = 0$. 
Meanwhile, we have 
\[ \frac{d}{dt}\Big{|}_{t=0} S_{\mathcal{P}} (u_t) =  \frac{\int_P q^2 e^q d\mu}{\int_P e^q d\mu} - \Big{(} \frac{\int_P qe^q d\mu}{\int_P e^q d\mu} \Big{)}^2 > 0 \]
by Cauchy--Schwarz inequality. 
This completes the proof. 
\end{proof}

\begin{cor}
Let $\mathcal{P}$ be the system associated to a polarized toric variety $(X, L)$ which is $\mu^\lambda$K-semistable for every $\lambda \le 0$. 
Then $\mathcal{P}$ is either K-semistable or mild K-unstable system. 
\end{cor}

In view of crystal conjecture (Conjecture \ref{regularity conjecture}), the assumption of the above theorem, which we call the \textit{mild regularity hypothesis}, is believed to be always true. 

\begin{rem}
As a weak evidence for the mild regularity hypothesis, we note 
\[ \int_{\partial P} u \log u d\sigma = \int_{\partial P} q e^q d\sigma < \infty \]
for $\mu$-canonical distribution $u = e^q$ of any temperature. 
This could be viewed as a limit of $N (u^{1 + \frac{1}{N}} - u)$. 

For $t \searrow 0$, we have pointwise monotone convergence $0 \le (e^q - e^{(1-t)q})/t \nearrow q e^q$ by the convexity of exponential. 
Then by monotone convergence theorem, we compute 
\[ \frac{d}{dt_+}\Big{|}_{t=0} \int_{\partial P} e^{(1-t) q} d\sigma = \lim_{t \searrow 0} \frac{\int_{\partial P} e^{(1-t) q} d\sigma - \int_{\partial P} e^q d\sigma}{t} = -\int_{\partial P} q e^q d\sigma. \]
Since $\int_P e^{(1-t) q} d\mu$ and $\bm{\sigma} ((1-t)q)$ are smooth for $1- t \in (0, \frac{n}{n-1})$, we directly compute
\begin{align*} 
0 
&\ge \frac{d}{dt_+}\Big{|}_{t=0} \NAmu^{-2\pi T} ((1-t) q) 
\\
&= \frac{2\pi}{\int_P e^q d\mu} (\int_{\partial P} q e^q d\sigma - \frac{\int_{\partial P} e^q d\sigma}{\int_P e^q d\mu} \int_P q e^q d\mu) - 2\pi T \frac{d}{dt}\Big{|}_{t=0} \bm{\sigma} ((1-t)q), 
\end{align*}
where the inequality holds as $q$ is a maximizer of $\NAmu^{-2\pi T}$. 
Since $q \in \E^{\exp, \frac{n}{n-1}} (P)$, the integrals
\[ \frac{\int_{\partial P} e^q d\sigma}{\int_P e^q d\mu} \int_P q e^q d\mu, \quad \frac{d}{dt}\Big{|}_{t=0} \bm{\sigma} ((1-t)q) = -\frac{\int_P q^2 e^q d\mu}{\int_P e^q d\mu} + \Big{(} \frac{\int_P qe^q d\mu}{\int_P e^q d\mu} \Big{)}^2 \]
are finite, thus $\int_{\partial P} q e^q d\sigma \ge 0$ is also finite by the inequality. 
\end{rem}

\begin{rem}
\label{Positive differential}
Later we further consider the assumption that $S_{\mathcal{P}}^{\mathrm{can}} (T)$ has strictly positive differential $\partial_T S_{\mathcal{P}}^{\mathrm{can}} (T) > 0$ at some $T' > 0$. 
In thermodynamical terminology, the quantity $\partial_T U_{\mathcal{P}}^{\mathrm{can}} (T) = T \partial_T S_{\mathcal{P}}^{\mathrm{can}} (T)$ is called the \textit{heat capacity} of $(\mathcal{P}, T)$. 

If a polarized toric manifold $(X, L)$ admits $\mu^\lambda$-cscK metric for every $\lambda \le 0$, then we can show the family of optimal vectors $\{ \xi_\lambda \in \mathfrak{t} \}_{\lambda \le 0}$ are smooth. 
(Apply implicit function theorem to $\mu$-cscK equation. Compare the proof of \cite[Theorem 5.1]{Ino2}. )
It follows that for the associated system $\mathcal{P}$, $S_{\mathcal{P}}^{\mathrm{can}} (T)= S_{\mathcal{P}} (u (\ket{\xi_{-2\pi T}}))$ is smooth. 
We note the positivity of differential implies the strict monotonicity, but the strict monotonocity does not necessarily imply the positivity of differential. 
\end{rem}

\subsubsection{Illustration}

Here we illustrate an explicit example with positive heat capacity. 

Let $X$ be the one point-blowing up of $\mathbb{C}P^2$ and $L$ be the anti-canonical polarization $-K_X$. 
The associated polytope $P$ can be illustrated as follows. 
\begin{figure}[H]
\begin{tikzpicture}
\draw(0,-0.5)node[below]{$(0,-1)$};
\draw(1,-0.5)node[right]{$(2,-1)$};
\draw(-0.5,1)node[left]{$(-1,2)$};
\draw(-0.5,0)node[left]{$(-1,0)$};
\draw[thick]
(0,-0.5)--(1,-0.5)--(-0.5,1)--(-0.5,0)--(0,-0.5);
\end{tikzpicture}
\end{figure}

It is shown in \cite{Ino2} that $(X, L)$ admits $\mu^\lambda$-cscK metric for every $\lambda \in \mathbb{R}$ and the optimal vectors $\xi_\lambda \in \mathfrak{t}$ is of the form $x_\lambda. \eta := (x_\lambda, x_\lambda) \in \mathbb{R}^2$. 

Similarly as \cite[section 5.2]{Ino5}, we can compute 
\begin{align*}
\int_{\partial P} e^{x \ket{\eta}} d\mu 
&= - \frac{1}{x} ((2-x)e^{-x} - (3x+2) e^x), 
\\
\int_P e^{x \ket{\eta}} d\mu 
&= \frac{1}{x^2} ((-x+1) e^{-x} + (3x-1) e^x), 
\\
\int_P x \ket{\eta} e^{x \ket{\eta}} d\mu 
&= \frac{1}{x^2} ((x^2-2) e^{-x} + (3x^2 - 4x+2) e^x). 
\end{align*}
Then we get the explicit expression
\begin{align*}
\NAmu (x \ket{\eta})
&= 2 \pi x \frac{(2-x)e^{-x} - (3x+2) e^x}{(-x+1) e^{-x} + (3x-1) e^x},
\\
\bm{\sigma} (x \ket{\eta}) 
&= \frac{(x^2-2) e^{-x} + (3x^2 - 4x+2) e^x}{(-x+1) e^{-x} + (3x-1) e^x} 
\\
&\qquad - \log \frac{1}{x^2} \Big{(} (-x+1) e^{-x} + (3x-1) e^x \Big{)} - \log \int_P e^{-n} d\mu
\end{align*}
and 
\begin{align*}
\frac{d}{dx} \NAmu (x \ket{\eta})
&= \frac{(x^2-2x+2) e^{-2x} - (9x^2-6x-2)e^{2x} +12x^3-16x^2-4x-4}{(x^2 -2x+1)e^{-2x} + (9x^2 - 6x+1) e^{2x} -3x^2+4x-1 }
\\
\frac{d}{dx} \bm{\sigma} (x \ket{\eta})
&= -\frac{(x^2-2x+2) e^{-2x} - (9x^2-6x-2) e^{2x} +12x^3-16x^2-4x-4}{(x^2 -2x+1)e^{-2x} + (9x^2 - 6x+1) e^{2x} -3x^2+4x-1 }
\\
&\qquad - \frac{(x^2-2)e^{-x} + (3x^2-4x+2) e^x}{x ((-x+1)e^{-x} +(3x-1)e^x)}
\\
&= \frac{(x^2-4x+2) e^{-2x} + (9x^2 -12x+2) e^{2x} -12x^4+16x^3-2x^2+16x -4 }{x ((x^2 -2x+1)e^{-2x} + (9x^2 - 6x+1) e^{2x} -3x^2+4x-1)}.
\end{align*}
We can see $\frac{d}{dx} \bm{\sigma} (x \ket{\eta}) < 0$ for $x < 0$ as illustrated in the following figure. 

\begin{figure}[H]
\includegraphics[width=5cm]{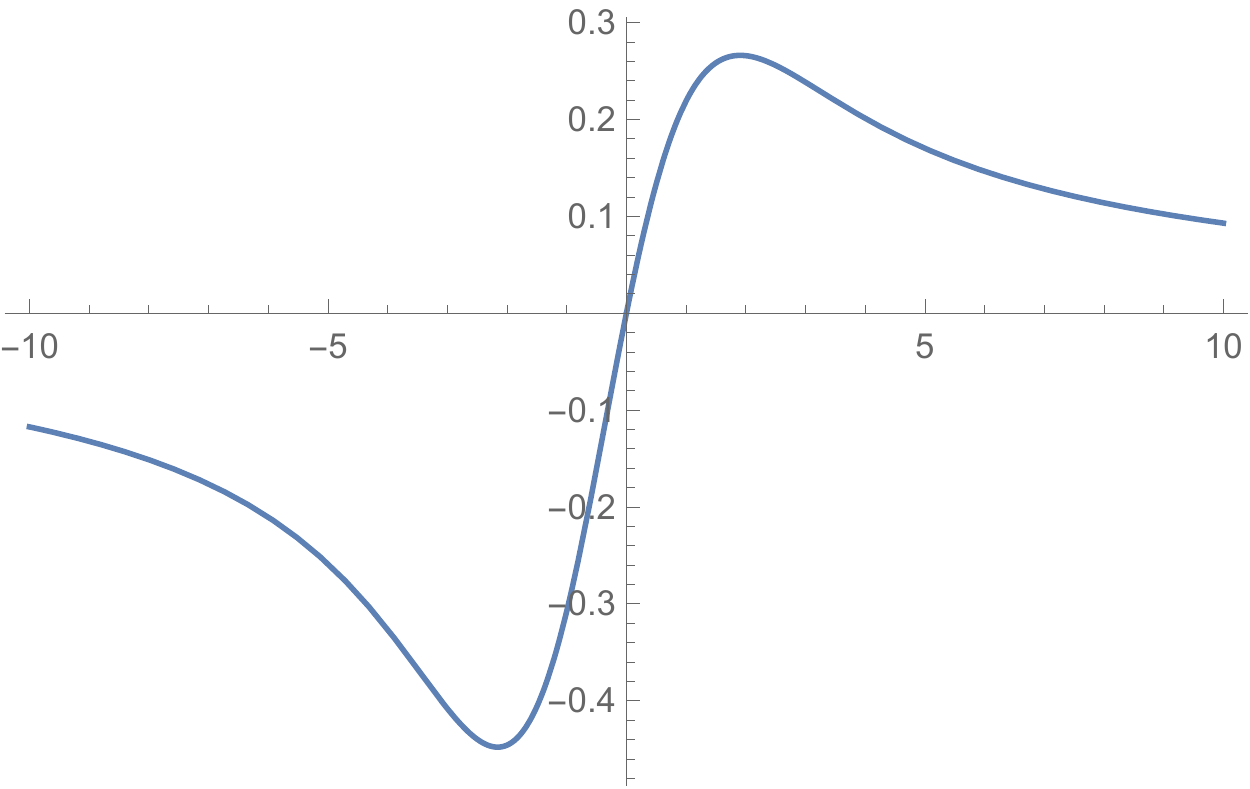}
\caption{The graph of $\frac{d}{dx} \bm{\sigma} (x \ket{\eta})$. }
\end{figure}

The optimal vector $x_\lambda. \eta$ is the critical point of $\NAmu^\lambda|_{\mathfrak{t}}$, so we have
\[ \frac{d}{dx} \NAmu (x_\lambda \ket{\eta}) + \lambda \frac{d}{dx} \bm{\sigma} (x_\lambda \ket{\eta}) = 0. \]
If we put
\[ \lambda (x) := 2\pi x \frac{(9x^2 - 6x -2)e^{2x} + (-x^2 +2x - 2) e^{-2 x} + (-12 x^3 + 16 x^2 +4x +4)}{(9x^2 - 12 x +2) e^{2x} + (x^2 - 4x +2) e^{-2x} + (-12 x^4 + 16 x^3 -2 x^2 + 16 x -4)}, \]
this is equivalent to the condition $\lambda (x_\lambda) = \lambda$. 

\begin{figure}[H]
\includegraphics[width=4cm]{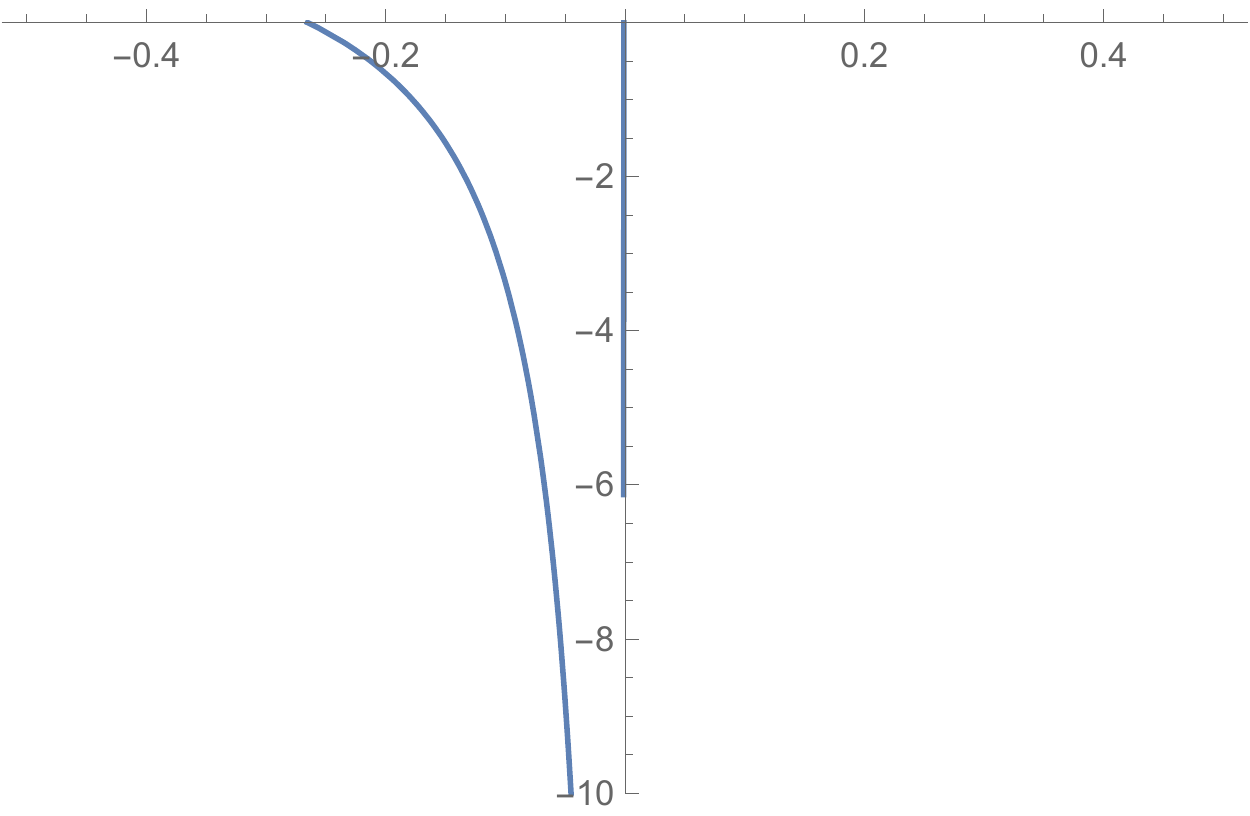}
\qquad
\includegraphics[width=4cm]{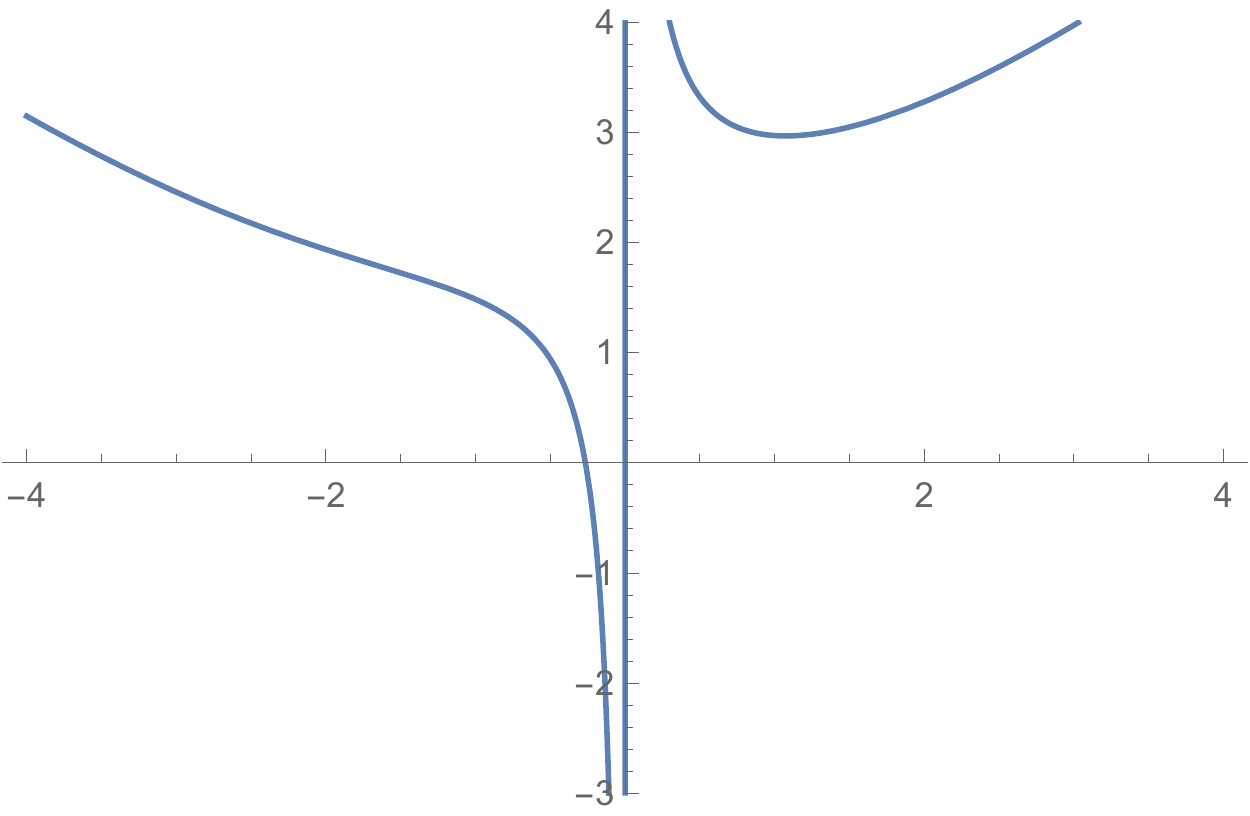}
\caption{The graph of $\frac{1}{2\pi} \lambda (x)$. The vertical axis represents $\lambda$. The left image focuses on the range $\lambda \le 0$, which illustrates $x_\lambda < 0$ and $\frac{d}{d\lambda} x_\lambda < 0$. In the right image, we can observe more than one solutions appear around $\frac{1}{2\pi} \lambda \approx 3$. }
\end{figure}

For $\lambda \le 0$, we have $x_\lambda < 0$ and $\frac{d}{d\lambda} x_\lambda = (\frac{d}{dx} \lambda (x_\lambda))^{-1} < 0$. 
It follows that 
\[ \partial_T S_{\mathcal{P}}^{\mathrm{can}} (T) = 2\pi \frac{d}{d\lambda} \bm{\sigma} (x_\lambda \ket{\eta}) = \frac{d}{d\lambda} x_\lambda \cdot \frac{d}{dx} \bm{\sigma} (x_\lambda \ket{\eta}) > 0 \] 
for $T = - \frac{\lambda}{2\pi} \ge 0$. 

\subsubsection{Heat bath}

It is well known in thermodynamics that the free energy can be derived as the entropy of a composition of the system of our interest and a sufficiently large system working as heat bath. 
We can realize large system as limit of infinitely many composition. 
This observation gives us a new interpretation of $\mu$-canonical distribution: $\mu$-canonical distribution is equilibrium of an infinite dimensional system. 

\begin{figure}[H]
\begin{tikzpicture}
\draw(0.5,0.5)node{$\mathcal{P}$};
\draw(1.25,-0.25)node{$\mathcal{P}_R$};
\draw(3.05,1.15)node{$\mathcal{P}_R$};
\draw(4,0.75)node{$\mathcal{P}_R^{\times N}$};
\draw(3.5,-0.25)node{$\cdots$};
\draw[thick, blue]
(0,0)--(1,0)--(1,1)--(0,1)--(0,0);
\draw[thick, red]
(1,-0.5)--(1,0)--(1.5,0)--(1.5,-0.5)--(1,-0.5);
\draw[thick, red]
(1,0)--(1,0.5)--(1.5,0.5)--(1.5,0)--(1,0);
\draw[thick, red]
(1,0.5)--(1,1)--(1.5,1)--(1.5,0.5)--(1,0.5);
\draw[thick, red]
(1,1)--(1,1.5)--(1.5,1.5)--(1.5,1)--(1,1);
\draw[thick, red]
(1.5,-0.5)--(1.5,0)--(2,0)--(2,-0.5)--(1.5,-0.5);
\draw[thick, red]
(1.5,0)--(1.5,0.5)--(2,0.5)--(2,0)--(1.5,0);
\draw[thick, red]
(1.5,0.5)--(1.5,1)--(2,1)--(2,0.5)--(1.5,0.5);
\draw[thick, red]
(2,-0.5)--(2,0)--(2.5,0)--(2.5,-0.5)--(2,-0.5);
\draw[thick, red]
(2,0)--(2,0.5)--(2.5,0.5)--(2.5,0)--(2,0);
\draw[thick, red]
(2.5,-0.5)--(2.5,0)--(3,0)--(3,-0.5)--(2.5,-0.5);
\draw[thick, red]
(3.1,0.8)--(2.7,1.1)--(3.0,1.5)--(3.4,1.2)--(3.1,0.8);
\draw[->] (2.5,1.1)--(2.25,0.9);
\end{tikzpicture}
\caption{Realization of heat bath as the limit of infinitely many composition}
\end{figure}
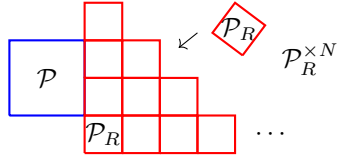

\begin{thm}[Heat bath limit]
Let $\mathcal{P}$ be a K-unstable system and $\mathcal{P}_R$ be a mild K-unstable system. 
Fix $U \in \mathfrak{U}_{\mathcal{P}}^*$ and $T_R \in [0, \infty)$. 
For $N \in \mathbb{N}$, consider the composite system 
\[ \tilde{\mathcal{P}}_N = \mathcal{P} \times \mathcal{P}_R^{\times N}. \]

Let $\tilde{u}_N \in \M^{\exp, 1} (\tilde{P}_N)$ be the equilibrium of internal $\mu$-energy $U + N U_{\mathcal{P}_R}^{\mathrm{can}} (T_R)$. 
Then the following associated equilibrium on the subsystem $P$
\[ u_N := \frac{1}{\int_{P_R^{\times N}} d\mu_{P_R}^{\times N}} \int_{P_R^{\times N}} \tilde{u}_N d\mu_{P_R}^{\times N} \in \M^{\exp, 1} (P), \]
converges in $L^p$-topology ($p \in [1, \frac{n}{n-1})$) to the $\mu$-canonical distribution $u_\infty \in \M^{\exp, 1} (P)$ of temperature $T_R$, which is independent of the choice of $U \in \mathfrak{U}_{\mathcal{P}}$. 
\end{thm}

\begin{proof}
We note $U + NU_{\mathcal{P}_R}^{\mathrm{can}} (T_R) \in \mathfrak{U}_{\tilde{\mathcal{P}}_N}^*$. 
Since $\mathcal{P}_R$ is mild, the composite system $\mathcal{P} \times \mathcal{P}_R^{\times N}$ is also mild. 
Let $T_N \in [0, \infty)$ be the element of the one point set 
\[ \mathbb{T}_{\tilde{\mathcal{P}}_N}^{\mathrm{eq}} (U + N U_{\mathcal{P}_R}^{\mathrm{can}} (T_R)). \]
Since $\tilde{u}_N$ on $\tilde{\mathcal{P}}_N$ is the $\mu$-canonical distribution of temperature $T_N$, $u_N$ on $\mathcal{P}$ is also the $\mu$-canonical distribution of temperature $T_N$ by Theorem \ref{canonical distribution of product}. 

Since 
\begin{align}
\label{internal energy equality}
U_{\mathcal{P}}^{\mathrm{can}} (T_N) + N U_{\mathcal{P}_R}^{\mathrm{can}} (T_N) = U_{\tilde{\mathcal{P}}_N}^{\mathrm{can}} (T_N) = U+NU_{\mathcal{P}_R}^{\mathrm{can}} (T_R), 
\end{align}
we compute 
\[ U^{\mathrm{can}}_{\mathcal{P}_R} (T_N) = \frac{1}{N} (U + N U_{\mathcal{P}_R}^{\mathrm{can}} (T_R) - U_{\mathcal{P}}^{\mathrm{can}} (T_N)) = U_{\mathcal{P}_R}^{\mathrm{can}} (T_R) + \frac{1}{N} (U - U_{\mathcal{P}}^{\mathrm{can}} (T_N)). \]
Since $U_{\mathcal{P}}^{\mathrm{can}} (T_N) \in \mathfrak{U}_{\mathcal{P}}^*$ is bounded, we get $U^{\mathrm{can}}_{\mathcal{P}_R} (T_N) \to U^{\mathrm{can}}_{\mathcal{P}_R} (T_R)$ as $N \to \infty$. 
It follows that $T_N \to T_R$. 
By the continuity we already proved, we conclude $u_N$ converges to the $\mu$-canonical distribution $u_\infty$ of temperature $T_R$. 
\end{proof}

Now we obtain a characterization of free $\mu$-energy in terms of equilibrium of composite system. 

\begin{thm}[Free $\mu$-energy as composite entropy]
Let $\mathcal{P}, \mathcal{P}_R$, $U, T_R$ and $\tilde{u}_N, u_\infty$ be the same as in the above theorem. 
Assume further the heat capacity $T_R \partial_T S_{\mathcal{P}}^{\mathrm{can}} (T_R)$ of $(\mathcal{P}_R, T_R)$ is positive. 
(See Remark \ref{Positive differential}. )
Namely we assume $T_R > 0$, $S^{\mathrm{can}}_{\mathcal{P}_R}$ is differentiable at $T_R$ and $\partial_T S^{\mathrm{can}}_{\mathcal{P}_R} (T_R) > 0$. 

Let $u_R^{\times N} \in \M^{\exp, 1} (P_R^{\times N})$ be the equilibrium of internal $\mu$-energy $NU_{\mathcal{P}_R}^{\mathrm{can}} (T_R)$ on $\mathcal{P}_R^{\times N}$. 
Then for any $u \in \M^{\exp, 1} (\mathcal{P}, U)$ of internal $\mu$-energy $U$, the difference of composite entropy
\begin{align*} 
\Delta S_{\tilde{\mathcal{P}}_N} 
&:= S_{\tilde{\mathcal{P}}_N} (\tilde{u}_N) - S_{\tilde{\mathcal{P}}_N} (u \times u_R^{\times N}) 
\end{align*}
converges to
\[ - \frac{1}{T_R} (F_{\mathcal{P}} (T_R, u_\infty) - F_{\mathcal{P}} (T_R, u)) \]
as $N \to \infty$, which is independent of the choice of the mild system $\mathcal{P}_R$. 
\end{thm}

\begin{proof}
When $T_R \in \mathbb{T}_{\mathcal{P}}^{\mathrm{eq}} (U)$, we have
\[ \mathbb{T}_{\tilde{\mathcal{P}}_N}^{\mathrm{eq}} (U + NU_{\mathcal{P}_R}^{\mathrm{can}} (T_R)) = \mathbb{T}_{\mathcal{P}}^{\mathrm{eq}} (U) \cap \mathbb{T}_{\mathcal{P}_R}^{\mathrm{eq}} (U_{\mathcal{P}_R}^{\mathrm{can}} (T_R)) = \{ T_R \}, \]
so that $T_N = T_R$. 
Then since $u_N = u_\infty$ is the $\mu$-canonical distribution of temperature $T_R$, we have $\tilde{u}_N = u_\infty \times u_R^{\times N}$. 
Then by $U_{\mathcal{P}} (u_\infty) = U_{\mathcal{P}}^{\mathrm{can}} (T_R) = U$, we can compute
\[ \Delta S_{\tilde{\mathcal{P}}_N} = S_{\mathcal{P}} (u_\infty) - S_{\mathcal{P}} (u) = - \frac{1}{T_R} (F_{\mathcal{P}} (T_R, u_\infty) - F_{\mathcal{P}} (T_R, u)) \]
for $u \in \M^{\exp, 1} (\mathcal{P}, U)$. 

When $T_R \notin \mathbb{T}_{\mathcal{P}}^{\mathrm{eq}} (U)$, we have either 
\[ T_R < T_N < \mathbb{T}_{\mathcal{P}}^{\mathrm{eq}} (U) \text{ or } T_R > T_N > \mathbb{T}_{\mathcal{P}}^{\mathrm{eq}} (U). \]
In either case, we have $T_R, T_N, T_R - T_N \neq 0$. 
Since $F_{\mathcal{P}_R}^{\mathrm{can}} (T)$ is differentiable on $T \in [0, \infty)$ by Proposition \ref{canonical free energy}, $U_{\mathcal{P}_R}^{\mathrm{can}} (T) = F_{\mathcal{P}_R}^{\mathrm{can}} (T) + TS_{\mathcal{P}_R}^{\mathrm{can}} (T)$ is also differentiable at $T_R$ by our assumption. 
Then by $T_N \to T_R$, we have 
\[ \lim_{N \to \infty} \frac{U_{\mathcal{P}_R}^{\mathrm{can}} (T_R) - U_{\mathcal{P}_R}^{\mathrm{can}} (T_N)}{T_R - T_N} = \partial_T U_{\mathcal{P}_R}^{\mathrm{can}} (T_R) = T_R \partial_T S_{\mathcal{P}_R}^{\mathrm{can}} (T_R) > 0. \]

On the other hand, by (\ref{internal energy equality}), we have 
\[ N (U_{\mathcal{P}_R}^{\mathrm{can}} (T_R) - U_{\mathcal{P}_R}^{\mathrm{can}} (T_N)) = U_P^{\mathrm{can}} (T_N) - U, \]
so that we compute 
\[ \lim_{N \to \infty} N ( U_{\mathcal{P}_R}^{\mathrm{can}} (T_R) - U_{\mathcal{P}_R}^{\mathrm{can}} (T_N)) = U_{\mathcal{P}}^{\mathrm{can}} (T_R) - U. \]
It follows that 
\[ \lim_{N \to \infty} N (T_R -T_N) = \lim_{N \to \infty} \frac{N (U_{\mathcal{P}}^{\mathrm{can}} (T_R) -U_{\mathcal{P}}^{\mathrm{can}} (T_N))}{(\frac{U_{\mathcal{P}}^{\mathrm{can}} (T_R) -U_{\mathcal{P}}^{\mathrm{can}} (T_N)}{T_R -T_N})} = \frac{U_{\mathcal{P}}^{\mathrm{can}} (T_R) - U}{\partial_T U_{\mathcal{P}_R}^{\mathrm{can}} (T_R)}. \]

Since $U_{\tilde{\mathcal{P}}_N}^{\mathrm{can}} (T_N) = U + NU_{\mathcal{P}_R}^{\mathrm{can}} (T_R)$, we have
\[ S_{\tilde{\mathcal{P}}_N} (\tilde{u}_N) = S_{\tilde{\mathcal{P}}_N}^{\mathrm{eq}} (U + NU_{\mathcal{P}_R}^{\mathrm{can}} (T_R)) = S^{\mathrm{can}}_{\tilde{\mathcal{P}}_N} (T_N) = S^{\mathrm{can}}_{\mathcal{P}} (T_N) + N S^{\mathrm{can}}_{\mathcal{P}_R} (T_N). \]
We then compute 
\begin{align*}
\Delta S_{\tilde{\mathcal{P}}_N} 
&= S^{\mathrm{can}}_{\mathcal{P}} (T_N) - S_{\mathcal{P}} (u) + N ( S^{\mathrm{can}}_{\mathcal{P}_R} (T_N) - S_{\mathcal{P}_R}^{\mathrm{can}} (T_R) ) 
\\
&= S^{\mathrm{can}}_{\mathcal{P}} (T_N) - S_{\mathcal{P}} (u) 
- \frac{1}{T_R} N (U_{\mathcal{P}_R}^{\mathrm{can}} (T_R) - U^{\mathrm{can}}_{\mathcal{P}_R} (T_N))
\\
&\quad + \frac{1}{T_R} N (F_{\mathcal{P}_R}^{\mathrm{can}} (T_R) - F^{\mathrm{can}}_{\mathcal{P}_R} (T_N))) + N (\frac{1}{T_N} - \frac{1}{T_R}) (U^{\mathrm{can}}_{\mathcal{P}_R} (T_N) - F^{\mathrm{can}}_{\mathcal{P}_R} (T_N))
\\
&= S^{\mathrm{can}}_{\mathcal{P}} (T_N) - S_{\mathcal{P}} (u) 
- \frac{1}{T_R} (U_{\mathcal{P}_R}^{\mathrm{can}} (T_N) - U)
\\
&\quad + \frac{N (T_R -T_N)}{T_R} \frac{F_{\mathcal{P}_R}^{\mathrm{can}} (T_R) - F^{\mathrm{can}}_{\mathcal{P}_R} (T_N)}{T_R - T_N} + \frac{N (T_R -T_N)}{T_R} S^{\mathrm{can}}_{\mathcal{P}_R} (T_N) 
\\
&\to S_{\mathcal{P}}^{\mathrm{can}} (T_R) - S_{\mathcal{P}} (u) 
- \frac{1}{T_R} (U_{\mathcal{P}}^{\mathrm{can}} (T_R) - U)
\\
&\quad+ \frac{U_{\mathcal{P}}^{\mathrm{can}} (T_R) - U}{T_R \partial_T U_{\mathcal{P}_R}^{\mathrm{can}} (T_R)} (\partial_T F^{\mathrm{can}}_{\mathcal{P}_R} (T_R) + S^{\mathrm{can}}_{\mathcal{P}_R} (T_R))
\\
&= - \frac{1}{T_R} (F_{\mathcal{P}}^{\mathrm{can}} (T_R) - F_{\mathcal{P}} (T_R, u)) = - \frac{1}{T_R} (F_{\mathcal{P}} (T_R, u_\infty) - F_{\mathcal{P}} (T_R, u)).
\end{align*}
\end{proof}

\begin{rem}
The above theorems are also valid for the case $U = U_{\mathcal{P}} (1_P)$. 
We do not even need to assume K-instability of $\mathcal{P}$, but for the proofs we need separate arguments. 
\end{rem}

\end{document}